\documentclass[11pt]{amsart}
\usepackage[utf8]{inputenc}

\usepackage{amsmath,amsthm,amssymb,amsrefs}
\usepackage{url,enumerate,epsfig,color,graphicx,epstopdf,mathrsfs}
\usepackage{hyperref}
\usepackage{enumitem}
\usepackage[all]{xy}
\usepackage{subcaption}
\usepackage{anyfontsize}

\usepackage[a4paper, top=3cm, bottom=3cm, left=2.5cm, right=2.5cm]{geometry}
\usepackage{tikz}
\usetikzlibrary{matrix, arrows,intersections,tqft,decorations.pathreplacing}
\usepackage[colorinlistoftodos]{todonotes}
\usepackage{import}

\usepackage{array}
\newcolumntype{C}[1]{>{\centering\arraybackslash$}p{#1}<{$}}

\newtheorem{theorem}[equation]{Theorem}
\newtheorem{proposition}[equation]{Proposition}
\newtheorem{lemma}[equation]{Lemma}

\newtheorem{corollary}[equation]{Corollary}

\theoremstyle{definition}
\newtheorem{definition}[equation]{Definition}
\newtheorem{example}[equation]{Example}
\newtheorem{assumption}[equation]{Assumption}

\newtheorem{construction}[equation]{Construction}

\theoremstyle{remark}
\newtheorem*{ack}{Acknowledgments}
\newtheorem{remark}[equation]{Remark}

\makeatletter
\DeclareFontEncoding{LS1}{}{}
\DeclareFontSubstitution{LS1}{stix}{m}{n}
\DeclareMathAlphabet{\mathscr}{LS1}{stixscr}{m}{n}
\makeatother
\numberwithin{equation}{section}
\def\mpar#1{}

\def\dc{d_{\bullet}}
\def\ec{e_{\bullet}}
\def\xt{\ol{x}}
\def\yt{\ol{y}}

\def\newx{\mathscr{x}}
\def\xp{\newx_+}
\def\xm{\newx_-}
\def\xpm{\newx_\pm}

\def\cD{\mathcal{D}}
\def\cF{\mathcal{F}}
\def\cFC{\cF_{\Cube}}
\def\cG{\mathcal{G}}
\def\cT{\mathcal{T}}
\def\cTKh{\cT_{\kh}}
\def\cGs{\cG_{\sigma}}
\def\cGsk{\cG_{\sigma^k}}
\def\Z{\mathbb{Z}}

\def\R{\mathbb{R}}
\def\F{\mathbb{F}}

\def\X{\mathcal{X}}

\def\cC{\mathcal{C}}
\def\cX{\mathcal{X}}
\def\cA{\mathcal{A}}
\def\cM{\mathcal{M}}
\def\cS{\mathcal{S}}
\def\cP{\mathcal{P}}
\def\tA{\mathrm{A}}
\def\tZ{\mathrm{Z}}
\def\ff{\mathfrak{f}}
\def\fs{\mathfrak{s}}
\def\fp{\mathfrak{p}}
\def\fz{\mathbf{z}}
\def\rH{\mathcal{R}^H}

\def\bA{\mathbb{C}^*}

\def\Cbl{\Cob^{3}_{\bullet/l}}
\def\Sn{\Perm_n}
\DeclareMathOperator{\CA}{CA}

\DeclareMathOperator{\coker}{coker}
\DeclareMathOperator{\Perm}{Perm}
\DeclareMathOperator{\ET}{E\theta}

\DeclareMathOperator{\ckh}{CKh}
\DeclareMathOperator{\kh}{Kh}
\DeclareMathOperator{\ekh}{EKh}

\DeclareMathOperator{\akh}{AKh}
\DeclareMathOperator{\cakh}{CAKh}
\DeclareMathOperator{\eakh}{EAKh}
\DeclareMathOperator{\ep}{EX}
\DeclareMathOperator{\ob}{Ob}

\DeclareMathOperator{\Ext}{Ext}
\DeclareMathOperator{\gr}{gr}
\DeclareMathOperator{\ind}{ind}
\DeclareMathOperator{\Hom}{Hom}
\DeclareMathOperator{\codim}{codim}
\DeclareMathOperator{\colim}{colim}

\DeclareMathOperator{\im}{Im}

\DeclareMathOperator{\Cob}{Cob}
\DeclareMathOperator{\Ann}{Ann}

\DeclareMathOperator{\Sq}{Sq}
\DeclareMathOperator{\Un}{Un}
\DeclareMathOperator{\Cube}{Cube}
\DeclareMathOperator{\Cubes}{Cube_\sigma}
\DeclareMathOperator{\Ob}{Ob}
\DeclareMathOperator{\Diff}{Diff}

\newcommand{\bnbracket}[1]{[\kern-1.5pt [ #1 ]\kern-1.5pt]_{\operatorname{BN}}}

\newcommand{\ol}[1]{\overline{#1}}
\newcommand{\wt}[1]{\widetilde{#1}}

\title{Khovanov homotopy type, periodic links and localizations}
\author{Maciej Borodzik}
\address{Institute of Mathematics, University of Warsaw, ul. Banacha 2, 02-097 Warsaw, Poland.}
\email{mcboro@mimuw.edu.pl}

\author{Wojciech Politarczyk}
\address{Department of Mathematics and Computer Science, Adam Mickiewicz University in Poznań,
  ul.~Umultowska 87, 61-614 Poznań, Poland \& Institute of Mathematics, University of Warsaw, ul. Banacha 2,
02-097 Warsaw, Poland.}
\email{politarw@amu.edu.pl, wpolitarczyk@mimuw.edu.pl}

\author{Marithania Silvero}
\address{Institute of Mathematics, Polish Academy of Sciences, ul. \'Sniadeckich 8,
00-656 Warsaw, Poland~\& Barcelona Graduate School of Mathematics at Universitat de Barcelona, Gran Vía de las Cortes Catalanas 585, 08007 Barcelona, Spain.}
\email{marithania@us.es}

\date{\today}

\makeatletter
\@namedef{subjclassname@2010}{\textup{2010} Mathematics Subject Classification}
\makeatother

\begin{document}
\begin{abstract}
  Given an $m$-periodic link $L\subset S^3$, we show that the Khovanov spectrum $\X_L$ constructed by Lipshitz and Sarkar admits a group action.
  We relate the Borel cohomology of $\X_L$ to the equivariant Khovanov homology of $L$ constructed by the second author.
  The action of Steenrod algebra on the cohomology of $\X_L$ gives an extra structure of the periodic link. Another consequence of our
  construction is an alternative proof of the localization formula for Khovanov homology, obtained first by Stoffregen and Zhang.
  By applying the Dwyer-Wilkerson theorem we express Khovanov homology of the quotient link in terms of equivariant Khovanov homology of the original link.

\end{abstract}
\maketitle

\section{Introduction}
\subsection{Overview}
\emph{Khovanov homology}~\cite{khovanov_categorification_2000} is a link invariant that assigns to any diagram $D \subset \R^{2}$ of a link $L\subset S^3$ a bigraded cochain complex $\ckh^{\ast,\ast}(D)$, whose homology groups, $\kh^{\ast,\ast}(D)$, are a link invariant.
Inspired by the construction of Cohen, Jones, and Segal~\cite{CJS}, Lipshitz, and Sarkar~\cite{LS_stable} constructed a~\emph{spatial refinement} of Khovanov homology.

\begin{theorem}\cite[Theorem 1.1]{LS_stable}\label{thm:lipsar}
  Let $D$ be a diagram representing an oriented link $L \subset S^3$.\mpar{(1)}
  For any $q \in \mathbb{Z}$ there exists a CW-complex $\mathcal{X}_{\kh}^{q}(D)$ such that the reduced cellular cochain complex $\wt{C}^*(\mathcal{X}_{\kh}^{q}(D);\Z)$ is a copy of the Khovanov complex $\ckh^{*,q}(D;\Z)$.
  In particular $\wt{H}^i(\mathcal{X}_{\kh}^{q}(D);\Z)$ is equal to $\kh^{i,q}(D;\Z)$.
  Moreover, the stable homotopy type of $\mathcal{X}_{\kh}^{q}(D)$ is an invariant of the link $L$.
\end{theorem}

Define $\mathcal{X}_{\kh}(D) = \bigvee_{q} \mathcal{X}_{\kh}^{q}(D)$.
We will often write $\mathcal{X}_{\kh}(L)$ instead of $\mathcal{X}_{\kh}(D)$ noting that $\mathcal{X}_{\kh}(L)$ is defined up to stable homotopy.
The space $\mathcal{X}_{\kh}(L)$ is called the \emph{Khovanov homotopy type} of a link $L$.\mpar{(2)}
There are various constructions of the Khovanov homotopy type, see~\cite{HKK,LS_stable,LLS_long,Everitt-Turner}, we refer to~\cite{lipshitz_spatial_2017} for a survey.

Given a link in $S^1\times D^2$, in~\cite{AsaedaPrzytyckiSikora} Asaeda, Przytycki, and Sikora (see also~\cite{Roberts}) showed that
the wrapping number around the $S^{1}$ factor induces a filtration of the Khovanov complex
\begin{equation}\label{eq:annular-filtration}
  0 \subset \ldots \subset \CA_{k-1} \subset \CA_{k} \subset \CA_{k+1} \subset \ldots \subset \ckh^{\ast,\ast}(L).
\end{equation}
The \emph{annular Khovanov homology} of $L$, denoted $\akh^{\ast,\ast,\ast}(L)$, is the homology of the associated graded cochain complex $\cakh^{\ast,\ast,\ast}(L)$.
Lifting the annular grading to the Khovanov flow category, leads to a construction of the \emph{annular Khovanov homotopy type} $\mathcal{X}_{\akh}(L) = \bigvee_{j,k \in \Z} \mathcal{X}_{\akh}^{j,k}(L)$, i.e. a spatial refinement of the annular Khovanov homology.

A link $L$ in $S^3$ is said to be \textit{$m$-periodic} if there exists an orientation-preserving action of a cyclic group $\Z_m$ on $S^3$ such that $L$ is an invariant subset of $S^3$ and the fixed point set is an unknot disjoint from $L$.
A diagram $D$ of a link $L$ is called \emph{$m$-periodic} if $0 \not\in D$ and $D$ is invariant under rotation $\rho_{m}$ of order $m$ of the plane about the point $0 \in \R^{2}$.
The Khovanov complex of an $m$-periodic link admits an induced action of $\Z_{m}$~\cite{Politarczyk_Khovanov,Chbili}.

Removing a tubular neighborhood of the fixed point axis $F$ of the rotation of $S^{3}$ produces an annular link $L \subset S^{1} \times D^{2}$ invariant under a fixed point-free rotation of $S^{1} \times D^{2}$.
Such links in $S^{1} \times D^{2}$ are also called $m$-periodic.
The $\Z_{m}$-action on the Khovanov complex preserves the filtration~\eqref{eq:annular-filtration}, hence it descends to a $\Z_{m}$-action on the annular Khovanov chain complex.

The primary purpose of this paper is to study the Khovanov homotopy type and the annular Khovanov homotopy type of periodic links.
The following two theorems constitute the central geometric part of the present article.

\begin{theorem}\label{thm:main-theorem}\ 
  \begin{itemize}
  \item[(a)]
    Let $D_{m}$ be an $m$-periodic diagram of an annular link $L$. 
    \begin{enumerate}
    \item For any $q,k \in \Z$, $\mathcal{X}_{\akh}^{q,k}(D_{m})$ admits an action of the finite cyclic group of order $m$ which is compatible with the action of $\Z_{m}$ on the annular Khovanov complex of $D_{m}$.
    \item The \emph{equivariant annular stable homotopy type} of $\mathcal{X}_{\akh}^{q,k}(D_{m})$ is an invariant of the associated annular $m$-periodic link.
    \end{enumerate}
  \item[(b)] 
    Let $D_{m}$ be a $m$-periodic link diagram.
    \begin{enumerate}
    \item For any $q \in \Z$, $\mathcal{X}_{\kh}^q(D_{m})$ admits an action of the finite cyclic group of order $m$ which is compatible with the action of $\Z_{m}$ on the Khovanov complex of $D_{m}$.
    \item The \emph{equivariant stable homotopy type} of $\mathcal{X}_{\kh}^q(D_{m})$ is an invariant of the associated $m$-periodic link.
    \end{enumerate}
  \end{itemize}
\end{theorem}

\noindent If we set $m=p$ a prime number, we have the following result. 

\begin{theorem}[Geometric Fixed Point Theorem]\label{thm:geom_local}\  
  \begin{itemize}
  \item[(a)]
    Let $D_{p}$ be a $p$-periodic diagram of an annular link $L$ and let $D$ be the associated quotient diagram.  For any $q,k \in \Z$
    \[\mathcal{X}^{q',k}_{\akh}(D_{p})^{\Z_{p}} = \mathcal{X}^{q,k}_{\akh}(D),\]
    where $q' = pq - (p-1)k$.
  \item[(b)] If $D_p$ is a $p$-periodic link diagram and $D$ is the associated quotient diagram. Then for any $q\in\Z$
    \[\mathcal{X}_{\kh}^{q}(D_{p})^{\Z_{p}} = \bigvee_{\substack{q',k' \in \Z \\ pq'-(p-1)k'=q}} \mathcal{X}_{\akh}^{q',k'}(D).\]
  \end{itemize}
\end{theorem}

From Theorem~\ref{thm:main-theorem} and Theorem~\ref{thm:geom_local}, we can obtain nontrivial relations between the (annular) Khovanov homology of a periodic link and the annular Khovanov homology of the quotient thereof.
The simplest forms of the relation are the following versions of the \emph{Smith inequality}.

\begin{theorem}\label{cor:smith-inequality-annular-kh}
  Let $p$ be a prime and let $L_{p}$ be a $p$-periodic link with associated quotient link $L$. Then, for every $q,k \in \Z$ the following holds
  \[\sum_{i} \dim_{\F_{p}} \akh^{i,pq-(p-1)k,k}(L_{p};\F_{p}) \geq \sum_{j} \dim_{\F_{p}} \akh^{j,q,k}(L;\F_{p}).\]
\end{theorem}

\begin{theorem}\label{cor:smith-inequality-kh}
  For any $p$-periodic link $L_{p} \subset S^{3}$ and any $q \in \Z$ we have
  \[\sum_{i} \dim_{\F_{p}}\kh^{i,q}(L_{p};\F_{p}) \geq \sum_{\substack{j,q',k' \\ pq'+(p-1)k'=q}} \dim_{\F_{p}} \akh^{j,q',k'}(L_{p};\F_{p}) \geq \sum_{j} \dim_{\F_{p}} \kh^{j,q}(L;\F_{p}).\]
\end{theorem}

Lipshitz and Sarkar~\cite{LS_refinement} showed that the action of stable cohomology operations on the Khovanov homology might lead to substantially stronger link invariants.
Similarly, in our case, stable cohomology operations can be used to strengthen Theorem~\ref{cor:smith-inequality-annular-kh} and Theorem~\ref{cor:smith-inequality-kh}.
As a corollary of our construction, we obtain Theorem~\ref{thm:khovanov-homology-quotient_link}, which gives a functorial way to determine the annular Khovanov homology of the quotient link from the equivariant (annular) Khovanov homology of a periodic link.
To be more precise, we show in Theorem~\ref{thm:borel} that equivariant (annular) Khovanov homology is isomorphic to the Borel equivariant cohomology of \(\mathcal{X}_{\kh}\) (respectively \(\mathcal{X}_{\akh}\)).
Careful analysis of the action of the Steenrod algebra on appropriately localized Borel cohomology, see~\cite{dwyer-wilkerson}, recovers the cohomology of the fixed point set \(\mathcal{X}_{\kh}^{\Z_{p}}\) (\(\mathcal{X}_{\akh}^{\Z_{p}}\), respectively).
Finally, by Theorem~\ref{thm:geom_local} we know that \(\mathcal{X}_{\kh}^{\Z_{p}}\) (and \(\mathcal{X}_{\akh}^{\Z_{p}}\)) are determined by \(\mathcal{X}_{\kh}\) (and \(\mathcal{X}_{\akh}\)) of the corresponding quotient link. This gives a passage from the Khovanov homology of a $p$-periodic link to the Khovanov homology of the quotient (with coefficients in $\Z_p$).
For details we refer to Section~\ref{sub:localization}.
Furthermore, a careful study of stable cohomology operations leads to a refinement of the periodicity criterion of~\cite{BP}, for which we refer the reader to a separate paper~\cite{Pol_Sil}.

\subsection{General context}
Since the advent of various homological invariants for three-manifolds or knots in three-manifolds, there has been a question on the behavior of these invariants under passing to the quotient by a group action.
One direction of the research was in the Floer theory.
Early results 
in knot Floer homology 
include Levine's paper~\cite{Levine_branched}, which was later used by Hendricks to obtain a rank inequality for knots in double branched covers (see~\cite{Hendricks_rank}).
More recent advances in this direction include another paper of Hendricks~\cite{Hendricks_local}, and finally, a paper by Lidman and Manolescu~\cite{LidMan}, where Smith-type inequalities are obtained for monopole Floer and Heegaard Floer homologies.

For Khovanov homology theory, the first localization results were obtained by Seidel and Smith~\cite{SeidelSmith_equiv}, where the authors used their own
definition of Khovanov homology based on the Lagrangian Floer theory~\cite{SeidelSmith_base}.
Note that the equivalence of the Seidel-Smith Khovanov homology with the original one is still conjectural in positive characteristic.
Motivated by their results, Hendricks, Lipshitz, and Sarkar~\cite{HLS} constructed equivariant Lagrangian Floer theory for more general groups.

The classical combinatorial definition of Khovanov homology enables an equivariant version,~\cite{Politarczyk_Khovanov}, generalizing earlier constructions of Chbili~\cite{Chbili}.

In order to study Khovanov homology using techniques from algebraic topology, it is convenient to realize Khovanov homology of a link as the singular homology of a topological space.
In a series of papers, Lipshitz and Sarkar, later also with Lawson, defined and studied the so-called \emph{Khovanov homotopy type}~\cite{LS_stable,LS_refinement,LS_Steenrod,LLS_long} (see also~\cite{Cohen} for a review in a language of algebraic topology) with the property that cohomology of the space is the Khovanov homology of a link.
A question remained whether Khovanov homotopy type (sometimes called `Khomotopy type') that they constructed admits a group action if the underlying link is periodic.

The affirmative answer was given in the first version of this paper, and, independently, by Stoffregen and Zhang \cite{StoffregenZhang}. The updated version of this paper
contains proofs of fixed point results, which were not present in the first version.
In particular, Theorems~\ref{thm:main-theorem} and~\ref{thm:geom_local} were proved first by Stoffregen and Zhang.
Note that the two constructions, even though they lead to the same result, are of
substantially different nature. Stoffregen and Zhang use the approach to Khovanov homotopy type via the Burnside category~\cite{LLS_long}.
Conceptually, this approach seems to require more case-by-case analysis.\mpar{(4)}
On the other hand, Burnside rings have deep connections with ordinary homology theory and Mackey functors; see the book of Costenoble and Waner~\cite{CostenobleWaner}.
Therefore, the construction of equivariant Khovanov homology via Burnside rings has the potential of revealing deeper structure in the equivariant Khovanov homology. 

On the contrary, our approach is very concrete and down-to-earth. Most of the arguments reduce either to the Riemann-Hurwitz formula or
to Counting Moduli Lemma~\ref{lemma:conn-components-moduli-space}, which is a direct application of relations in $\Cbl$.
Moreover, we give a specific and conceptual reason why the fixed point category of the Khovanov flow category is the annular Khovanov flow category and not just the Khovanov flow category; see Section~\ref{sec:fixed-points}, especially Lemma~\ref{lem:close_to_zhang}.\mpar{(5)}

Even more important is that we get an explicit cell decomposition of geometric realizations. Consequently, without much effort
we obtain an identification of the chain complex $C_*(\cX_{\kh})$
with the Khovanov chain complex $\ckh$ as $R[\Z_m]$ modules (for some ring $R$); see Proposition~\ref{prop:identification-kh-hom}. It follows that
Borel homology of the geometric realization is the equivariant Khovanov homology defined by Politarczyk. Theorem~\ref{thm:borel} might seem to have complicated proof, but
this is because we have rather general assumptions on the coefficient module.
Finally, methods of algebraic topology, like the Dwyer-Wilkerson theorem, allow us to recover the annular homology of the quotient link in terms of the equivariant Khovanov homology of the original link; see Theorem~\ref{thm:khovanov-homology-quotient_link}. The latter result is not present in the Stoffregen-Zhang paper. Furthermore, to the best of our understanding, passing from the results of Stoffregen and Zhang
to Theorem~\ref{thm:khovanov-homology-quotient_link} might require a few steps.

We expect that the equivariant homotopy type of Stoffregen-Zhang is equivariantly homotopy equivalent to our construction. We do not have proof of that fact.

The special case ($p=2$) of Theorem~\ref{cor:smith-inequality-annular-kh} was proved by Zhang~\cite{Zhang}.
She also proved Corollary~\ref{cor:smith-inequality-kh} for $p=2$ and certain classes of periodic links.

\subsection{Outline of the paper}

Our construction of the equivariant Khovanov homotopy type is based on the construction of the Khovanow homotopy type via cubical flow categories~\cite{LLS_long}, which is a simplification
of the original construction~\cite{LS_stable}. We consider an equivariant version of cubical flow category, called equivariant cubical flow category (see Subsection~\ref{sub:group_action}).
A remarkable difference from the non-equivariant definition is that the grading function $\gr$ is replaced by an equivariant grading function $\gr_G$ taking values in the representation ring $RO(G)$. Consequently, the moduli spaces are expected to be of dimension $\gr_G(x)-\gr_G(y)$ (refer to Definition~\ref{def:eq_dimension} 
for the definition of ``dimension'' in this setting). This approach is motivated by the construction of ordinary (Bredon) homology theory~\cite{CostenobleWaner}, and it makes the construction
of equivariant Khovanov homotopy type significantly simpler.

After defining equivariant cubical flow categories and a suitable generalization of the notion of a neat embedding to the equivariant case, we construct the equivariant Khovanov homotopy type.
Thanks to the choice of the grading function, this part of the construction is straightforward.

To show invariance under the choice of link diagram, we need to do substantially more work. The key tool is, as in~\cite{LS_stable}, the Whitehead theorem, but in
the equivariant case, the assumptions of the Whitehead theorem are much harder to verify. In particular, before proving invariance, we have to study fixed points of the
equivariant cubical flow category; Subsection~\ref{sub:fixed} is devoted to this study. Apart from that, the invariance of the group action on the choice of the diagram is
proved analogously as in the non-equivariant case.

The fixed point theorem requires even more work.
From Subsection~\ref{sub:fixed}, we know that the fixed point category is a cubical flow category, but we need to show that this category is the (annular) Khovanov flow category of the associated quotient link.
This is the statement of Theorem~\ref{thm:fixed-points}.
The proof requires a more in-depth understanding of topological and combinatorial properties of the morphism spaces $\cM(x,y)$.
The general idea is to use Bar-Natan's formulation of the Khovanov theory in terms of dotted cobordisms.\mpar{(6)}
A moduli space $\cM(x,y)$ is nontrivial if there exists a suitable cobordism $\Sigma$ between resolution configurations.
Counting Moduli Lemma~\ref{lemma:conn-components-moduli-space} expresses the number of connected components of the moduli spaces
in terms of the genera of the components of $\Sigma$.
If we pass to a cover, we can use the Riemann-Hurwitz Theorem to study the genus of the cover of the cobordism.
Then, Bar-Natan's formalism allows us to relate the moduli spaces of the periodic link and the moduli space of its quotient link.

Next, we pass to homological statements.
Our primary tool is the BQAS (Borel-Quillen-Atiyah-Segal) Localization Theorem~\cite{Borel,Quillen} and a Smith-type inequality~\cite{Smith} which relates the rank of the homology groups of a periodic knot with the rank of the homology group of the quotient knot.
As an immediate corollary of Theorem~\ref{thm:main-theorem} we obtain Smith inequalities for (annular) Khovanov homology.

While analogs of the BQAS Localization Theorem recover only the rank of the homology of the quotient knot, by applying more refined tools from algebraic topology 
we obtain a significantly stronger result.
Indeed, using the result of Dwyer and Wilkerson~\cite{dwyer-wilkerson}, it is possible to give a complete description of the Khovanov homology of the quotient knot in terms of the \emph{equivariant Khovanov homology} of a $p$-periodic knot, for a prime $p$.
By Theorem~\ref{thm:borel}, the Borel cohomology of $\mathcal{X}_{\kh}(D)$ can be identified with the \emph{equivariant Khovanov homology} $\ekh^{\ast,\ast}(L;\F_{p})$ introduced by the second author~\cite{Politarczyk_Khovanov}.
Repeating the construction of~\cite{Politarczyk_Khovanov} one can obtain the \emph{equivariant annular Khovanov homology} $\eakh^{\ast,\ast,\ast}(L;\F_{p})$, which, by an analog of Theorem~\ref{thm:borel}, is isomorphic to the Borel cohomology of $\mathcal{X}_{\akh}(L)$.
Therefore, $\ekh^{\ast,\ast}(L;\F_{p})$ and $\eakh^{\ast,\ast,\ast}(L;\F_{p})$ admit an action of the cohomology algebra $H^{\ast}(B\Z_{p};\F_{p})$, of the classifying space of $\Z_{p}$ and the action of the mod~$p$ \emph{Steenrod algebra}~$\cA_{p}$.
These two algebraic structures are sufficient to recover the annular Khovanov homology of the quotient knot from equivariant annular Khovanov homology of the periodic knot, as shown in Theorem~\ref{thm:khovanov-homology-quotient_link}.

\smallskip
The structure of the paper is as follows. Section~\ref{sec:flow_cat} recalls the construction of Lipshitz and Sarkar. The reader familiar with the construction
can skim through this section, maybe except Subsection~\ref{sub:posets_as}, where the degree of the cover map $\ff$ is expressed in terms of maximal chains in suitably defined posets.
Section~\ref{sec:equivariant} generalizes the construction of a geometric realization of a cubical flow category to the construction of a geometric realization of an equivariant
cubical flow category.  The results in this section are stated for general equivariant flow categories and general finite groups.

Section~\ref{sec:review} deals with Khovanov homotopy type. We construct the equivariant Khovanov flow category as well as its annular analog.
We show that passing to geometric realization yields a space that is independent of various choices up to equivariant stable homotopy equivalence.
This independence is proved in Section~\ref{sec:main1}.
In Section~\ref{sec:cobcat}, we make preparatory steps to prove the fixed-point theorems.
We recall Bar-Natan's construction of Khovanov homology via $\Cbl$-category and use this construction to establish Counting Moduli Lemma~\ref{lemma:conn-components-moduli-space}, which computes the number of connected components of the moduli space in terms of the genus of the cobordism in Bar-Natan's setting.
Section~\ref{sec:fixed-points} proves Categorical Fixed Point Theorem (Theorem~\ref{thm:fixed-points}).

In Section~\ref{sec:equivhomo}, we change the setting and deal with homologies of geometric realizations.
We show that Borel homology of the equivariant geometric realization of the Khovanov category coincides with Politarczyk's equivariant Khovanov homology of a periodic link (Theorem~\ref{thm:borel}).
The Dwyer-Wilkerson theory allows us to calculate the Khovanov homology of a quotient link in terms of the equivariant Khovanov homology of the associated periodic link, see Theorem~\ref{thm:khovanov-homology-quotient_link}.

Some technical results are moved to the Appendix.
In Appendix~\ref{sec:n_manifolds}, we review the definitions of manifolds with corners, while in Appendix~\ref{sec:permutohedra}, we review the definition and basic properties of permutohedra.
We also establish a technical result, Proposition~\ref{prop:permutinter}, which
essentially says that the intersection of a permutohedron with a hyperplane is a permutohedron of lower dimension. To the best of our knowledge, it is a result not known in the literature.
A consequence of this technical fact is Proposition~\ref{prop:fixed_permut}. It states that if a group acts on $\R^n$ by permuting coordinates, a fixed point set of a permutohedron is again a permutohedron. 

Finally, we note that we present detailed examples of computations in a forthcoming paper~\cite{BJP}.
\begin{ack}
  The authors wish to thank Stefan Jackowski, Kristen Hendricks, Robert Lipshitz, Sucharit Sarkar, and Krzysztof Ziemia\'nski for stimulating discussions.
  They express their gratitude to the referees for their comments on the paper, especially for suggestions of streamlining the proof of the Categorical Fixed Point Theorem.

  The first author is supported by the National Science Center grant 2016/22/E/ST1/00040.
  The second author is supported by the National Science Center grant 2016/20/S/ST1/00369.
  The third author is partially supported by the Spanish research project MTM2016-76453-C2-1-P, ERC grant PCG-336983, and by Basque Government grant IT974-16. 

  Part of the project was done during the stay of the third author in Warsaw. We are grateful to the Institute of Mathematics, Polish Academy of Science, for hospitality.
\end{ack}

\section{Flow categories and their geometric realizations}\label{sec:flow_cat}
\subsection{Flow categories}\label{sub:flow_and_covers}
In this section we use the notion of an $\langle n\rangle$-manifold introduced in the Appendix~\ref{sec:n_manifolds}.
The necessary background on permutohedra is given in Appendix~\ref{sec:permutohedra}.
\begin{definition}\label{def:flow}
  A \emph{flow category} is a topological category $\cC$ such that the set of objects is finite, discrete, and is equipped with a \emph{grading function} $\gr_{\cC} \colon \ob(\cC) \to \Z$.\mpar{(7)}
  Morphism spaces satisfy the following three conditions:
  \begin{enumerate}[label=(FC-\arabic*)]
  \item\label{item:FC-1} For any $x \in \ob(\cC)$, $\hom_{\cC}(x,x) = \{id\}$.
  \item\label{item:FC-2} For any $x, y \in \ob(\cC)$ with $\gr(x)-\gr(y) = d$, $\hom_{\cC}(x,y)$ is a (possibly empty) $(d-1)$-dimensional $\langle d-1 \rangle$-manifold.
  \item\label{item:FC-3} If $\gr(x)-\gr(y) = d$, then the composition maps induce diffeomorphisms of $\langle d-2 \rangle$-manifolds\mpar{(8): the product structure is defined in Appendix~A.}
    \[\bigsqcup_{\stackrel{z \in \ob(\cC) \setminus \{x,y\}}{\gr(z) - \gr(y)=i}} \hom_{\cC}(z,y) \times \hom_{\cC}(x,z) \cong \partial_i \hom_{\cC}(x,y).\]
    Moreover, for any $x,y \in \ob(\cC)$ we define the \emph{moduli space from $x$ to $y$} as
    \[
      \cM_{\cC}(x,y) =
      \begin{cases}
        \hom_{\cC}(x,y), & \text{ if } x \neq y, \\
        \emptyset,             & \text{ otherwise}.
      \end{cases}
    \]
  \end{enumerate}
\end{definition}

If $\tau \in \Z$, we define the \emph{$\tau$-th suspension of $\cC$}, $\Sigma^\tau(\cC)$, to be the flow category with the same objects and morphisms and associated grading function
\[\gr_{\Sigma^\tau(\cC)}(x) = \gr_{\cC}(x)+\tau.\]

\begin{definition}[see~{\cite[Section 3.1]{LLS_long}}]\label{def:cube-flow-category}
  The \emph{cube flow category} $\Cube(n)$, for $n \in \Z_{+}$, is the flow category such that:\mpar{(9)}
  \begin{enumerate}
  \item $\ob(\Cube(n)) = \{0,1\}^n$ with grading defined by
    \[\gr(u) = |u|=\sum_i u_i,\]
    where $u = (u_1,u_2,\ldots,u_n)$.
    The set of objects of $\Cube(n)$ can be partially ordered:
    \[u \geq v, \quad \text{ if } u_{i} \geq v_{i} \text{ for all } 1 \leq i \leq n.\]
    For two objects $u > v$ of the flow category with $\gr(u) - \gr(v)=d$ we define
    \[\cM_{\Cube(n)}(u,v) = \Pi_{d-1} \subset \prod_{i \colon u_i > v_i} \R,\]
    where $\Pi_{d-1}$ is a $(d-1)$-dimensional permutohedron as in Definition~\ref{def:permuto}.
  \item Composition of morphisms
    \begin{equation}\label{eq:composition_of_morphisms}
      \cM_{\Cube(n)}(w,v) \times \cM_{\Cube(n)}(u,w) \to \cM_{\Cube(n)}(u,v)
    \end{equation}
    is defined with the aid of identification from Lemma~\ref{lem:codimk}.
    Namely, for a triple of objects $u > w > v$ such that $\gr(u) - \gr(w) = k$, $\gr(w) - \gr(v) = l$, there exists $a_{1}, a_{2},\ldots, a_{k+l} \in \{1,2\ldots,n\}$ with $a_{1} < a_{2} < \ldots < a_{k+l}$, such that
    \[u_{a_1}=u_{a_2}=\cdots = u_{a_{k+l}}=1, \quad v_{a_1}=v_{a_2}=\cdots = v_{a_{k+l}}=0, \]
    and $u_{j} = v_{j}$, for $j \neq a_{1}, a_{2}, \ldots, a_{k+l}$.
    Let $\cP$ be the subset of $\{1,2,\ldots,k+l\}$ consisting of indices $s$ such that $w_{a_s}=1$.
    By Lemma~\ref{lem:codimk} the facet (see Subsection~\ref{sub:permutstart} for terminology) $\Pi_{\cP}$ of $\Pi_{k+l-1}$ can be identified with
    \[\Pi_{l-1} \times \Pi_{k-1} = \cM_{\Cube(n)}(w,v) \times \cM_{\Cube(n)}(u,w).\]
    The composition~\eqref{eq:composition_of_morphisms} is given by the embedding map
    \begin{equation}\label{eq:embedding_cube}
      \cM_{\Cube(n)}(w,v) \times \cM_{\Cube(n)}(u,w) = \Pi_{l-1} \times \Pi_{k-1} = \Pi_{\cP} \hookrightarrow \partial \cM_{\Cube(n)}(u,v).
    \end{equation}
  \end{enumerate}
  We use the notation $0_n=(0,\dots,0)\in\ob(\Cube(n))$ and $1_n=(1,\dots,1)\in\ob(\Cube(n))$.
\end{definition}

\begin{example}\label{ex:flow_on_Rn}
  In~\cite[Definition 3.14]{LS_stable} there is described
  a method to assign a flow category $\cC_{f}$ to every Morse-Smale function $f \colon M \to \R$, where $M$ is a smooth compact manifold.
  Objects of $\cC_{f}$ are critical points of $f$, the grading of an object is the index of the associated critical point, and the morphism spaces are moduli spaces of non-parametrized gradient flow lines of $f$.\mpar{(10): we removed the last two sentences.}
\end{example}
The $n$-dimensional cube $[0,1]^{n}$ can be equipped with the structure of a CW-complex with cells
\[X_{u,v} = \{w = (w_{1},\ldots,w_{n}) \in [0,1]^{n} \colon \forall_{1 \leq i \leq n} \quad v_{i} \leq w_{i} \leq u_{i}\},\]
where $u = (u_{1},\ldots,u_{n}) \in \{0,1\}^{n}$, $v = (v_{1},\ldots,v_{n}) \in \{0,1\}^{n}$ and $v_{i} \leq u_{i}$ for every $1 \leq i \leq n$.
Let $C^{\ast}([0,1]^{n};\F_{2})$ denote the cellular cochain complex of the cube associated to the CW-structure described above.
\begin{definition}
  A \emph{sign assignment} $\nu$ is a cochain $\nu \in C^{1}([0,1]^{n};\F_{2})$ such that $\partial^{\ast}\nu = 1_{2}$, the cochain in $C^{2}([0,1]^{n};\F_{2})$ with constant value~$1$.
\end{definition}
Since $H^{1}([0,1]^{n};\F_{2})$ is trivial, it is always possible to find a sign assignment.
Moreover, for any two sign assignments $\nu_{1}$ and $\nu_{2}$ we have $\nu_{1}-\nu_{2}=\partial^{\ast} t$ for some $t \in C^{0}([0,1]^{n},\F_{2})$.
The \emph{standard sign assignment} is given by the following formula
\begin{equation}\label{eq:sign-assignment}
  \nu_{st}((\epsilon_{1},\ldots,\epsilon_{j-1},1,\epsilon_{j+1},\ldots,\epsilon_{n}),(\epsilon_{1},\ldots,\epsilon_{j-1},0,\epsilon_{j+1},\ldots,\epsilon_{n})) = \epsilon_{1} + \cdots + \epsilon_{j-1},
\end{equation}
where we use the shortened notation $\nu(X_{u,v}) = \nu(u,v)$.

\begin{definition}[see \expandafter{\cite[Section 3.5]{LLS_long}}]\label{def:cubical-flow-category}
  A \emph{cubical flow category} is a flow category $\cC$ equipped with a grading-preserving functor $\ff \colon \Sigma^\tau\cC \to \Cube(n)$, for some $\tau \in \Z$ and $n \in \mathbb{N}$, such that for any pair of objects $x,y$ of $\cC$ the map $\ff_{x,y} \colon \cM_{\cC}(x,y) \to \cM_{\Cube(n)}(\ff(x),\ff(y))$ is a covering map.\mpar{(11)}
\end{definition}

To conclude this subsection we recall a definition of~\cite[Section 3.4.2]{LS_stable}.

\begin{definition}\label{def:downward}
  Let $\cC$ be a flow category and let $\mathcal{C}'$ be a subcategory of $\mathcal{C}$.
  We say that $\cC'$ is \emph{downward closed} (respectively \emph{upward closed}) if, for any $x,y \in \ob(\cC)$ such that $\cM_{\cC}(x,y)\neq \emptyset$, $x \in \Ob(\cC')$ implies that $y \in \ob(\cC')$ (respectively, $y\in \Ob(\cC')$ implies that $x\in \ob(\cC')$).

  Given a downward closed subcategory $\cC'$ of $\cC$, we consider a full subcategory $\cC''$ of $\cC$ whose objects are objects not in $\cC'$.
  The category $\mathcal{C}'$ is upward closed. We call it the \emph{complementary upward closed category} of $\cC'$.
  A complementary downward closed category of an upward closed category is constructed similarly.
\end{definition}

\subsection{Posets associated to cubical flow categories}\label{sub:posets_as}
The goal of this subsection is to calculate combinatorially the degree of the map $\ff_{x,y}\colon\cM_{\cC}(x,y)\to\cM_{\Cube}(\ff(x),\ff(y))$.
Proposition~\ref{prop:chains_moduli} is a step in establishing
Counting Moduli Lemma~\ref{lemma:conn-components-moduli-space} below,
which is needed to prove the Categorical Fixed Point Theorem (Theorem~\ref{thm:fixed-points}).

Let $P$ be a finite poset.
A \emph{chain} in $P$ is a linearly ordered subset of $P$.
A chain is called \emph{maximal} if it is maximal with respect to the inclusion relation.
We denote by $\max(P)$ the set of maximal chains of $P$.
\begin{example}
  Let $u,v\in\ob(\Cube(n))$.
  Define the poset
  \[P(u,v)=\{w \in \ob(\Cube(n)) \colon u \ge w \ge v\}.\]
  If $c = \{w_{1} > w_{2} > \ldots > w_{k}\}$ is a chain in $P(u,v)$, then we say that $c$ is a \emph{full chain} if $w_{1} = u$ and $w_{k}=v$.
  Every maximal chain is necessarily full.
\end{example}
We write $P(\Cube(n))$ for the poset of all objects of $\Cube(n)$.
While $P(\Cube(n))=\ob(\Cube(n))$, we use the notation $P(\Cube(n))$ whenever we want to emphasize the partial order on the objects of the cube category.

Choose $u>v$ in $\Cube(n)$. Set $s=\gr(u)-\gr(v)$.
An element $w\in P(u,v)$ determines a facet of $\cM_{\Cube(n)}(u,v)=\Pi_{s-1}$, which is the image of
\[\cM_{\Cube(n)}(w,v)\times\cM_{\Cube(n)}(u,w)\to \cM_{\Cube(n)}(u,v).\]
More generally, every full chain in $P(u,v)$ determines a face of $\Pi_{s-1}$.
Namely, to a full chain $u > w_1 > \ldots > w_k > v$, we associate the face which is the image of the composition map
\begin{equation}\label{eq:face-of-the-permutahedron}
  \cM_{\Cube(n)}(w_k,v)\times\cM_{\Cube(n)}(w_{k-1},w_k)\times\dots\times\cM_{\Cube(n)}(u,w_1) \xrightarrow{\circ} \cM_{\Cube(n)}(u,v).
\end{equation}
A maximal chain in $P(u,v)$ corresponds to a vertex of $\cM_{\Cube(n)}(u,v)$.
Conversely, to a vertex $\fz=(z_1,\dots,z_s)$ of $\Pi_{n-1}$ we associate a maximal chain $u=w_1>w_2>\dots>w_s=v$ such that $w_{i+1}$ differs from $w_i$ at the $z_i$-th coordinate.
Denote this maximal chain by $P_{\fz}(u,v)$.

\smallskip
Suppose now $\cC$ is a cubical flow category and $\ff\colon\cC\to\Cube(n)$ is the cubical functor.
Until the end of this subsection, we will make the following assumption. 
\begin{assumption}\label{ass:ass}
  For any $x,y\in\ob(\cC)$ such that $\gr(x)-\gr(y)=1$, the moduli space $\cM_{\cC}(x,y)$ is either empty or it is a single point.
\end{assumption}
Note that this assumption is trivially satisfied in the case of the Khovanov, respectively the annular Khovanov
flow category, defined in Subsections~\ref{sub:khomotopy_type} and~\ref{sub:annul-khov-homot-1}.

Under Assumption~\ref{ass:ass} we can define the following relation on objects: we say that $x\succ y$ if $\gr(x)-\gr(y)=1$ and $\cM_{\cC}(x,y)$ is non-empty.
In general $\succ$ is the transitive closure of this relation.
\begin{lemma}\label{lem:non_empty}
  Given $x,y \in \ob(\cC)$, $x \succ y$ if and only if $\cM_{\cC}(x,y)$ is non-empty.
\end{lemma}
\begin{proof}
  If $x\succ y$, there exists a chain $x=x_0\succ x_1\succ\ldots\succ x_s=y$ and therefore $\cM_{\cC}(x,y)$ contains $\cM_{\cC}(x_{s-1},x_s)\times\cM_{\cC}(x_{s-2},x_{s-1})\times\ldots\times \cM_{\cC}(x_0,x_1)$, so it is non-empty.
  Conversely, if $\cM_{\cC}(x,y)$ is non-empty, then it is a union of permutohedra $\Pi_{s-1}$.
  Choose a vertex of one of these permutohedra, which corresponds to $\cM_{\cC}(x_{s-1},x_{s})\times\cM_{\cC}(x_{s-2},x_{s-1})\times\ldots\times\cM_{\cC}(x_0,x_1)$ for some sequence $x=x_0,x_1,\ldots,x_s=y$ of objects in $\cC$.
  Then, $\gr(x_{i-1})-\gr(x_{i})=1$ and $\cM_{\cC}(x_{i-1},x_{i})$ is non-empty.
  Hence $x_{i-1}\succ x_i$ and therefore $x\succ y$.
\end{proof}
\begin{remark}
  The map $\ff\colon\ob(\cC)\to\ob(\Cube(n))$ is order-preserving.
\end{remark}

Let $x,y\in\ob(\cC)$ and $s=\gr(x)-\gr(y)$.
Assume that $\cM_{\cC}(x,y)$ is non-empty.
The poset $P_{\cC}(x,y)$ is the poset of all $x'\in\ob(\cC)$ such that $x \succcurlyeq x'\succcurlyeq y$.
Any full chain $x=x_0\succ x_1\succ x_2\succ\ldots\succ x_s=y$ in $P_{\cC}(x,y)$ corresponds to a face of $\cM_{\cC}(x,y)$ defined via the composition map $\cM_{\cC}(x_{s-1},x_{s})\times\ldots \times\cM_{\cC}(x_0,x_1)\subset \cM_{\cC}(x,y)$.\mpar{(12): A \emph{facet} is a codimension one face. We recall this distinction (present in LLS papers) in the appendix and we've checked all the instances of face/facet used in the paper.}
A maximal chain in $P_{\cC}(x,y)$ corresponds to a single vertex in $\cM_{\cC}(x,y)$, because if the chain is maximal, all the moduli spaces $\cM_{\cC}(x_i,x_{i+1})$ consist of a single element by Assumption~\ref{ass:ass}.
\begin{definition}\label{def:assoc_vertex}
  For a maximal chain $m\in P_{\cC}(x,y)$, the \emph{associated vertex} $v_m\in\cM_{\cC}(x,y)$ is the vertex associated to $m$ by the above construction.
\end{definition}
The correspondence can be reversed. Each face of $\cM_{\cC}(x,y)$ determines a chain in $P_{\cC}(x,y)$ precisely as in the case of the cube flow category.
The following result is a special case.
\begin{lemma}\label{lem:assoc_reverse}
  For every vertex $v\in\cM_{\cC}(x,y)$ there exists a maximal chain $m\in P_{\cC}(x,y)$ such that $v=v_m$.
\end{lemma}
\begin{proof}
  By definition, a vertex $v$ in $\cM_{\cC}(x,y)$ is an image of $\cM_{\cC}(x_r,y)\times\ldots\times \cM_{\cC}(x,x_1)$, where all moduli spaces are zero-dimensional.
  In particular, with $x_0=x$ and $x_{r+1}=y$, we have $\gr(x_i)-\gr(x_{i-1})=1$, which implies that the chain $x_0\succ\ldots\succ x_{r+1}$ is maximal.
  Clearly, the vertex associated to this chain is $v$.
\end{proof}
Given $u=\ff(x)$, $v=\ff(y)$ two objects in $\Cube(n)$, the map $\ff_{x,y}\colon \cM_{\cC}(x,y)\to \cM_{\Cube(n)}(u,v)$ induces a map of posets
\[\ff^P_{x,y}\colon P_{\cC}(x,y)\to P_{\Cube(n)}(u,v).\]
For any vertex $z \in \Pi_{n-1}$ we define $P_{\fz}(x,y) \subset P_{\mathcal{C}}(x,y)$ to be the preimage of $P_{\fz}(u,v)$ under $\ff^P_{x,y}$.

\begin{proposition}\label{prop:chains_moduli}
  For any vertex $\fz$ of $\Pi_{n-1}$,
  \[\# \max P_{\fz}(x,y) = \# \pi_{0}(\cM_{\cC}(x,y)).\]
\end{proposition}
\begin{proof}
  Fix a vertex $\fz\in\Pi_{n-1}$.
  As the map $\ff$ is a cover, we infer that 
  $\#\pi_{0}(\cM_{\cC}(x,y))=\#\ff^{-1}_{x,y}(\fz)$.
  To show that $\#\max P_{\fz}(x,y)=\#\ff^{-1}_{x,y}(\fz)$, let first $v \in \ff^{-1}_{x,y}(\fz)$, and denote by $m$ the maximal chain $m$ in $P(x,y)$ associated to $v$.
  Clearly $m \in \max P_{\fz}(x,y)$.
  On the other hand, every maximal chain $m \in \max P_{\fz}(x,y)$ has an associated vertex $v_m\in\cM_{\cC}(x,y)$ such that $\ff(v_m)=\fz$.
  This shows that there is a bijection between $\max P_{\fz}(x,y)$ and $\ff^{-1}_{x,y}(\fz)$.
\end{proof}
\subsection{Neat embeddings}\label{sub:neat-embeddings}
Recall that Lawson, Lipshitz and Sarkar described in~\cite[Section 3]{LLS_long} a construction that turns a cubical flow category into a CW-complex.
The construction is a simplification of the construction of Lipshitz and Sarkar in~\cite{LS_stable}
In Subsections~\ref{sub:neat-embeddings},~\ref{sub:framed} and~\ref{sub:geometric} we give a brief review.

Let ($\cC$, $\ff)$ be a cubical flow category, and fix $\dc = (d_{0},d_{1},\ldots,d_{n-1}) \in \mathbb{N}^{n}$ and $R>0$.
For any $u>v$ in $\Ob(\Cube(n))$ define
\[E_{u,v} = \left[\prod_{i=|v|}^{|u|-1} [-R,R]^{d_{i}}\right] \times \cM_{\Cube(n)}(u,v).\]
For any triple of objects $u>v>w$ there is a map $\Upsilon \colon E_{v,w} \times E_{u,v} \to E_{u,w}$ defined as the composition:
\begin{align}\label{eq:eee}
  E_{v,w} \times E_{u,v} &\cong \prod_{i=|w|}^{|v|-1} [-R,R]^{d_{i}} \times \cM_{\Cube(n)}(v,w) \times \left[\prod_{j=|v|}^{|u|-1} [-R,R]^{d_{j}}\right] \times \cM_{\Cube(n)}(u,v) \nonumber\\
                         &\cong \prod_{i=|w|}^{|u|-1} [-R,R]^{d_{i}} \times \cM_{\Cube(n)}(v,w) \times \cM_{\Cube(n)}(u,v) \\
                         &\hookrightarrow \prod_{i=|w|}^{|u|-1} [-R,R]^{d_{i}} \times \cM_{\Cube(n)}(u,w)=E_{u,w}.\nonumber
\end{align}

A \emph{cubical neat embedding} $\iota$ of a cubical flow category $(\cC,\ff)$ relative to $\dc = (d_{0},d_{1},\ldots,d_{n-1}) \in \mathbb{N}^{n}$ is a collection of neat embeddings
\[\iota_{x,y} \colon \cM_{\cC}(x,y) \hookrightarrow E_{\ff(x),\ff(y)}\]
such that
\begin{enumerate}[label=(CNE-\arabic*)]
\item\label{item:CNE-1}For each $x,y \in \Ob(\cC)$ the following diagram commutes
  \begin{center}
    \begin{tikzpicture}
      \matrix(m)[matrix of math nodes, column sep=1.5 cm, row sep=1cm]{
        \cM_\cC(x,y) & E_{\ff(x),\ff(y)} \\
        & \cM_{\Cube(n)}(\ff(x),\ff(y)).\\
      };
      \path[->]
      (m-1-1) edge node[above] {$\iota_{x,y}$} (m-1-2)
      (m-1-2) edge node[right] {projection} (m-2-2)
      (m-1-1) edge node [above] {$\ff_{x,y}$} (m-2-2);
    \end{tikzpicture}
  \end{center}
\item\label{item:CNE-2} For any $u,v \in \Ob(\Cube(n))$ the map
  \[\bigsqcup_{\ff(x)=u, \ff(y)=v} \iota_{x,y} \colon \bigsqcup_{\ff(x)=u, \ff(y)=v} \cM_{\cC}(x,y) \hookrightarrow E_{u,v}\]
  is a neat embedding (see Definition~\ref{def:neat_embedding}).
\item\label{item:CNE-3} For any triple $x>y>z \in \Ob(\cC)$ the following diagram commutes
  \begin{center}
    \begin{tikzpicture}
      \matrix(m)[matrix of math nodes, row sep=1cm, column sep=1.5cm] {
        \cM_{\cC}(y,z) \times \cM_\cC(x,y) & \cM_{\cC}(x,z) \\
        E_{\ff(y),\ff(z)} \times E_{\ff(x),\ff(y)} & E_{\ff(x),\ff(z)}. \\
      };
      \path[->]
      (m-1-1) edge[above] node {$\circ$} (m-1-2)
      (m-1-1) edge[left] node {$\iota_{y,z}\times\iota_{x,y}$} (m-2-1)
      (m-1-2) edge[right] node {$\iota_{x,z}$} (m-2-2)
      (m-2-1) edge[above] node {$\Upsilon$} (m-2-2);
    \end{tikzpicture}
  \end{center}
  Here the vertical maps are given by $\iota$, the top horizontal map is the composition of morphisms and the bottom horizontal map is as defined in~\eqref{eq:eee}.
\end{enumerate}
\subsection{Framed cubical neat embeddings}\label{sub:framed}
To perform the construction  of Lawson, Lipshitz and Sarkar, we need to construct an extension of $\iota$ to a \emph{framed cubical neat embedding} $\bar{\iota}$, i.e\ a collection of embeddings
\[\bar{\iota}_{x,y} \colon \prod_{i=|\ff(y)|}^{|\ff(x)|-1} [-\epsilon,\epsilon]^{d_{i}} \times \cM_{\cC}(x,y) \to E_{\ff(x),\ff(y)},\]
for some $\epsilon > 0$, in such a way that the commutativity from~\ref{item:CNE-3} is preserved with $\bar{\iota}$ replacing~$\iota$.
In general $\bar{\iota}$ can be constructed as follows:
\begin{align}
  \bar{\iota}_{x,y} \colon \prod_{i=|\ff(y)|}^{|\ff(x)|-1} [-\epsilon,\epsilon]^{d_{i}} \times \cM_{\cC}(x,y) &\to E_{\ff(x),\ff(y)}  = \prod_{i=|\ff(y)|}^{|\ff(x)|-1} [-R,R]^{d_{i}} \times \cM_{\Cube(n)}(\ff(x),\ff(y)) \nonumber\\
  (t, \gamma) &\mapsto (t + \pi_{u,v}^{R} \iota_{x,y}(\gamma), \, \pi_{u,v}^{M} \iota_{x,y}(\gamma)), \label{eq:framed-extension-iota}
\end{align}
where
\begin{align}
  \pi_{u,v}^{R} \colon \prod_{i=|v|}^{|u|-1} [-R,R]^{d_{i}} \times \cM_{\Cube(n)}(u,v) &\to \prod_{i=|v|}^{|u|-1} [-R,R]^{d_{i}}, \label{eq:projection-R}\\
  \pi_{u,v}^{M} \colon \prod_{i=|v|}^{|u|-1} [-R,R]^{d_{i}} \times \cM_{\Cube(n)}(u,v) &\to \cM_{\Cube(n)}(u,v) \label{eq:projection-M}
\end{align}
are projections, with $\ff(x) = u$, $\ff(y)=v$.

If $\bar{\iota}$ is a framed neat embedding of the cube flow category, then $\bar{\iota}$ determines a sign assignment.\mpar{(13)}
Namely, for $u,v \in \{0,1\}^{n}$ such that $\gr(u)-\gr(v)=1$, we set $\nu(u,v)=0$ if $\iota_{u,v}(\cM_{\Cube(n)}(u,v))$ is framed positively with respect to the standard framing of $[-R,R]^{d_{|v|}}$, and $\nu(u,v)=1$ otherwise.
In this case, we say that $\bar{\iota}$ \emph{refines} $\nu$.

\begin{lemma}\label{lemma:sign-assignment-framed-embedding}
  Any sign assignment $\nu$ determines a framed cubical neat embedding of the cube flow category which refines $\nu$.
\end{lemma}
\begin{proof}
  The lemma follows directly from~\cite[Proposition 4.12]{LS_stable}.
\end{proof}
Any sign assignment for the cube flow category induces a sign assignment for its cover in an obvious way. In particular, a framed neat embedding of a cubical flow category
induces a framed neat embedding of the underlying cube flow category, hence a sign assignment on the cube flow category; see \cite[Section 3.5]{LLS_long} for
more details.

\subsection{Cubical realizations}\label{sub:geometric}

Let us fix a cubical flow category $(\cC,\ff)$, a cubical neat embedding $\iota$ of $\cC$ relative to a tuple $\dc = (d_{0},d_{1},\ldots,d_{n-1})$ and fix $\epsilon>0$ in such a way that the map~\eqref{eq:framed-extension-iota} is an embedding.\mpar{(14)}
As in \cite[Definition 3.29]{LLS_long} we construct a based CW-complex $\left(||\cC||, x_{0}\right)$ in the following way:
\begin{enumerate}
\item For any $x \in \Ob(\cC)$, if $u=\ff(x)$, we define the cell associated to $x$ as
  \begin{equation}\label{eq:Pdef}
    X(x) = \prod_{i=0}^{|u|-1} [-R,R]^{d_{i}} \times \prod_{i=|u|}^{n-1} [-\epsilon,\epsilon]^{d_{i}} \times \widetilde{\cM}_{\Cube(n)}(u,0_n),
  \end{equation}
  where $\widetilde{\cM}_{\Cube(n)}(u,0_n)$ is defined to be $\{0\}$ if $u=0_n$ and $[0,1] \times \cM_{\Cube(n)}(u,0_n)$ otherwise.
\item The cells $X(x)$ are glued together inductively.
  First we start with a disjoint union of cells $X(y)$ for $\{y\colon\ff(y)=0_{n}\in\{0,1\}^n\}$.
  For arbitrary $x\in\Ob(\cC)$, the cell $X(x)$ is glued to the union $\bigcup_{y\colon \ff(x)>\ff(y)} X(y)$.
  The gluing map is described below.
\item For any $x,y \in \Ob(\cC)$ with $\ff(x) = u > v = \ff(y)$ the cubical embedding provides an embedding $\theta_{y,x}\colon X(y)\times\cM_{\cC}(x,y)\to X(x)$
  given by
  \begin{align}\label{eq:thetadef}
    &X(y) \times \cM_{\cC}(x,y) = \nonumber\\
    &= \prod_{i=0}^{|v|-1} [-R,R]^{d_{i}} \times \prod_{i=|v|}^{n-1} [-\epsilon, \epsilon]^{d_{i}} \times \widetilde{\cM}_{\Cube(n)}(v,0) \times \cM_{\cC}(x,y) \nonumber\\
    &\cong \prod_{i=0}^{|v|-1} [-R,R]^{d_{i}} \times \prod_{i=|u|}^{n-1} [-\epsilon,\epsilon]^{d_{i}} \times \widetilde{\cM}_{\Cube(n)}(v,0) \times \left(\prod_{i=|v|}^{|u|-1} [-\epsilon,\epsilon]^{d_{i}} \times \cM_{\cC}(x,y) \right) \nonumber\\
    &\hookrightarrow \prod_{i=0}^{|v|-1} [-R,R]^{d_{i}} \times \prod_{i=|u|}^{n-1} [-\epsilon,\epsilon]^{d_{i}} \times \widetilde{\cM}_{\Cube(n)}(v,0) \times \left(\prod_{i=|v|}^{|u|-1} [-R,R]^{d_{i}} \times \cM_{\Cube(n)}(u,v)\right) \\
    &\cong \prod_{i=0}^{|u|-1} [-R,R]^{d_{i}} \times \prod_{i=|u|}^{n-1} [-\epsilon,\epsilon]^{d_{i}} \times \widetilde{\cM}_{\Cube(n)}(v,0) \times \cM_{\Cube(n)}(u,v) \nonumber\\
    &\hookrightarrow \prod_{i=0}^{|u|-1} [-R,R]^{d_{i}} \times \prod_{i=|u|}^{n-1} [-\epsilon,\epsilon]^{d_{i}} \times \partial(\widetilde{\cM}_{\Cube(n)}(u,0)) \subset X(x) .\nonumber
  \end{align}
  The first inclusion is given by the map $\ol{\iota}_{x,y}$.
  The last inclusion comes from the composition map if $v \neq 0$, or the inclusion $\{0\} \hookrightarrow [0,1]$ if $v=0$.
  Denote by $X_{y}(x)\subset X(x)$ the image of the above map.
\item The attaching map for $X(x)$ sends $X_{y}(x) \cong X(y) \times \cM_{\cC}(x,y)$ to $X(y)$ via the projection onto the first factor.
  The complement of $\cup_{y} X_{y}(x)$ in $\partial X(x)$ is mapped to the base point.
\end{enumerate}
\begin{remark}
  It is proved in~\cite[Lemma 3.30]{LLS_long} that the attaching maps are well-defined.\mpar{(15)}
  This boils down to showing that if $x,y,z\in\Ob(\cC)$ are such that $\ff(x)>\ff(y)>\ff(z)$, then there exists a map $\kappa_{x,y,z}$ that makes the following diagram commute.
  \begin{equation}\label{eq:commute}
    \xymatrix{%
      X_z(x)\cap X_y(x)\ar@{^{(}->}[r]\ar[rrdd]^{\kappa_{x,y,z}}\ar@{^{(}->}[d] & X_y(x)\ar[r] & X(y)\\
      X_z(x)\ar[d] &&\partial X(y)\ar@{^{(}->}[u]\\
      X(z)&&X_z(y)\ar[ll]\ar@{^{(}->}[u].}
  \end{equation}
\end{remark}
\begin{definition}\label{def:cubical_realization}
  The CW-complex $||\cC||$ is called the \emph{cubical realization} of the cubical category $\cC$. The formal desuspension:\mpar{(16): we now use the terminology: cubical realization instead of geometric realization. The former is precisely defined here. This distinction also clarifies point (36) and, partially also (43).}
  \[\mathcal{X}(\cC) = \Sigma^{-\tau-d_{0}-d_{1}-\cdots-d_{n-1}}||\cC||,\]
  where $\tau$ is as in Definition~\ref{def:cubical-flow-category},
  is called the $\cC$-homotopy type.
\end{definition}
We note that we deviate slightly from \cite{LLS_long}. We want the cubical realization to be a CW-complex, i.e. a topological space. After desuspension
we obtain an object in the Spanier--Whitehead category, for which we use different notation $\cX(\cC)$. 
Thanks to this distinction, many statements become more transparent, like the statement of Proposition~\ref{prop:fixedcategory}.

\begin{remark}\label{rem:direct_sum_wedge_sum}
  It follows directly from the construction that if $\cC$ is a union of categories $\cC_1,\dots,\cC_s$ (in the sense that objects are set sums of objects, and there are no
  morphisms between objects in different summands), then $\cX(\cC)$ is the wedge sum of $\cX(\cC_1),\dots,\cX(\cC_s)$.
\end{remark}
\subsection{Chain complex associated with a cubical flow category}\label{sub:chain_complex}
For completeness of the exposition, we recall how to compute the singular cohomology of the cubical realization.
A detailed account is given in~\cite[Section 3]{LS_stable} and~\cite[Section 3.2]{LLS_long}.

Let $(\cC,\ff)$ be a cubical flow category with $\ff\colon\Sigma^\tau\cC\to\Cube(n)$. Choose a sign assignment $\nu$ for $\Cube(n)$.
Define a cochain complex $C^*(\cC, \ff)$ in the following way:
\begin{itemize}
\item The group $C^k(\cC)$ is freely generated over $\Z$ by the objects of $\cC$ whose grading is equal to $k$;
\item If $x\in\Ob(\cC)$ has grading $k$, then we define
  \begin{equation}\label{eq:chain_complex}
    \partial\langle x\rangle =\sum_{\substack{y\in\Ob(\cC)\\ \gr(y)=k+1}} n_{x,y}\langle y\rangle,
  \end{equation}
  where $n_{x,y}$ is the signed count of points in $\cM_{\cC}(x,y)$.
  In particular, if we choose a framed cubical neat embedding which refines a sign assignment $\nu$, then
  \begin{equation}\label{eq:chain2}
    n_{x,y} = (-1)^{\nu(\ff(x),\ff(y))}\# \cM_{\cC}(x,y).
  \end{equation}
\end{itemize}

The following result follows immediately from the construction of $||\cC||$.
\begin{lemma}\label{lem:cochain_complex}
  $C^*(\cC, \ff)$ is a cochain complex, that is, $\partial^2=0$, and its associated cohomology is equal to the cohomology of $\cX(\cC)$, the $\cC$-homotopy type..
\end{lemma}

\section{Equivariant flow categories}\label{sec:equivariant}
In this section we adapt the construction from Section~\ref{sec:flow_cat} to the equivariant setting.
First, we will introduce some terminology from equivariant differential topology.
General references include~\cite{dieck_transformation_1987,May,wasserman_equivariant_1969}.

\subsection{Terminology}\label{sub:terminology}

Let $G$ be a finite group.
An \emph{orthogonal representation} of $G$ is a homomorphism $\rho \colon G \to O(V)$, where $O(V)$ denotes the group of orthogonal automorphisms of some inner product space $V$.
In particular, $V$ is implicitly equipped with an inner product which is preserved by $G$.
In the present article, we consider only finite-dimensional representations.

If it does not lead to confusion, we will refer to a representation $\rho\colon G\to O(V)$ as $V$.
In particular, for a subgroup $H\subset G$, the notation $V|_H$ means the representation $\rho|_H\colon H\to O(V)$.
For two representations, $V, W$ we denote by $\hom_G(V, W)$ the space of $G$-equivariant linear maps from $V$ to $W$.

If $W \subset V$ are two $G$-representations, then by $V-W$, we denote the orthogonal complement of $W$ in $V$. This notation is extended to the case when $W$ is not necessarily a subrepresentation of $V$ by introducing a Grothendieck group (actually a ring)
of representations. More specifically,
the representation ring $RO(G)$ is the ring whose elements are formal differences $V-W$ of orthogonal $G$-representations, where $V_{1}-W_{1} = V_{2} - W_{2}$ in $RO(G)$
if $V_{1} \oplus W_{2}$ is equivalent to $V_{2} \oplus W_{1}$.
Notice that if $W \subset V$, then $V-W$ is isomorphic in $RO(G)$ to the orthogonal complement of $W$ in $V$.
The direct sum induces the addition, and the tensor product over~$\R$ induces the multiplication.

We pass to the definition and basic properties
of $G$-manifolds. Some more technical results are deferred to the Appendix.
General references for group actions on manifolds include~\cite{dieck_transformation_1987,wasserman_equivariant_1969,Kanka}.

We say that $M$ is a \emph{$G$-manifold}, if it is a manifold (possibly with boundary) equipped with a smooth action of $G$.
Observe that for any \(x \in M\), the isotropy group $G_x$ acts on the tangent space $T_x M$.
By abuse of notation, we will denote by $T_x M$ the tangent representation of $G_x$.
For any subgroup $H \subset G$ define
\[M^{H} = \{x \in M \colon \forall_{h \in H} h \cdot x = x\} = \{x \in M \colon H \subset G_{x}\}.\]

We say that $M$ is of \textit{dimension $V-W \in RO(G)$}, if for any $x \in M$ there exists an isomorphism of $G_x$-representations $T_x M \oplus W|_{G_{x}} \cong V|_{G_x}$.

Let $M$ be a compact $G$-manifold and let $p \colon E \to M$ be a vector bundle over $M$.
We say that $E$ is a \emph{$G$-vector bundle} if there exists an action of $G$ on $E$ by vector bundle morphisms such that $p$ commutes with the action of $G$ on $E$ and $M$.
If $V$ is a $G$-representation, then a \emph{$V$-bundle} is a $G$-vector bundle $p \colon E \to M$ such that for any $x \in M$ there exists an isomorphism of $G_{x}$ representations between $V|_{G_{x}}$ and $p^{-1}(x)$.
We denote by $\underline{V}_{M}$ the trivial $V$-bundle over $M$, i.e.\ $\underline{V}_{M} = V \times M$.

\begin{example}
  Let $M$ be a $G$-manifold and let $V$ be a $G$-representation.
  The tangent bundle $TM$ is a $V$-bundle if and only if $M$ is of dimension $V$.
\end{example}

A \emph{framing} of a $V$-bundle is a choice of an isomorphism of $V$-bundles $\phi \colon E \to \underline{V}_{M}$.
A \emph{stable framing} is a choice of an isomorphism of $V$-bundles $\phi \colon E \oplus \underline{W}_{M} \to \underline{V}_{M} \oplus \underline{W}_{M}$
for some trivial bundle $\underline{W}_M$.
Any framing of a $V$-bundle determines an orientation of the bundle.

\subsection{Equivariant cell complexes}\label{sub:equiv_cell}
In order to fix the terminology, we recall the notion of a $G$-cell and a $G$-cell complex.\mpar{This is a new subsection. It's purpose is to deal with remark (28)}
\begin{definition}\label{def:g-cell}
  Let $H \subset G$ be a subgroup and let $V$ be an $H$-representation.
  A \emph{$G$-cell} of type $(H,V)$, denoted by $E(H,V)$, is $G \times_{H} B_{R}(V)$, where $B_{R}(V)$ denotes the closed ball in $V$ centered at \(0\) and of radius $R>0$.\mpar{(18)}
  Notice that if $V$ is a $G$-representation, then $E(H,V|_{H}) \cong G \times_{H} B_{R}(V)$.
  A \emph{$G$-cell complex} is a topological space $X$ with a filtration
  \[X_{0} \subset X_{1} \subset \ldots \subset X_{n} \subset \ldots\]
  such that
  \begin{itemize}
  \item $X_{0}$ is a disjoint union of orbits,
  \item for any $n>0$, $X_{n} = X_{n-1} \cup_{f} E(H_n,V_n)$, where
    \[f \colon \partial E(H_n,V_n) \to X_{n-1}\]
    is an equivariant map,
  \item $X = \colim_{n} X_{n}$.
  \end{itemize}
  \end{definition}
  The $G$-cell complex is called a $\operatorname{Rep}(G)$-complex in~\cite[Section 1.13]{ferland}.
  If we restrict the class of cells allowed in the construction, we obtain the following special cases.
  \begin{itemize}
  \item If we assume that all representations $V_n$ are of the form a $V \oplus \R^{a_n}$, for some fixed representation $V$ and some integers $a_n$,
    we obtain a $G$-CW($V$) complex in the sense of e.g.~\cite[Section 1.1.2]{CostenobleWaner} or~\cite[Section X.2]{May}.
  \item If, on the other hand, all $V$ are trivial representations, we obtain a $G$-CW complex as in~\cite[Section 1.1.3]{CostenobleWaner} or~\cite[Section 1.3]{May}.
  \end{itemize}
  Topological spaces we construct are usually $G$-cell complexes, while in Section~\ref{sec:equivhomo}, we apply theorems for $G$-CW complexes. Therefore
  we need to translate from one object to another. The following result is well-known to experts.

  \begin{proposition}\label{prop:repG_to_unrep}
    Any $G$-cell complex has a \(G\)-homotopy type of a $G$-CW complex.
  \end{proposition}
  \begin{proof}
    There are essentially two ways of approaching this result. In \cite[Proposition X.2.8]{May} it is proved that a $G$-CW$(V)$ complex is $G$-homotopy equivalent
    to a $G$-CW complex, and the proof can be adapted to the case of general $G$-cell complexes.

    Another way is to refine the cell structure, namely to find a triangulation of $B_R(V)$ by cells such that $G$ acts on $B_R(V)$ by permuting cells.
    This can be done using the results of Illman \cite{illman2} (if $G$ is a finite group, \cite{illman1} suffices).
  \end{proof}
  \subsection{Equivariant Spanier--Whitehead category}\label{sub:spanier_whitehead}
  For completeness of exposition we recall the definition of equivariant Spanier-Whitehead category.\mpar{This is a new subsection. It should clarify the problems raised by the
  referee in remark (43)}
Suppose $X$ and $Y$ are finite $G$-CW complexes. A (\(G\)-)equivariant homotopy of \(G\)-maps \(f,g \colon X \to Y\) is an equivariant map
\[H \colon X \times [0,1] \to Y,\]
where \(G\) acts trivially on \([0,1]\).
We denote by \([X,Y]_{G}\) the set of homotopy classes of maps \(f \colon X \to Y\).
A \(G\)-map is called a (\(G\)-)\emph{equivariant homotopy equivalence} if it admits an equivariant homotopy inverse.
A \(G\)-map \(f \colon X \to Y\) is an \emph{equivariant stable homotopy equivalence} if there exists a \(G\)-representation \(V\) such that the map
\[f \wedge id_{S^{V}} \colon X \wedge S^{V} \to Y \wedge S^{V}\]
is an equivariant homotopy equivalence.
Here, $S^V$ is the one-point compactification \(S^{V}\) of \(V\).

\begin{definition}
The \emph{equivariant Spanier-Whitehead category} \(SW_{G}\) is the category whose objects are the pairs \((X, V)\), where \(X\) is a finite \(G\)-CW complex  and \(V\) is a virtual \(G\)-representation.
Morphisms are defined by
\[\Hom_{SW_{G}}\left((X,V), (Y,W)\right) = \colim_{Z}\left[ X \wedge S^{V \oplus Z}, Y \wedge S^{W \oplus Z}  \right]_{G},\]
where \(Z\) runs through the family of finite-dimensional \(G\)-representations such that both \(V \oplus Z\) and \(W \oplus Z\) are \(G\)-representations.
\end{definition}
The equivariant Spanier-Whitehead category is a full subcategory of the equivariant stable homotopy category, see~\cite[Proposition XII.7.3]{May} and the preceding discussion.

  \subsection{Group actions on flow categories}\label{sub:group_action}
  We introduce now the definition of a group action on a flow category.
  To understand the details, it might be helpful the reader to keep in mind that the construction
  is modeled on the flow category associated with an equivariant Morse function.

  \begin{definition}
    Let $G$ be a finite group and let $\cC$ be a flow category.
    We say that $\cC$ is a \emph{$G$-equivariant flow category} (as usual, we will omit $G$ when it is clear from the context) if it is equipped with the following data:
    \begin{enumerate}
    \item for any $g \in G$ there exists a grading preserving functor
      \[\cG_g \colon \cC \to \cC,\]
    \item there is an \emph{equivariant grading function}
      \[\gr_{G} \colon \ob(\cC) \to \bigsqcup_{H \subset G} RO(H).\]
    \end{enumerate}
    Moreover, these data must satisfy the following conditions:
    \begin{enumerate}[label=(EFC-\arabic*)]
    \item\label{item:EFC-1} $\cG_{e}$ is the identity functor.
    \item\label{item:EFC-2} For any $g_1, g_2 \in G$ we have $\cG_{g_1} \circ \cG_{g_2} = \cG_{g_1 \cdot g_2}$.
    \item\label{item:EFC-3} $(\cG_g)_{x,y} \colon \cM_{\cC}(x,y) \to \cM_{\cC}(\cG_g(x), \cG_g(y))$ is a diffeomorphism of $\langle \gr(x)-\gr(y)-1 \rangle$-manifolds, which satisfies the following property
      \[(\cG_{g})_{x,y}|_{\cM_{\cC}(z,y) \times \cM_{\cC}(x,z)} = (\cG_{g})_{z,y} \times (\cG_g)_{x,z},\]
      for all $z\in\ob(\cC)$ such that $\gr(y) < \gr(z) < \gr(x)$. Here we identify $\cM_{\cC}(z,y) \times \cM_{\cC}(x,z)$ with the respective facet of $\partial \cM_{\cC}(x,y)$.
    \item\label{item:EFC-4} $\gr_G(x) \in RO(G_x)$, where $G_x = \{g \in G \colon \cG_g(x) = x\}$.
    \item\label{item:EFC-5} $\dim_{\R} \gr_G(x) = \gr(x)$.
    \item\label{item:EFC-6} If there exists $g \in G$ such that $\cG_g(x_1) = x_2$, for some $x_1,x_2 \in \ob(\cC)$, then $\gr_G(x_2) = {\upsilon_{g}}(\gr_G(x_1))$, where ${\upsilon_{g}} \colon RO(G_x) \to RO(G_{g\cdot x})$ is induced by the map\mpar{(19)}
      \[G_x \owns h \mapsto ghg^{-1} \in g G_x g^{-1} = G_{g \cdot x}.\]
      In particular, for any $g_{1},g_{2} \in G$, $\upsilon_{g_{1}} \circ \upsilon_{g_{2}} = \upsilon_{g_{1} \cdot g_{2}}$.
    \item\label{item:EFC-7} Let $x,y \in \ob(\cC)$ and define $G_{x,y} = \{g \in G \colon \cG_g(\cM_{\cC}(x,y)) \subset \cM_{\cC}(x,y)\} = G_{x} \cap G_{y}$.
      The moduli space $\cM_{\cC}(x,y)$ is a compact $G_{x,y}$-manifold of dimension
      \[\gr_{G}(x)|_{G_{x,y}} - \gr_{G}(y)|_{G_{x,y}} - \R.\]
    \end{enumerate}
  \end{definition}

  In the non-equivariant setting, it is possible to define the suspension $\Sigma^{k}\cC$ of a flow category $\cC$ by shifting the grading function by $k \in \Z$.
  In the equivariant setting, we define the \emph{suspension of a flow category $\cC$} by any virtual representation $V-W \in RO(G)$.
  The category $\Sigma^{V-W}\cC$ has the same objects and morphisms as $\cC$ but different grading function given by
  \[(\gr_G)_{\Sigma^{V-W}\cC}(x) = (\gr_G)_{\cC}(x) + (V - W)|_{G_x} \in RO(G_x).\]

  \begin{definition}\label{def:equi_functor}
    Given two $G$-equivariant flow categories $\cC_{1}$ and $\cC_{2}$, a functor $\ff \colon \cC_{1} \to \cC_{2}$ is said to be an \emph{$G$-equivariant functor} if
    \begin{itemize}
    \item $\ff$ commutes with group actions on $\cC_{1}$ and $\cC_{2}$,
    \item for any object $x$ in $\cC_{1}$ there is a $G_x$-equivariant map
      \begin{equation}\label{eq:grading_def}\ff_{\gr_{G}(x)} \colon \gr_{G}(x) \to \gr_{G}(\ff(x)),\end{equation}
      such that for any $g \in G$, we have
      \[\upsilon_{g} \circ \ff_{\gr_{G(x)}} = \ff_{\gr_{G}(\cG_{g}(x))} \circ \upsilon_{g}.\]
    \end{itemize}
  \end{definition}

  \begin{definition}
    A $G$-equivariant functor $\ff\colon \cC_1\to\cC_2$ is called a \emph{(trivial) $G$-cover} if for any $x,y\in\Ob(\cC_1)$ the map $\ff_{x,y}\colon\cM_{\cC_1}(x,y)\to \cM_{\cC_2}(\ff(x),\ff(y))$ is topologically a (trivial) covering map and for any object $x$, $\ff_{\gr_G(x)}$ is an isomorphism of $G_x$-representations.
  \end{definition}

  The notion of a cover will allow us to check easily some of the conditions~\ref{item:EFC-1}-\ref{item:EFC-7} for $\cC_1$ if they are satisfied
  for $\cC_2$.
  More precisely, we have the following result.\mpar{This part takes care of (26). We show that in most cases in the paper it is enough to focus on EFC-1 through EFC-3. We decided to
    keep the full definition of the equivariant flow category, including $\gr_G$, for possible further use. Using $\gr_G$ only via cover would also obscure the equivariant grading,
  which -- as we believe -- is an important conceptual method of this paper.}

  \begin{lemma}\label{lem:lift_to_cover}
    Suppose $\cC_2$ is a $G$-equivariant flow category, $\cC_1$ is flow category and $\ff\colon\cC_1\to\cC_2$ is a trivial cover.
    Assume there is an action of \(G\) on $\cC_1$ satisfying conditions~\ref{item:EFC-1}, \ref{item:EFC-2} and~\ref{item:EFC-3}, such that $\ff$
    commutes with the action.
    Then, there is a unique structure of a $G$-equivariant flow category on $\cC_1$ such that $\ff$ is a trivial $G$-cover.
  \end{lemma}
  \begin{proof}
    For an element $x\in\ob(\cC_1)$ we set $\gr_G(x)=gr_G(\ff(x))$.
    Then~\ref{item:EFC-4}-\ref{item:EFC-6} are satisfied.
    Condition~\ref{item:EFC-7} follows from the fact that the \(G\)-dimension is preserved under
    maps that are local \(G\)-diffeomorphisms.
  \end{proof}

  \subsection{Equivariant cube flow category}\label{sub:group-actions-cube-flow-cat}
  Recall that objects of the cube flow category are elements of $\{0,1\}^{n}$.
  If $\sigma \in \Sn$ is a permutation of an $n$-element set such that $\sigma^{m}=id$, then $\sigma$ induces an action of $\Z_{m}$ on $\{0,1\}^{n}$.
  As in Appendix~\ref{sub:combiposets} consider the action of \(\Z_{m}\) on \(\R^{n}\) defined by formula~\eqref{eq:sigma-action}.
  We will denote this representation by \(V_{\sigma}\).
  For \(x \in \ob(\Cube(n))\) denote by \((\Z_{m})_{x}\) the isotropy group of \(x\) and consider the following \((\Z_{m})_{x}\)-representation
  \[V_{x} = \prod_{i \colon x_{i}=1} \R \subset V_{\sigma}.\]

  \begin{proposition}\label{prop:equiv_cube_flow_cat}
    Let $\sigma \in \Sn$ satisfy $\sigma^m=id$.
    The cube flow category $\Cube(n)$ can be equipped with the structure of a $\Z_m$-equivariant flow category such that the action on the set of objects is generated by $\sigma$.
    Moreover, for any object $x$ we have $\gr_{\Z_{m}}(x) = V_{x}$.
  \end{proposition}
  \begin{proof}
    For $x \in \ob(\Cube(n))$ and $1 \leq k \leq m$ we define $\cGsk(x) = \sigma^k(x)$.
    In order to define $\cGsk$ on morphisms (it is enough to define $\cGs$ only), recall that we regard $\cM_{\Cube(n)}(x,y)$ as a  subset of $\prod_{i \colon x_i > y_i} \R \subset \R^n$.
    Now, $\sigma$ yields a linear isomorphism
    \begin{center}
      \begin{tikzpicture}
        \matrix(m)[matrix of math nodes, row sep = 1cm, column sep = 2cm]
        {
          \cM_{\Cube(n)}(x,y) & \cM_{\Cube(n)}(\sigma(x),\sigma(y)) \\
          \prod_{i \colon x_i > y_i} \R & \prod_{i \colon \sigma(x)_i > \sigma(y)_i} \R. \\
        };
        \path[->,font=\scriptsize]
        (m-1-1) edge node[above] {$\bar{\sigma}$} (m-1-2)
        (m-2-1) edge node[above] {$\bar{\sigma}$} (m-2-2);
        \path[right hook->]
        (m-1-1) edge node {} (m-2-1)
        (m-1-2) edge node {} (m-2-2);
      \end{tikzpicture}
    \end{center}
    Therefore, we can define $(\cGs)_{x,y} = \bar{\sigma}|_{\cM_{\Cube(n)}(x,y)}$.
    Lemma~\ref{lemma:equiv-dim-permutohedron} implies that conditions~\ref{item:EFC-1},~\ref{item:EFC-2} and~\ref{item:EFC-3} are satisfied.

    In order to define the grading function
    \[\gr_{\Z_{m}} \colon \ob(\Cube(n)) \to \bigsqcup_{H \subset \Z_m} RO(H),\]
    observe that Lemma~\ref{lemma:equiv-dim-permutohedron} implies that for any $x \in \ob(\Cube(n))$, the space $\cM_{\Cube(n)}(x,0_n)$ is a $(V_{x} - \R)$-dimensional manifold.
    Therefore, in order to satisfy condition~\ref{item:EFC-7}, the only choice for $\gr_{\Z_{m}}(x)$ is $\gr_{\Z_{m}}(x) = V_{x}$.\mpar{(20): we refer to Lemma~\ref{lemma:equiv-dim-permutohedron} to argue that EFC-7 holds.}
    Conditions~\ref{item:EFC-4},~\ref{item:EFC-5} are satisfied automatically.
  Condition~\ref{item:EFC-6} is satisfied, indeed, if $g = \sigma^k \in \Z_m$ and $y = g \cdot x$,\mpar{(21)} then the map
    \[\sigma^k \colon V_{x} \to V_{y}\]
    gives the required identification of the $(\Z_{m})_{x}$-representation $\gr_{\Z_{m}}(x)$ and the ($g (\Z_{m})_{x} g^{-1}$)-representation $\gr_{\Z_{m}}(y)$.
  \end{proof}

  \begin{corollary}
    Using the notation from Proposition~\ref{prop:equiv_cube_flow_cat}, suppose that \(\sigma\) is a product of \(n/m\) disjoint cycles of lenght \(m\). Then, for any \(x \in \ob(\Cube(n))\) we have \(\gr_{\Z_{m}}(x) = \R[(\Z_{m})_{x}]^{\gr(x)/|(\Z_{m})_x|}\).
  \end{corollary}
  \begin{proof}
    If \(\sigma\) is a product of \(n/m\) disjoint cycles of length \(m\), then \(V_{\sigma} \cong \R[\Z_{m}]^{n/m}\).
    It is easy to verify that \(V_{x} \cong \R[(\Z_{m})_{x}]^{\gr(x)/|(\Z_{m})_{x}|}\).
  \end{proof}

  \begin{definition}
    Given $\sigma \in \Sn$ such that $\sigma^m=id$, we denote by $\Cubes(n)$ the $\Z_{m}$-equivariant cube flow category for which the action on objects is generated by $\sigma$.

    Let $\cC$ be a $\Z_m$-equivariant flow category.
    We say that $\cC$ is a \emph{$\Z_{m}$-equivariant cubical flow category} if it is a cubical flow category and, for some $\Z_{m}$-virtual representation $V-W$ and some $\sigma \in \Sn$ satisfying $\sigma^m=id$, the functor $\ff \colon \Sigma^{V-W}\cC \to \Cubes(n)$ is a $\Z_m$-equivariant cover.

  \end{definition}
  \begin{remark}
    In the construction of the equivariant Khovanov homotopy type it is enough to restrict to categories $\Cubes(n)$, where $\sigma$ is a product of $n/m$ distinct cycles of length $m$.
  \end{remark}

  \begin{remark}
    Note that the constructions in this section work equally well with any fixed subgroup \(G \subset \Sn\).
    We restrict our attention to cyclic groups because this is the only relevant case for us.
  \end{remark}

  \subsection{Equivariant neat embedding}\label{sub:equiv-geom-real}

  Let $G=\Z_m$  and let $(\cC,\ff)$ be a $G$-equivariant cubical flow category.
  Fix a sequence $\ec = (e_{0},e_{1},\ldots,e_{n-1})$ of positive integers.
  For an orthogonal $G$-representation $V$ and any $u>v \in \Ob(\Cubes(n))$ define
  \begin{equation}\label{eq:EVuv}E(V)_{u,v} = \prod_{i=|v|}^{|u|-1} B_{R}(V_{u,v})^{e_{i}} \times \cM_{\Cubes(n)}(u,v),
  \end{equation}
  where $B_{R}(V)$ denotes the closed ball in \(V\) centered at \(0\) and of radius $R$. 
  We abbreviate $V_{u,v} = V|_{G_{u,v}}$,
  recalling that the symbol $V|_H$ denotes the restriction of the representation to the subgroup $H$, the underlying linear
  space is the same.

  \begin{definition}\label{def:ECNE}
    An \emph{equivariant cubical neat embedding} of a cubical flow category $(\cC,\ff)$ relative to $\ec$ and relative to the representation $V$, is a
    cubical neat embedding (i.e. satisfying axioms \ref{item:CNE-1},~\ref{item:CNE-2} and~\ref{item:CNE-3}). The maps $\iota_{x,y}$ are required
    to be $G_{x,y}:=G_x\cap G_y$-equivariant.  Furthermore, any $x,y \in \Ob(\cC)$ and any $g \in G$ the following diagram is commutative\mpar{(23): We significantly changed the logic
      of the definition of ECNE. The old definition was built as an analogy of a CNE. Now an ECNE is a CNE satisfying an extra condition. This makes the paper more concise and
    takes care of (23).}
    \begin{center}
      \begin{tikzpicture}
        \matrix(m)[matrix of math nodes,row sep=1cm,column sep=2cm]{
          \cM_{\cC}(x,y)                 & E(V)_{\ff(x),\ff(y)} \\
          \cM_{\cC}(gx,gy) & E(V)_{\ff(gx),\ff(gy)}. \\
        };
        \path[->]
        (m-1-1) edge node[above] {$\iota_{x,y}$} (m-1-2)
        (m-1-1) edge node[left] {$(\cG_{g})_{x,y}$} (m-2-1)
        (m-1-2) edge node[right] {$g \cdot (-)$} (m-2-2)
        (m-2-1) edge node[above] {$\iota_{gx,gy}$} (m-2-2);
      \end{tikzpicture}
    \end{center}
    The right vertical arrow is labeled by $g\cdot(-)$, which should be read that the map is induced by the group action. More specifically,
    $G$ acts on $V$ and $g\in G$ takes $V_{u,v}$ to $V_{gu,gv}$.
    Moreover, $g$ takes $\cM_{\Cubes(n)}(u,v)$ to $\cM_{\Cubes(n)}(gu,gv)$.
    Combining these actions we have the map
    that takes $E(V)_{\ff(x),\ff(y)}$ to $E(V)_{\ff(gx),\ff(gy)}$.
    This is the right vertical map in the above diagram.\mpar{(24) We gave an explanation of the right vertical map, along the lines sketched by the referee.}

  \end{definition}
  \begin{remark}
    From Definition~\ref{def:ECNE} it follows that the diagrams of maps
    in~\ref{item:CNE-1} and~\ref{item:CNE-3} are diagrams of $G_{x,y}$
    maps (in case of \ref{item:CNE-1}), respectively $G_{x,y,z}=G_x\cap G_y\cap G_z$ (in case of \ref{item:CNE-3}).
  \end{remark}
  \begin{proposition}\label{prop:equivariant_neat}
    Any equivariant cubical flow category admits an equivariant cubical neat embedding.
  \end{proposition}
  \begin{proof}
    Consider $x,y\in\Ob(\cC)$.
    The space $\cM_{\cC}(x,y)$ is, by~\ref{item:EFC-7} a compact $G_{x,y}$-manifold of dimension $\gr_G(x)|_{G_{x,y}}-\gr_G(y)|_{G_{x,y}}-\R$.
    In particular, by the Mostow-Palais Theorem (see Theorem~\ref{thm_mostowpalais}) there exists a representation $W_{x,y}$ such that $\cM_{\cC}(x,y)$ embeds in $W_{x,y}$.
    Define $V$ to be the direct sum of $W_{x,y}$ over all pairs $x,y\in\Ob(\cC)$.

    We want to construct embeddings $\iota_{x,y}\colon\cM_{\cC}(x,y)\to E(V)_{\ff(x),\ff(y)}$.
    Recall we have $E(V)_{u,v}=\prod_{i=|v|}^{|u|-1} B_R(V_{u,v})^{e_i}\times\cM_{\Cubes(n)}(u,v)$.
    The map $\iota_{x,y}$ will be a product $j_{x,y}\times\ff$, where
    \[j_{x,y}\colon \cM_{\cC}(x,y)\to\prod_{i=|\ff(y)|}^{|\ff(x)|-1}B_R(V_{\ff(x),\ff(y)})^{e_i}\]
    and $\ff\colon\cM_{\cC}(x,y)\to\cM_{\Cubes(n)}(\ff(x),\ff(y))$ is given by the definition of the cubical flow category (see Definition~\ref{def:cubical-flow-category} above).

    Our task is therefore to construct the map $j_{x,y}$. We shall proceed by induction on $\delta=|\ff(x)|-|\ff(y)|$.
    For $\delta=1$, the space $\cM_{\cC}(x,y)$ is a finite set of points. The construction of $j_{x,y}$ in this case is obvious.
    Conditions~\ref{item:CNE-1},~\ref{item:CNE-2} are satisfied, while~\ref{item:CNE-3} is empty. The diagram in Definition~\ref{def:ECNE} commutes.

    Suppose the embedding has been constructed for all $x,y$ with $\delta<k$ and we aim to construct a map $j_{x,y}$ for $|\ff(x)|-|\ff(y)|=k$.
    By the induction assumption, the map $j_{x,y}$ is defined already on the boundary of $\cM_{\cC}(x,y)$.
    We extend this map to a $G$-equivariant map on the whole of $\cM_{\cC}(x,y)$ by Lemma~\ref{lem:equivversion}, maybe increasing the values of some of the $e_i$.
    Conditions~\ref{item:CNE-1} and~\ref{item:CNE-2} for $j_{x,y} \times\ff$ is trivially satisfied. 
    Condition~\ref{item:CNE-3} follows from the construction, because $j_{x,z}$ on the interior of $\cM_{\cC}(x,z)$ is an extension of $j_{x,z}$ on the boundary. Commutativity of the diagram in Definition~\ref{def:ECNE}
    follows from equivariance of $j_{x,y}$.
  \end{proof}

  The next step in the construction of Lawson, Lipshitz, and Sarkar is the construction of a framed cubical neat embedding.
  The notion of an equivariant framed cubical neat embedding is a direct generalization of the notion of a framed cubical neat embedding.
  Namely, given the set of maps $\iota_{x,y}\colon \cM_{\cC}(x,y)\to E(V)_{\ff(x),\ff(y)}$ constituting an equivariant cubical neat embedding (see Definition~\ref{def:ECNE}), an equivariant \emph{framed} cubical neat embedding is an extension of $\iota_{x,y}$ to equivariant maps
  \[\ol{\iota}_{x,y}\colon \prod_{i=|\ff(y)|}^{|\ff(x)|-1}B_{\varepsilon}(V_{\ff(x),\ff(y)})^{e_i}\times\cM_{\cC}(x,y)\to E(V)_{\ff(x),\ff(y)}.\]
  We require that
  that~\ref{item:CNE-3} holds for $\iota_{x,y}$ replaced by $\ol{\iota}_{x,y}$ and $\cM_{\cC}(x,y)$ replaced by the product $\prod_{i=|\ff(y)|}^{|\ff(x)|-1}B_{\varepsilon}(V_{\ff(x),\ff(y)})^{e_i}\times\cM_{\cC}(x,y)$.

  In the non-equivariant setting, passing from a cubical neat embedding to a framed cubical neat embedding is described in Subsection~\ref{sub:framed}.
  In the equivariant setting, no adjustments are needed, because the projections $\pi_{u,v}^{R}$ and $\pi_{u,v}^{M}$ considered in Subsection~\ref{sub:framed} are already equivariant by construction.

  \subsection{Equivariant cubical realization}\label{sub:equivariant_geometric}
  We consider now an analog of the construction of a CW-complex $||\cC||$
  given in Subsection~\ref{sub:geometric}.
  For each $x\in\Ob(\cC)$ such that $u=\ff(x)$ we define:\mpar{(25) We removed the incorrect statement that $\ep$ is a $G$-cell.}
  \begin{equation}\label{eq:EP_def}
    \ep(x)=\prod_{i=0}^{|u|-1}B_R(V)^{e_i}\times\prod_{i=|u|}^{n-1}B_{\epsilon}(V)^{e_i}\times\wt{\cM}_{\Cubes(n)}(u,0),
  \end{equation}
  where $\wt{\cM}(u,0)=[0,1]\times\cM(u,0)$ if $u\neq0$ and $\wt{\cM}(0,0)=\{0\}$.
  The group action on the interval $[0,1]$ is assumed to be trivial.
  Note that $\ep(x)$ is homeomorphic to the cell $X(x)$ constructed in Subsection~\ref{sub:geometric} (we need to set $d_i=e_i\dim V$), the point is that the present construction is equivariant.

  For $x,y\in\Ob(\cC)$ such that $\ff(x)=u > v = \ff(y)$ we construct a map $\ET_{y,x}\colon \ep(y)\times\cM_{\cC}(x,y)\hookrightarrow \ep(x)$ by
  an analogous formula as~\eqref{eq:thetadef} in Subsection~\ref{sub:geometric}, namely

  \begingroup
  \allowdisplaybreaks
  \begin{align}
    &\ep(y) \times \cM_{\cC}(x,y) = \label{eq:ethetadef}\\
    &= \prod_{i=0}^{|v|-1} B_R(V)^{e_{i}} \times \prod_{i=|v|}^{n-1} B_{\epsilon}(V)^{e_{i}} \times \widetilde{\cM}_{\Cubes(n)}(v,0) \times \cM_{\cC}(x,y)  \nonumber\\
    &\cong \prod_{i=0}^{|v|-1} B_R(V)^{e_{i}} \times \prod_{i=|u|}^{n-1} B_\epsilon(V)^{e_{i}} \times \widetilde{\cM}_{\Cubes(n)}(v,0) \times \left(\prod_{i=|v|}^{|u|-1} B_\epsilon(V)^{e_{i}} \times \cM_{\cC}(x,y)\right) \nonumber\\
    &\hookrightarrow \prod_{i=0}^{|v|-1} B_R(V)^{e_{i}} \times \prod_{i=|u|}^{n-1} B_\epsilon(V)^{e_{i}} \times \widetilde{\cM}_{\Cubes(n)}(v,0) \times\nonumber\\
    &\times \left(\prod_{i=|v|}^{|u|-1} B_R(V)^{e_{i}} \times \cM_{\Cubes(n)}(u,v)\right) \nonumber\\
    &\cong \prod_{i=0}^{|u|-1} B_R(V)^{e_{i}} \times \prod_{i=|u|}^{n-1} B_\epsilon(V)^{e_{i}} \times \widetilde{\cM}_{\Cubes(n)}(v,0) \times \cM_{\Cubes(n)}(u,v)\hookrightarrow \nonumber\\
    &\hookrightarrow \prod_{i=0}^{|u|-1} B_R(V)^{e_{i}} \times \prod_{i=|u|}^{n-1} B_\epsilon(V)^{e_{i}} \times \partial(\widetilde{\cM}_{\Cubes(n)}(u,0)) \subset \ep(x).\nonumber
  \end{align}
  \endgroup
  In fact, with the choice of $d_i=e_i\dim V$ and an identification $\ep(x)\cong X(x)$, $\ET$ is exactly the same map as $\theta$.
  Again the key point is that $\ET_{x,y}$ is $G_{x,y}$-equivariant.
  Write $\ep_y(x)\subset\ep(x)$ for the image of $\ET(y)$.

  Analogously to the non-equivariant case, the complex $||\cC||$ is constructed inductively by taking the cells $\ep(x)$ and the attaching map taking $\ep_y(x)$ to $\ep(y)$ via the projection $\ep_y(x)\cong \ep(y)\times\cM_{\cC}(x,y) \hookrightarrow\ep(y)$.
  As in the non-equivariant case, the remaining part $\partial\ep(x)\setminus\bigcup_y\ep_y(x)$ is mapped to the base point.

  \begin{remark}
    The map $E\theta_{y,x}$ gives a well-defined attaching map, see item (4) in Subsection~\ref{sub:geometric} and equation~\eqref{eq:commute}.
    This is because, as we mentioned above, $E\theta$ is essentially the map $\theta$ from Subsection~\ref{sub:geometric}.
    Another possibility is to observe that the map $\kappa_{x,y,z}$ constructed in the proof of~\cite[Lemma 3.16]{LLS_long} is equivariant because of the axioms~\ref{item:EFC-1}–\ref{item:EFC-3}.
    We omit the details.
  \end{remark}

  \begin{proposition}
    The space $||\cC||$ has the structure of a $G$-cell complex of Definition~\ref{def:g-cell}. 
  \end{proposition}
  \begin{proof}
    If $x_{1}, x_{2},\ldots,x_{k}$ is an orbit of $x_1\in\Ob(\cC)$, then there exists an equivariant homeomorphism\mpar{(27): we replaced $\gr_G$ with a closed ball in the representation
    space (not a one point compactification)}
    \begin{equation}\label{eq:G-cell-homeo}
      \ep(x_{1}) \sqcup \ep(x_{2}) \sqcup \ldots \sqcup \ep(x_{k}) \cong G_{x_{1}} \times_{G_{x}} \left(\prod_{i=0}^{|u|-1}B_R(V)^{e_i}\times\prod_{i=|u|}^{n-1}B_{\epsilon}(V)^{e_i} \times B_{R}(\gr_{G}(x_{1}))\right),
    \end{equation}
    i.e.\ we obtain a $G$-cell of type $(G_{x},V^{e_{1}+\cdots+e_{n-1}} \oplus \gr_{G}(x_{1}))$.
    It is easy to verify that the gluing maps are compatible with the homeomorphism from~\eqref{eq:G-cell-homeo}.
  \end{proof}

  \begin{definition}\label{def:equivariant_cubical}
    The \emph{equivariant cubical realization} of $\cC$ is defined to be the $G$-cell complex $||\cC||$.
    The formal desuspension $\cX(\cC):=\Sigma^{-W-V^{e_{0}+\cdots+e_{n-1}}}||\cC||$, where $W$ denotes a representation of $G$ such that $\ff \colon \Sigma^{W}\cC \to \Cubes(n)$ is the cubical functor, is called the \emph{equivariant $\cC$ homotopy type}.
  \end{definition}
  \begin{remark}
    By a formal desuspension \(\Sigma^{-W-V^{e_{0}+\cdots+e_{n-1}}}||\cC||\) of \(||\cC||\) we mean the object of the Spanier-Whitehead category \(SW_{G}\) given by \((||\cC||, -W-V^{e_{0}+\cdots+e_{n-1}})\); compare Subsection~\ref{sub:spanier_whitehead}.
    Note that Stoffregen-Zhang's equivariant Khovanov homotopy type belongs to the same category~\cite[discussion below Proposition 4.17]{StoffregenZhang}.
  \end{remark}
  The following result is a direct consequence of the construction:\mpar{(28) we rephrased the next proposition. We also added Subsection~\ref{sub:equiv_cell} to explain
  the notion of a $G$-CW-complex in more detail.}
  \begin{proposition}\label{prop:conclude}
    Let $(\cC, \ff \colon \Sigma^{W} C \to \Cubes(n))$ be a $G$-equivariant cubical flow category.
    Let $\iota$ be an equivariant cubical neat embedding relative to $\ec = (e_{0},e_{1},\ldots,e_{n-1}) \in \mathbb{N}^{n}$ and relative to an orthogonal $G$-representation $V$.
    There exists a $G$-cell complex $||\cC||$,
    such that every object $x \in \Ob(\cC)$ corresponds to a single cell of $||\cC||$ of dimension $\gr_{G}(x)$.
    Moreover, the forgetful functor \textup{(}i.e.\ the one which forgets the action of $G$\textup{)} maps $||\cC||$ to the stable homotopy type constructed by Lawson, Lipshitz and Sarkar in~\cite{LS_stable}.
  \end{proposition}

  \subsection{Fixed points of the cubical realization}\label{sub:fixed}
  The purpose of this subsection is to study the fixed point sets (with respect to a subgroup $H$) of the group action on the
  cubical realization.\mpar{(29)}
  The results  will play an essential role in the proof of the invariance of the equivariant Khovanov homotopy type under Reidemeister moves.
  Recall that $X^H$ denotes the set of fixed points of $H$, that is, $X^H=\{x \in X \, | \, x \cdot h = x, \forall h \in H\}$.

  Let $\cC$ be an equivariant cubical flow category.
  For any $H \subset G$ define the \emph{$H$-fixed subcategory} $\cC^H$ in the following way.
  \begin{itemize}
  \item The objects of $\cC^H$ are those objects of $\cC$ that are fixed under the action of $H$, that is $\Ob(\cC^H)=\Ob(\cC)^H$;
  \item The  morphisms between objects are given by fixed point submanifolds, that is,
    \[\cM_{\cC^H}(x,y)=\begin{cases} \cM_{\cC}(x,y)^H &x\neq y\\ \{id\} & x=y;\end{cases}\]
  \item The grading of $x\in\Ob(\cC)^H$ is $\dim \gr_{G}(x)^{H}$.
  \end{itemize}
  \begin{remark}\label{rem:subgroup}
    If $H$ is a normal subgroup of $G$ (in the paper we work with $G$ cyclic, so any subgroup of $G$ is normal), it is possible to endow $\cC^H$ with the structure of a $G/H$-equivariant flow category.
  \end{remark}
  We will now give an instance of an $H$-fixed subcategory that is the most important in our approach.
  \begin{proposition}\label{prop:fixed-cube-flow-cat}
    Let $H$ be a subgroup of $\Z_m$ and consider $\Cubes(n)$ for $\sigma \in \Sn$ such that $\sigma^m=id$.
    Then there is a functor ${\rH}\colon \Cubes(n)^H\to\Cube(n')$ that induces an isomorphism of categories.

    The integer $n'$ is calculated as follows.
    If $\sigma$ is a product of $p$ disjoint cycles $(a_{i1},\ldots,a_{in_i})$ with $n_i|m$ and $\sum_{i=1}^{p} n_i=n$,
    then we set $\ell_i=\gcd(n_i,m/|H|)$ and $n'_i=n_i/\ell_i$.
    We have $n'=\sum n'_i$.
  \end{proposition}
  \begin{proof}
    The key idea is to use Theorem~\ref{thm:generalpermut}. There is a technical difficulty namely Theorem~\ref{thm:generalpermut}
    does not give us a canonical diffeomorphism. Therefore we first fix a concrete diffeomorphism between $\cM_{\Cube(n)}(1_n,0_n)^{H}$ and 
    $\cM_{\Cube(n')}(1_{n'},0_{n'})$, next we show that it can be used to define a map between all moduli spaces of the $\Cube(n)^H$ category
    and corresponding moduli spaces of the $\Cube(n')$ category.\mpar{This paragraph partially solves the issue raised in (32).}

    \smallskip
    To begin with, if $(v_1,\ldots,v_n)$ is an object in $\Ob(\Cubes(n)^H)$, then by definition it is an object in $\Cubes(n)$ fixed by the action of $H$.
    This amounts to saying that, for $i=1,\ldots,p$ and $j=1,\ldots,n'_i$, we have
    \[v_{i,j}=v_{i,j+n'_i}=\cdots=v_{i,j+(\ell_i-1)n'_i},\]
    where to simplify the notation we write $v_{i,j}$ instead of $v_{a_{i,j}}$.
    The functor ${\rH}$ on objects is defined as
    \begin{equation}\label{eq:rhob}
      {\rH}(v_1,\ldots,v_n)=(v_{1,1},\ldots,v_{1,n'_1},v_{2,1},\ldots,v_{p,n'_p}).
    \end{equation}
    We now define ${\rH}$ on morphisms.

    Consider first $0_n,1_n\in\Cubes(n)$.
    They are fixed under the action of any subgroup $H \subset G$.
    The space $\cM_{\Cubes(n)}(1_n,0_n)$ is, by definition, the permutohedron $\Pi_{n-1}\subset\R^n$.
    The set $\cM_{\Cubes(n)}(1_n,0_n)^H$ of fixed points under $H$ is given by $\Pi_{n-1}\cap L$, where $L$ is a linear subspace of $\R^n$ given by
    \[L=\bigcap_{i=1}^{p}\bigcap_{j=1}^{n'_i}\{x_{i,j}=x_{i,j+n'_i}=\cdots=x_{i,j+(\ell_i-1)n'_i}\},\]
    where we also used the notation $x_{i,j}$ as a shorthand for $x_{a_{i,j}}$.
    We note that the dimension of $L$ is precisely $\sum n'_i=n'$.

    By Theorem~\ref{thm:generalpermut} there is an identification $\psi$ of $\Pi_{n-1}\cap L$ with $\Pi_{n'-1}$.
    Choose one such $\psi$.
    The map $\psi$ identifies $\cM_{\Cubes(n)}(1_n,0_n)^H$ with $\cM_{\Cube(n')}(1_{n'},0_{n'})$.

    Take now general $u,v\in\Ob(\Cubes(n))^H$ with $u>v$. We assume that $u\neq 1_n$, $v\neq 0_n$. The case where precisely
    one inequality holds is analogous and it is left to the reader.\mpar{(31) we added the assumption. The other case is similar.}
    Consider the product
    \[\Pi_{u,v}=\cM_{\Cubes(n)}(v,0_n)\times\cM_{\Cubes(n)}(u,v)\times\cM_{\Cubes(n)}(1_n,u).\]
    By the axioms of the cube category $\Pi_{u,v}$ embeds as a codimension~$2$ face in the moduli space $\cM_{\Cubes(n)}(1_n,0_n)=\Pi_{n-1}$.

    In fact, consider the partition
    \begin{equation}\label{eq:partition_def}
      \fp=\mathcal{P}_{1_n,u}\cup\mathcal{P}_{u,v}\cup\mathcal{P}_{v,0_n},
    \end{equation}
    and $\mathcal{P}_{u',v'}$ (with $(u',v')=(1_n,u)$, $(u',v')=(u,v)$ and $(u',v')=(v,0_n)$) is the set of indices $i$ such that $u'_i\neq v'_i$.
    Then $\Pi_{u,v}$ corresponds to the face $\Pi_{\fp}$.

    We define now the map ${\rH}\colon \cM_{\Cubes(n)}(u,v)^H\to\cM_{\Cube(n')}({\rH}u,{\rH}v)$ as the composition:
    \begin{align*}
      \cM_{\Cubes(n)}(u,v)^H&\to \cM_{\Cubes(n)}(v,0_n)^H \times\cM_{\Cubes(n)}(u,v)^H \times\cM_{\Cubes(n)}(1_n,u)^H\to\\
                            &\xrightarrow{\psi} \cM_{\Cube(n')}({\rH}v,0_{n'}) \times \cM_{\Cube(n')}({\rH}u,{\rH}v) \times \cM_{\Cube(n')}(1_{n'},{\rH}v)\to\\
                            &\to \cM_{\Cube(n')}({\rH}u,{\rH}v).
    \end{align*}
    The first map is an embedding to a fiber $\{pt\}\times \cM_{\Cubes(n)}(u,v)^H\times\{pt\}$ for two chosen points in $\cM_{\Cubes(n)}(v,0_n)^H$ and $\cM_{\Cubes(n)}(1_n,u)^H$, respectively.
    ${\rH}$ does not depend on the choice.
    The last map is the projection onto the second factor.
    The map $\psi$ was defined above as a map from $\Pi_{n-1}\cap L$ to $\Pi_{n'-1}$.
    It takes the face $\Pi_{\fp}$ to the face $\Pi'_{\fp^B}$, where $\Pi'=\Pi_{n'-1}$ and $\fp^B$ is a reduction of $\fp$
    with $B$ determined from the orbits of $\sigma$; see Proposition~\ref{prop:fixed_permut}.
    A straightforward calculation using~\eqref{eq:rhob} reveal that $\fp^B$ is a partition into three subsets $\mathcal{P}'_{{\rH} v,0_{n'}}$, $\mathcal{P}'_{{\rH} u,{\rH} v}$ and $\mathcal{P}'_{1_{n'},{\rH} u}$, where $\mathcal{P}'_{\cdot,\cdot}$ denotes the subset of indices at which the vectors in
    the subscripts differ.
    This means that  $\Pi'_{\fp^B}$ is exactly $\cM_{\Cube(n')}({\rH}v,0_{n'}) \times \cM_{\Cube(n')}({\rH}u,{\rH}v) \times \cM_{\Cube(n')}(1_{n'},{\rH}u)$; we omit the details.

    \smallskip
    We sketch the proof of the fact that ${\rH}$ respects the compositions.
    Suppose that $u,w,v\in\Ob(\Cubes(n))^H$ with $u>w>v$. 
    Let ${\rH}u,{\rH}w,{\rH}v$ be the corresponding objects in $\Cube(n')$.
    We need to show that the following diagram commutes.
    \begin{equation}\label{eq:commutes_H}
      \xymatrix{%
        \cM_{\Cubes(n)^H}(w,v)\times \cM_{\Cubes(n)^H}(u,w)\ar[r]\ar[d]^{{\rH}} &
        \cM_{\Cubes(n)^H}(u,v)\ar[d]^{{\rH}}\\
        \cM_{\Cube(n')}({\rH}w,{\rH}v)\times \cM_{\Cube(n')}({\rH}u,{\rH}w)\ar[r] &
        \cM_{\Cube(n')}({\rH}u,{\rH}v).}
    \end{equation}
    This commutativity is true if $\psi$ takes
    \[M_1=\cM_{\Cubes(n)}(1_n,u)^H
      \times\cM_{\Cubes(n)}(u,w)^H
      \times\cM_{\Cubes(n)}(w,v)^H
      \times\cM_{\Cubes(n)}(v,0_n)^H\]
    to
    \[
      M_2=\cM_{\Cube(n')}(1_{n'},{\rH}u) \times \cM_{\Cube(n')}({\rH}u,{\rH}w) \times \cM_{\Cube(n')}({\rH}w,{\rH}v) \times \cM_{\Cube(n')}({\rH}v,0_{n'}).\]
    Consider the refinement $\fp_w$ of the partition $\fp$ defined in \eqref{eq:partition_def} given as
    \[\fp_w=\mathcal{P}_{1_n,u}\cup\mathcal{P}_{u,w}\cup\mathcal{P}_{w,v}\cup\mathcal{P}_{v,0_n},\]
    where the subsets $\mathcal{P}_{\cdot,\cdot}$ are as above (below \eqref{eq:partition_def}). Let $\Pi_{u,w,v}$ be the face corresponding
    to this partition. Then $M_1=\Pi_{u,w,v}\cap L$.
    By construction of $\psi$, it takes $A$ to a face $\Pi'_{(\fp_w)^B}$ of $\Pi'$, where $(\fp_w)^B$ is the reduction of $\fp_w$.

    On the other hand, the reduction $(\fp_w)^B$ is easily seen to be the partition
    \[(\fp^B)_w=\mathcal{P}'_{1_{n'},{\rH} u}\cup \mathcal{P}'_{{\rH} u,{\rH} w}\cup \mathcal{P}'_{{\rH} w,{\rH} v}\cup \mathcal{P}'_{{\rH} v,0_{n'}}.\]
    But then the corresponding face is $\Pi'_{(\fp_w)^B}=M_2$.\mpar{(32) The proof has been streamlined rather than extended. Many small changes should improve on readability
      of the argument that \eqref{eq:commutes_H}, which essentially boils down to the statement that refinements commute with reductions. We can still add some more
      details, but the proof might eventually become less readable. Making the morphism $\psi$ in 
    Proposition~\ref{prop:permutinter} might be possible, but it would definitely required a \emph{much longer} proof.}

    Finally, the equivariant grading on $\Cubes(n)$ described in Proposition~\ref{prop:equiv_cube_flow_cat} has the property that if $x\in\Ob(\Cubes(n))^H$, then $\gr_G(x)^H$ is equal to the grading of ${{\rH}(x)}$.
    This is a straightforward verification.
  \end{proof}

  \begin{lemma}\label{lem:isflow}
    The pair $(\cC^H, \ff^{H})$, where $\ff^{H} = {\rH} \circ \ff|_{\cC^{H}}$ and ${\rH}$ is as in Proposition~\ref{prop:fixed-cube-flow-cat}, is a cubical flow category.
  \end{lemma}
  \begin{proof}
    In order to prove that $\cC^{H}$ is a flow category, we need to verify the axioms~\ref{item:FC-1},~\ref{item:FC-2} and~\ref{item:FC-3}.
    The axiom~\ref{item:FC-1} is obvious.
    The axiom~\ref{item:FC-2} follows from the axiom~\ref{item:EFC-7} and Proposition~\ref{prop:permutinter}.
    The axiom~\ref{item:FC-3} follows from the axiom~\ref{item:EFC-3}.  This shows that $\cC^H$ is a flow category.  It remains to prove that the functor $\ff^H$ makes $\cC^H$ a cubical flow category.

    Since $\ff$ commutes with the group action, it takes objects in $\cC$ that are fixed under $H$ to objects of $\Cubes(n)$ that
    are fixed under $H$.
    In particular, $\ff^H$ is well-defined on objects.

    To show that it is well-defined on morphisms, observe that for any $x,y \in \ob(\cC)^{H}$, the map
    \[\ff_{x,y} \colon \cM_{\cC}(x,y)^{H} \to \cM_{\Cubes(n)}(\ff(x),\ff(y))^{H}\]
    is a diffeomorphism when restricted to any connected component of $\cM_{\cC}(x,y)^{H}$.
    In particular ${\rH}\circ\ff_{x,y}$ is a covering map.
    Therefore, $\ff^H$ turns $\cC^H$ into a cubical flow category.
  \end{proof}

  \begin{lemma}
    Let $\cC$ be a framed cubical flow category and $\iota$ a neat embedding of $\cC$ relative to $\ec=(e_{1},e_{2},\ldots,e_{n-1})$ and relative to a representation $V$.\mpar{(33)}
    Then, for any $H \subset G$, $\iota$ yields a neat embedding of $\cC^{H}$, denoted by $\iota^{H}$, relative to
    \[\ec^{H} = (e_{1}+\cdots+e_{k-1},e_{k}+e_{k+1}+\cdots+e_{2k-1},\ldots,e_{n-k}+e_{n-k+1}+\cdots+e_{n-1}) \text{ and } V^{H},\]
    where \(k\) denotes the order of \(H\).\mpar{(34)}
  \end{lemma}
  \begin{remark}[Remark~\ref{rem:subgroup} continued]
    One can construct $\iota^H$ in such a way that it is a $G/H$-equivariant neat embedding.
  \end{remark}
  \begin{proof}
    An equivariant neat embedding of $\cC$ is given by a collection of 
    equivariant maps $\iota_{x,y}\colon\cM_{\cC}(x,y)\to E(V)_{\ff(x),\ff(y)}$
    satisfying axioms~\ref{item:CNE-1},~\ref{item:CNE-2} and~\ref{item:CNE-3},
    see Definition~\ref{def:ECNE}.
    An equivariant neat embedding
    \[\iota_{x,y} \colon \cM_{\cC}(x,y) \to E(V)_{\ff(x),\ff(y)},\]
    where $x,y \in \Ob(\cC)^{H}$, yields an embedding
    \[\iota_{x,y}|_{H} \colon \cM_{\cC}(x,y)^{H} \to E(V)_{\ff(x),\ff(y)}^{H}.\]
    Observe that
    \[E(V)_{\ff(x),\ff(y)}^{H} = \prod_{i=|\ff(y)|}^{|\ff(x)|-1} \left(B_{R}(V_{\ff(x), \ff(y)}^{H})\right)^{e_{i}} \times \cM_{\Cubes(n)}(\ff(x),\ff(y))^{H},\]
    because $H \subset G_{\ff(x),\ff(y)}$.
    Since $|\ff(x)| = k \cdot |\ff^{H}(x)|$ and $|\ff(y)|= k \cdot |\ff(y)|^{H}$, there exists an equivariant linear embedding
    \[\eta^H \colon \prod_{i=|\ff(y)|}^{|\ff(x)|-1} B_{R}(V^{H})^{e_{i}} \hookrightarrow \prod_{i=|\ff^{H}(y)|}^{|\ff^{H}(x)|-1} B_{R’}(V^{H})^{e_{k\cdot i}+e_{k\cdot i+1}+\cdots+e_{k \cdot (i+1)-1}},\]
    for some $R’>R$.
    Using the map ${\rH}$ from Proposition~\ref{prop:fixed-cube-flow-cat} we obtain a neat embedding
    \[\eta^H \times {\rH} \colon E(V)_{\ff(x),\ff(y)}^{H} \hookrightarrow E(V^{H})_{\ff^{H}(x),\ff^{H}(y)}.\]
    We define $\iota^{H}_{x,y} = (\eta^H \times \rH) \circ (\iota_{x,y}|_{H})$.

    Properties~\ref{item:CNE-1},~\ref{item:CNE-2} and~\ref{item:CNE-3} 
    for $\iota^H_{x,y}$ follow immediately from analogous properties
    of the maps $\iota_{x,y}$.\mpar{(36): as explained in (23) we changed the definition of an ECNE. Now we verify conditions CNE (which is easy) and equivariance in
    Definition~\ref{def:ECNE}.}
  \end{proof}

  \begin{proposition}\label{prop:fixedcategory}
    Suppose $||\cC||$ is an equivariant cubical realization of an equivariant cubical flow category $\cC$.
    Then the fixed point set $||\cC||^H$ is homeomorphic to the cubical realization of the fixed point flow category $\cC^H$.\mpar{(36) This is clarified by introducing
    a formal distinction between $||\cC||$ and $\cX(\cC)$ in Definition~\ref{def:cubical_realization}.}
  \end{proposition}
  \begin{proof}
    We need to show essentially two facts: the equality of cells, and the equality of attaching maps.
    First, if $x\in\Ob(\cC)^H$, we can construct a cell $X_H(x)$ using the construction of Subsection~\ref{sub:geometric} taking $\cC^H$ as the starting category.
    This corresponds to a cell used for constructing $||\cC^H||$.
    Alternatively we can take $\ep(x)^H$ to be the set of $H$-fixed points of the cell $\ep(x)$ constructed in Subsection~\ref{sub:equivariant_geometric}.
    We claim that $X_H(x)\cong\ep(x)^H$ once we have set $d_i=e_i\dim V^H$. 

    To see this recall that by~\eqref{eq:EP_def} we have
    \begin{align*}
      \ep(x)^H&=\prod_{i=0}^{|\ff(x)|-1}B_R(V^H)^{e_i}\times\prod_{i=|\ff(x)|}^{n-1}B_{\epsilon}(V^H)^{e_i}\times\wt{\cM}_{\Cubes(n)}(\ff(x),0_n)^H, \\
      X_{H}(x)&=\prod_{i=0}^{|\ff^{H}(x)|-1}B_{R}(V^{H})^{e_{i}^{H}} \times \prod_{i=|\ff^{H}(x)|}^{n'-1}B_{\epsilon}(V^{H})^{e_{i}^{H}}\times\wt{\cM}_{\Cube(n')}(\ff^{H}(x),0_n),
    \end{align*}
    where $e_{i}^{H} = e_{k\cdot i}+e_{k\cdot i+1}+\cdots+e_{k\cdot(i+1)-1}$ and \(k\) denotes the order of \(H\).
    Discussion in Proposition~\ref{prop:fixed-cube-flow-cat} implies that $\ep(x)^{H} \cong X_{H}(x)$, for any $x \in \ob(\cC^{H})$.

    In order to complete the proof of Proposition~\ref{prop:fixedcategory}, we need to show that the attaching maps coincide.
    This holds, provided that $\theta_H(y,x)=\ET(y,x)^H$, where $\theta_H(y,x)$ is the map $\theta$ of Subsection~\ref{sub:geometric} constructed for $\cC^H$, and $\ET(y,x)^H$ is the restriction of $\ET$ to the set of fixed points.
    Choose $x,y \in\Ob(\cC)^H$.
    Going through the construction of $\theta$ and $\ET$ (given in~\eqref{eq:thetadef} and~\eqref{eq:ethetadef}) we see that the equality $\ET(y,x)^H=\theta_H(y,x)$ follows from the commutativity of the diagram
    \begin{center}
      \begin{tikzpicture}
        \matrix(m)[matrix of math nodes, row sep=1cm, column sep=1cm] {
          \cM_{\Cubes(n)}(\ff(y),0_n)^H\times\cM_{\Cubes(n)}(\ff(x),\ff(y))^H & \cM_{\Cubes(n)}(\ff(x),0_n)^H \\
          \cM_{\Cube(n)^H}(\ff^{H}(y),0_n)\times\cM_{\Cube(n)^H}(\ff^{H}(x),\ff^{H}(y)) & \cM_{\Cube(n)^H}(\ff^{H}(x),0_n), \\
        };
        \path[->,font=\scriptsize]
        (m-1-1) edge node[above] {$\circ^{H}$} (m-1-2)
        (m-1-1) edge node[left] {${\rH} \times {\rH}$} (m-2-1)
        (m-1-2) edge node[right] {${\rH}$} (m-2-2)
        (m-2-1) edge node[above] {$\circ$} (m-2-2);
      \end{tikzpicture}
    \end{center}
    where ${\rH}$ is a map from the fixed point set of a permutohedron  to a permutohedron of lower dimensions, as described in detail in the proof of Proposition~\ref{prop:fixed-cube-flow-cat}.
    Commutativity of the diagram follows from the construction of this map (see Proposition~\ref{prop:permutinter} and Proposition~\ref{prop:fixed_permut}).
  \end{proof}
  \subsection{Equivariant chain complexes}\label{sub:equivchain}
  In Subsection~\ref{sub:chain_complex} we constructed a cochain complex  $C^*(\cC,\ff)$, whose cohomology was equal to the cohomology of the cubical realization $||\cC||$.
  Suppose now that the underlying cubical flow category admits an action of the group \(G = \Z_{m}\).
  In order to describe the induced action of \(G\) on the chain complex, notice that, for any \(g \in G\), we obtain a homomorphism of abelian groups
  \[\mathcal{G}_{g} \colon C^{\ast}(\cC,\ff) \to C^{\ast}(\cC,\ff),\]
  yielding an action of \(G\).
  This action, however, does not, in general, commute with the differential on \(C^{\ast}(\cC,\ff)\).

  The differential of the chain complex~\eqref{eq:chain_complex} depends on the sign assignment $\nu$ on the cube flow category \(\Cube(n)_{\sigma}\), see~\eqref{eq:chain2}.\mpar{(37)}
  We will denote, abusing the notation, a generator of \(G\) by \(\sigma\).
  The symmetry group acts on sign assignments via
  \[\sigma(\nu)(x,y) = \nu(\sigma(x),\sigma(y)).\]
  However, in general, $\sigma(\nu) \neq \nu$, that is, the sign assignment $\nu$ is not necessarily $\sigma$-invariant.

  To remedy this, we recall that the sign assignments form a $1$-chain in $[0,1]^n$ with values in $\F_2$ (see Subsection~\ref{sub:review_khovanov}).
  The difference between any two sign assignments satisfies a cocycle condition.
  Therefore, there exists a $0$-cochain $c\in C^0([0,1]^n;\F_2)$ such that $\sigma(\nu)-\nu=\partial^*c$.
  That is,
  \begin{equation}\label{eq:iscochain}
    \nu(\sigma(\ff(x)),\sigma(\ff(y)))-\nu(\ff(x),\ff(y))= c(\ff(x))-c(\ff(y)).
  \end{equation}

  \begin{lemma}\label{lem:action}
    The map $t_{\sigma}\colon C^*(\cC)\to C^*(\cC)$ given by $x\mapsto (-1)^{c(\ff(x))}\mathcal{G}_{\sigma}(x)$ commutes with the differential and therefore it generates the $G$-action on the chain complex $C^*(\cC)$.\mpar{(38) the purpose of Lemma~\ref{lem:action} is to define the group action on $C^*(\cC)$.}
  \end{lemma}
  \begin{proof}
    We need to check that the coefficient in $\partial{t_{\sigma}(y)}$ at $t_{\sigma}(x)$ is equal to the coefficient in $\partial y$ at $x$.
    The latter is equal to
    \begin{equation}\label{eq:coef1}
      (-1)^{\nu(\ff(x),\ff(y))}\#\cM_{\cC}(x,y),
    \end{equation}
    compare to~\eqref{eq:chain2}.
    We want to compute now the former.
    Write $y'=\mathcal{G}_{\sigma}(y)$, $x'=\mathcal{G}_{\sigma(x)}$.
    By~\eqref{eq:chain2} the coefficient in $\partial y'$ at $x'$ is equal to
    \begin{equation}\label{eq:coef2}
      (-1)^{\nu(\ff(x'),\ff(y'))}\#\cM_{\cC}(x',y')=(-1)^{\nu(\ff(x'),\ff(y'))}\#\cM_{\cC}(x,y),
    \end{equation}
    Given the definition of $t_{\sigma}$, we have $t_{\sigma}(x)=(-1)^{c(\ff(x))}x'$ and $t_{\sigma}(y)=(-1)^{c(\ff(y))}y'$.
    Thus, in light of~\eqref{eq:coef2}, the coefficient in $\partial t_{\sigma}{G}(y)$ at $t_{\sigma}(x)$ is given by
    \begin{equation}\label{eq:coef3}
      (-1)^{c(\ff(x))+c(\ff(y))+\nu(\ff(\sigma(x)),\ff(\sigma(y)))}\#\cM_{\cC}(x,y).
    \end{equation}
    Finally, to show the equality of~\eqref{eq:coef1} and~\eqref{eq:coef3} we need to guarantee that
    \[c(\ff(x))+c(\ff(y))+\nu(\ff(\sigma(x)),\ff(\sigma(y)))=\nu(\ff(x),\ff(y))\bmod 2,\]
    but this follows immediately from~\eqref{eq:iscochain}.
  \end{proof}
  \begin{remark}
    This sign problem is not uncommon. It appears in the construction of the equivariant Khovanov homology~\cite[Section 2]{Politarczyk_Khovanov}.
    The approach in~\cite{Politarczyk_Khovanov} is essentially the same as the one we use here, but it is expressed in a different language.
  \end{remark}

  \subsection{Equivariant subcategories}

  Suppose that  $\cC'$ is an equivariant downward closed subcategory of $\cC$.
  Let $\cC''$ be the complementary upward closed subcategory.
  As $\cC'$ is invariant under the group action, the subcategory $\cC''$ is also an invariant subcategory.

  The following result is a direct generalization of~\cite[Lemma 3.32]{LS_stable}.
  \begin{proposition}\label{prop:puppewpuppe}
    If $\cC$, $\cC'$ and $\cC''$ are as above, then there exist three equivariant maps, an inclusion $\iota\colon||\cC'||\to||\cC||$, a collapse $\kappa\colon||\cC||\to||\cC''||$ and the Puppe map $\rho\colon||\cC''||\to\Sigma||\cC'||$, that induce the following cohomology long exact sequence
    \begin{equation}\label{eq:LSexact}
      \ldots\to \wt{H}^i(||\cC||)\stackrel{\iota^*}{\to} \wt{H}^i(||\cC'||)\stackrel{\rho^*}{\to}\wt{H}^{i+1}(||\cC''||)\stackrel{\kappa^*}{\to}\ldots.
    \end{equation}
  \end{proposition}

  Suppose $\cC$ is an equivariant cubical flow category, $\cC'$ is a downward closed subcategory, and $\cC''$ is the complementary upward closed category, and let the maps $\iota,\kappa$ and $\rho$ be as in Proposition~\ref{prop:puppewpuppe}.
  We ask under which conditions one of these maps is an equivariant homotopy equivalence.
  This holds under some extra assumptions that we spell in Lemma~\ref{lem:hardertocheck}.
  Although these assumptions are harder to verify, the methods developed in Subsection~\ref{sub:fixed} simplify the process.
  \begin{lemma}\label{lem:hardertocheck}
    Let $\iota,\kappa$ and $\rho$ be as described in Proposition~\ref{prop:puppewpuppe}.
    \begin{itemize}
    \item[(a)] If for any subgroup $H\subset G$ the reduced homology $\wt{H}^*(||\cC''||^H)$ is trivial, then the map $\iota$ is an equivariant stable homotopy equivalence.
    \item[(b)] If for any subgroup $H\subset G$ the reduced homology $\wt{H}^*(||\cC'||^H)$ is trivial, then the map $\kappa$ is an equivariant stable homotopy equivalence.
    \item[(c)] If for any subgroup $H\subset G$ the reduced homology $\wt{H}^*(||\cC||^H)$ is trivial, then the map $\rho$ is an equivariant stable homotopy equivalence.
    \end{itemize}
  \end{lemma}
  \begin{proof}
    We prove only part (a) since the proofs for the other two statements are analogous.
    Our assumptions imply that for any $H \subset G$,
    \[\iota^{H} \colon ||\cC'||^{H} \to ||\cC||^{H}\]
    is a stable homotopy equivalence.
    Since $||\cC'||$ and $||\cC||$ are equivariantly homotopy equivalent to $G$-CW-complexes by Proposition~\ref{prop:repG_to_unrep}, the equivariant version of the Whitehead Theorem (see e.g.\ \cite[Section VI.3]{May}) implies that $\iota$ is an equivariant stable homotopy equivalence.
  \end{proof}

  \section{Khovanov homotopy type}\label{sec:review}
  In this section we introduce the equivariant Khovanov homotopy type.
  We start with a short recollection of the construction of the Khovanov chain complex and the annular Khovanov chain complex (Subsections~\ref{sub:review_khovanov} and~\ref{sub:annul-khov-chain}). 
  Next, we give a rather  brief review of the construction of the Khovanov homotopy type and the annular Khovanov homotopy type (Subsection~\ref{sub:khomotopy_type} and~\ref{sub:annul-khov-homot-1}).
  Finally, in Subsection~\ref{sub:equiv_flow_cat} we construct the equivariant Khovanov flow category.
  The results from Section~\ref{sec:equivariant} lead immediately to the construction of the equivariant Khovanov homotopy type and the equivariant annular Khovanov homotopy type.
  Invariance of the equivariant homotopy types is proved in Section~\ref{sec:main1}.

  \subsection{Khovanov chain complex}\label{sub:review_khovanov}

  In this subsection we rely on \cite[Section 2]{LS_stable}.
  Let $V$ be a two-dimensional vector space over a field $\F$ with $\xp$ and $\xm$ as generators.
  We make it a graded space by assigning a grading $q(\xp)=1$, $q(\xm)=-1$, called the \emph{quantum grading}.

  A \emph{resolution configuration} $\cD$ is a pair $(\tZ(\cD),\tA(\cD))$ where $\tZ(\cD)$ is a set of pairwise disjoint embedded circles in $S^{2}$
  and $\tA(\cD)$ is a totally-ordered set consisting of disjoint embedded arcs in $S^{2}$ such that the boundary of every arc lies in $\tZ(\cD)$.
  The \emph{index} of the resolution configuration $\cD$, denoted $\operatorname{ind}(\cD)$, is the cardinality of $\tA(\cD)$.
  A \emph{labeled resolution configuration} is a pair $(\cD,\xt)$ consisting of a resolution configuration $\cD$ and a map $\xt$ assigning a label, $\xp$ or $\xm$, to each element of $\tZ(\cD)$.

  Given two resolution configurations $\cD_{1}$ and $\cD_{2}$, we define the resolution configuration $\cD_{1} \setminus \cD_{2}$ by declaring, see~\cite{LS_stable}:\mpar{(39)}
  \[\tZ(\cD_{1} \setminus \cD_{2}) = \tZ(\cD_{1}) \setminus \tZ(\cD_{2}), \quad
    \tA(\cD_{1}\setminus\cD_{2}) = \{A \in \tA(\cD_{1}) \colon \forall_{Z \in \tZ(\cD_{2})} \partial A \cap Z = \emptyset\}.\]
  For a resolution configuration $\cD$ we can choose a subset $B \subset \tA(\cD)$ and obtain a new resolution configuration $\fs_{B}(\cD)$, called the \emph{surgery of} $\cD$ \emph{along B}, by performing a surgery along the arcs in $B$. We use a shortened notation $s(\cD)$ for the surgery $s_{\tA(\cD)}(\cD)$.
  Another operation that we can perform on a resolution configuration $\cD$ is taking the \emph{dual resolution configuration} $\cD^{\ast}$: $\tZ(\cD^{\ast}) = \tZ(\fs(\cD))$ and $A(\cD^{\ast})$ consists of arcs dual to arcs from $A(\cD)$, as explained in Figure~\ref{fig:dual-arcs}.

  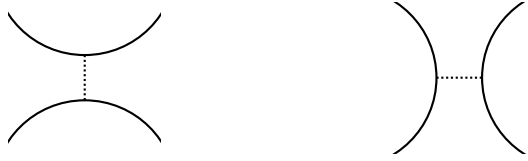
\begin{figure}
    \centering
    \begin{subfigure}[b]{.3 \linewidth}
      \centering
      \begin{tikzpicture}
        \begin{scope}
          \clip (-3,-1) rectangle (-5,1);
          \draw[thick] (-4,1.5) circle (1.2);
          \draw[thick] (-4,-1.5) circle (1.2);
          \draw[thick,densely dotted] (-4,0.3) -- (-4,-0.3);
        \end{scope}
      \end{tikzpicture}
    \end{subfigure}
    \begin{subfigure}[b]{.3 \linewidth}
      \centering
      \begin{tikzpicture}
        \begin{scope}
          \clip (-1,-1) rectangle (1,1);
          \draw[thick] (-1.5,0) circle (1.2);
          \draw[thick] (1.5,0) circle (1.2);
          \draw[thick,densely dotted] (-0.3,0) -- (0.3,0);
        \end{scope}
      \end{tikzpicture}
    \end{subfigure}
    \caption{Dual arcs.}
    \label{fig:dual-arcs}
  \end{figure}

  We can define a partial ordering on the set of labeled resolution configurations.
  Let $(\cD,\xt)$ and $(\cD',\yt)$ be two resolution configurations such that $\operatorname{ind}(\cD) - \operatorname{ind}(\cD')=1$.
  We say that $(\cD,\xt) \prec (\cD',\yt)$ if $\cD'$ can be obtained from $\cD$ by surgery along a single arc in $A \in \tA(\cD)$ and one of the following conditions holds:
  \begin{enumerate}
  \item If $\partial A$ is on a single circle $Z$ which splits during the surgery into two circles $Z_{1}$ and $Z_{2}$, then
    \begin{itemize}
    \item If $\xt(Z) = \xp$ then either $\yt(Z_{1})=\xp$ and $\yt(Z_{2})=\xm$ or $\yt(Z_{1})=\xm$ and $\yt(Z_{2})=\xp$.
    \item If $\xt(Z) = \xm$ then $\yt(Z_{1})=\xm$ and $\yt(Z_{2})=\xm$.
    \end{itemize}
  \item If $\partial A$ lies on two circles $Z_{1}$ and $Z_{2}$ which are merged during the surgery into a single circle $Z$, then
    \begin{itemize}
    \item If $\xt(Z_{1})=\xt(Z_{2})=\xp$, then $\yt(Z)=\xp$.
    \item If $\xt(Z_{1})=\xp$ and $\xt(Z_{2})=\xm$ or $\xt(Z_{1})=\xm$ and $\xt(Z_{2})=\xp$, then $\yt(Z)=\xm$.
    \end{itemize}
  \end{enumerate}
  For general labeled resolution configurations the partial order is defined as the transitive closure of the above relation.

  \begin{definition}\label{def:decorated}
    A \emph{decorated resolution configuration} is a triple $(\cD,\xt,\yt)$ where $(\cD,\yt)$ and $(\fs(\cD),\xt)$ are labeled resolution configurations such that $(\cD,\yt) \prec (\fs(\cD),\xt)$.

    Define $P(\cD,\xt,\yt)$ to be the poset consisting of all labeled resolution configurations $(\fs_{\tA}(\cD),\yt')$, where $\tA \subset \tA(\cD)$, such that
    \[(\cD,\yt) \preceq (\fs_{\tA}(\cD),\yt') \preceq (\fs(\cD),\xt).\]
  \end{definition}

  \begin{figure}

    \begin{tikzpicture}
      \draw[thick] (-1,-1) -- (1,1);
      \fill[draw=none,white] (0,0) circle (0.1);
      \draw[thick] (1,-1) -- (-1,1);
      \begin{scope}
        \clip (3,-1) rectangle (5,1);
        \draw[thick] (2.5,0) circle (1.2);
        \draw[thick] (5.5,0) circle (1.2);
      \end{scope}

      \begin{scope}
        \clip (-3,-1) rectangle (-5,1);
        \draw[thick] (-4,1.5) circle (1.2);
        \draw[thick] (-4,-1.5) circle (1.2);
        \draw[thick,densely dotted] (-4,0.3) -- (-4,-0.3);
      \end{scope}
      \draw[->,thick](1.5,0) -- node [scale=0.7,midway,above] {1-resolution} (3,0);
      \draw[->,thick](-1.5,0) -- node [scale=0.7,midway,above] {0-resolution} (-3,0);
    \end{tikzpicture}
    \caption{\small{Resolutions of a crossing.}}\label{fig:resolutions}
  \end{figure}
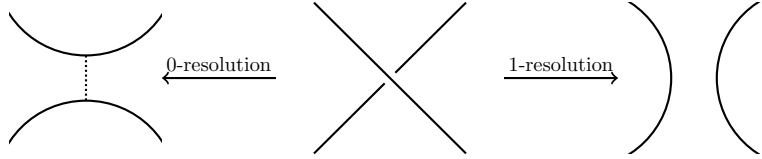

  Fix a field $\F$.
  Let $D$ be an oriented link diagram with $n = n_{+}+n_{-}$ ordered crossings, where $n_+$ and $n_-$ denote the number of positive, respectively negative, crossings.
  For every $v \in \{0,1\}^{n}$ we can define a resolution configuration $\cD_D(v) = (\tZ(\cD_D(v)), \tA(\cD_D(v)))$ obtained by smoothing the $i^{th}$ crossing of $D$ according to the $i^{th}$ coordinate of $v$, as depicted in Figure~\ref{fig:resolutions}. The arcs correspond to 0-resolutions; in Figure~\ref{fig:resolutions} the dotted line represents
  the arc associated to that particular resolution.
  To avoid cumbersome notation, we will drop the subscript $D$ when it is clear from the context.

  Let $V(\cD(v))$ be the vector space over $\F$ generated by all possible labeled resolution configurations $(\cD(v), x)$.
  Define the \emph{Khovanov complex} of $D$ in homological grading $i=|v|-n_-$ as
  \[\ckh^i(D;\F)=\bigoplus_{\substack{v\in\{0,1\}^n\\ |v|=i+n_-}} V(\cD(v)).\]
  The vector space $\ckh^i$ inherits the quantum grading from $V$. 
  To be more precise, to a homogeneous element $x=\newx_{\epsilon_1}\otimes \ldots \otimes \newx_{\epsilon_t}\in V(\cD(v))$, $\epsilon_i\in\{+,-\}$,
  we associate the grading $q(x)=\sum q(\newx_{\epsilon_i})+n_+-2n_-+|v|$.
  Then $\ckh^i$ splits as a direct sum of spaces $\ckh^{i,q}$, where the second index denotes the quantum grading.

  In order to make $\ckh^*$ into a cochain complex, we need to choose a sign assignment $\nu$.
  This done, we define the differential of an element $(\cD(v),\xt)$ with homological grading $i$ as
  \[\partial_{i}(\cD(v),\xt) = \sum_{\stackrel{|u|=|v|+1}{(\cD(v),\xt)\prec (\cD(u),\yt)}}(-1)^{\nu(u,v)} (\cD(u),\yt).\]
  The cohomology groups of the complex $\ckh$, that is $\kh^{i,q}(D) = \ker(\partial_i)/\im(\partial_{i-1})$, are link invariants~\cite{khovanov_categorification_2000}.
  They are the \textit{Khovanov homology groups} of the link represented by $D$.

  \subsection{Annular Khovanov chain complex}\label{sub:annul-khov-chain}

  Asaeda, Przytycki, and Sikora~\cite{AsaedaPrzytyckiSikora} gave a construction of a variant of Khovanov homology for a link in the solid torus.
  This construction was later refined by Roberts~\cite{Roberts}.
  Given $L \subset S^{1} \times \R^{2}$, we fix a diagram $D$ of $L$ so that $D$ can be drawn on an annulus $S^{1} \times \R^{1}$.
  The starting point of the construction of Asaeda, Przytycki, and Sikora is to assign an extra annular grading to each of the generators of the Khovanov complex of $D$.
  For any $v \in \{0,1\}^{n}$ and any labeled resolution configuration $\cD = (\cD(v),\xt)$, the annular grading of $\cD$, denoted $\Ann(\cD)$,  is defined in the following way.
  Let $\tZ(\cD(v)) = \{Z_{1},Z_{2},\ldots,Z_{k}\}$.
  We say that the circle $Z_{i}$, for $1 \leq i \leq k$ is \emph{trivial}, if it is null-homotopic in $S^{1} \times D^{1}$, and \emph{nontrivial} otherwise.
  For any $1 \leq i \leq k$ we define
  \[\Ann(Z_{i},\xt) =
    \begin{cases}
      1,  & \text{if } Z_{i} \text{ is nontrivial and } \xt(Z_{i}) = \xp, \\
      -1, & \text{if } Z_{i} \text{ is nontrivial and } \xt(Z_{i}) = \xm, \\
      0,  & \text{if } Z_{i} \text{ is trivial}, \\
    \end{cases}
  \]
  and
  \[\Ann(\cD,\xt) = \sum_{i=1}^{k} \Ann(Z_{i},\xt).\]
  It is easy to check that for any decorated resolution configuration $(\cD,\xt,\yt)$ we have
  \[\Ann(\cD,\yt) \geq \Ann(\fs(\cD),\xt).\]
  Indeed, the above property is trivial for decorated resolution configurations of index one.
  The general case follows from the transitivity of the relation $\succ$.
  Consequently, there exists a filtration of the Khovanov complex
  \[0 \subset \ldots \subset \CA_{k-1}(D) \subset \CA_{k}(D) \subset \CA_{k+1}(D) \subset \ldots \subset \ckh^{\ast,\ast}(D),\]
  where $\CA_{k}(D)$ is the subcomplex of $\ckh^{\ast,\ast}(D)$ generated by those labeled resolution configurations $(\cD(v),\xt)$ such that $\Ann(\cD(x),\xt) \leq k$.
  The \emph{annular Khovanov complex} of $D$ is the triply-graded cochain complex defined as
  \[\cakh^{i,j,k}(D) = \left(\ckh^{i,j}(D) \cap \CA_{k}(D)\right) / \left(\ckh^{i,j}(D) \cap \CA_{k-1}(D)\right).\]
  In this setting, the \emph{annular Khovanov homology} of $D$, denoted by $\akh^{\ast,\ast,\ast}(D)$, is defined as the homology of $\cakh^{\ast,\ast,\ast}(D)$. Annular
  Khovanov homology is an invariant of an annular link.

  \subsection{Khovanov homotopy type}\label{sub:khomotopy_type}

  In this subsection we apply constructions described in Sections~\ref{sec:flow_cat} and~\ref{sec:equivariant} to a specific flow category, the \emph{Khovanov flow category}, which is at the heart of the Lipshitz-Sarkar construction.
  Let $D$ be an oriented link diagram with $n = n_{+} + n_{-}$ ordered crossings.
  The starting point of the construction is to assign to every decorated resolution configuration $(\cD(v),\xt,\yt)$ the moduli space $\cM_{\kh}(\cD(v),\xt,\yt)$, which is a disjoint union of permutohedra $\Pi_{{m-1}}$, with $m = \operatorname{ind}(\cD(v))$.
  If $m=1$, $\cM_{\kh}(\cD(v),\xt,\yt)$ consists of a single point.
  If $m=2$, the moduli space $\cM_{\kh}(\cD(v),\xt,\yt)$ can be defined once we choose another piece of data called the \emph{ladybug matching} (for details refer to~\cite[Section~5.1]{LS_stable}).
  For $m>2$ the moduli spaces $\cM_{\kh}(\cD(v),\xt,\yt)$ can be constructed inductively.

  \begin{definition}[\cite{LS_stable}]\label{def:LS_flow}
    The \emph{Khovanov flow category}, $\cC_{\kh}(D)$, is a cubical flow category such that:
    \begin{itemize}
    \item $\Ob(\cC_{\kh}(D))$ consists of all labeled resolution configurations $(\cD(v),\xt)$, where $v \in \{0,1\}^{n}$.
      The grading of an object is equal to its homological grading $i(\cD(v),\xt) = |v|-n_{-}$ (recall that each object has an additional quantum grading, as explained in Subsection~\ref{sub:review_khovanov}).
    \item The morphism space is defined in the following way
      \[
        \cM_{\cC_{\kh}}((\cD(v), \xt), (\cD(u), \yt)) =
        \begin{cases}
          \cM_{\kh}(\cD(u)\setminus \cD(v),\xt',\yt'), & \text{if } (\cD(u), \yt) \prec (\cD(v), \xt), \\
          \emptyset, & \text{otherwise},
        \end{cases}
      \]
      where $\xt’$ and $\yt’$ are the restrictions of $\xt$ and $\yt$, respectively, to $\cD(u)\setminus \cD(v)$ and $\fs(\cD(u)\setminus \cD(v))$.
    \item The functor $\ff \colon \Sigma^{n_{-}}\cC_{\kh}(D) \to \Cube(n)$ maps a labeled resolution configuration $(\cD(v),\xt)$ to $v$.
    \end{itemize}
  \end{definition}
  \begin{remark}
    It is worth to stress that while $\cM_{\kh}$ denotes the moduli space associated with a pair of configurations $(\cD(v),\xt,\yt)$, the morphism space for the Khovanov flow category $\cC_{\kh}$ is denoted by $\cM_{\cC_{\kh}}$.
  \end{remark}

  By definition $\mathcal{M}_{\kh}(\cD(v),\xt,\yt) = \emptyset$ unless the $q$-gradings of $(\cD,\yt)$ and $(\fs(\cD),\xt)$ are equal.
  Consequently, for any $j \in \Z$, we can distinguish the full subcategory $\mathcal{C}_{\kh}^{j}(D)$ of $\mathcal{C}_{\kh}(D)$ consisting of objects whose $q$-grading is equal to $j$.
  It is easy to see that
  \begin{equation}\label{eq:q-grading-decomposition-flow-cat}
    \mathcal{C}_{\kh}(D) = \bigsqcup_{j \in \Z} \mathcal{C}_{\kh}^{j}(D).
  \end{equation}

  In this setting, and after making some choices (such as a framing and a neat embedding of $\cC_{\kh}(D)$), we obtain the cubical realization of the Khovanov flow category $||\cC_{\kh}(D)||$.
  This CW-complex is called the~\textit{Khovanov space}.
  The Khovanov homology of $D$, as constructed in Subsection~\ref{sub:review_khovanov}, is canonically isomorphic with the reduced cohomology of $||\cC_{\kh}(D)||$, up to grading shift.
  Finally, the stable homotopy type of the formal desuspension of the Khovanov space $\Sigma^{-n_-}||\cC_{\kh}(D)||$ is the \emph{Khovanov homotopy type} $\mathcal{X}_{\kh}(D)$ constructed in~\cite{LS_stable,LLS_long}, where it was proven to be a link invariant.

  Notice that the decomposition~\eqref{eq:q-grading-decomposition-flow-cat} induces a decomposition
  \[\mathcal{X}_{\kh}(D) = \bigvee_{j\in\Z} \mathcal{X}_{\kh}^{j}(D),\]
  where $\mathcal{X}_{\kh}^{j}(D)$ denotes the cubical realization of $\mathcal{C}_{\kh}^{j}(D)$.

  \subsection{Annular Khovanov homotopy type}\label{sub:annul-khov-homot-1}

  Recall that any labeled resolution configuration $\cD$ of a link $L \subset S^{1} \times D^{2}$, has an associated annular grading $\Ann(\cD)$.
  Define the \emph{annular Khovanov flow category} $\cC_{\akh}(D)$ to be the subcategory of $\cC_{\kh}(D)$ with the same set of objects but with morphisms preserving the annular grading.

  For $k \in \Z$, define the subcategories $\cC_{\kh}^{k}(D)$, $\cC_{\kh}^{\geq k}(D)$ and $\cC_{\kh}^{\leq k}(D)$ of the flow category $\cC_{\kh}(D)$ as the categories consisting of all labeled resolution configurations $\cD$ such that $\Ann(\cD) = k$, $\Ann(\cD) \geq k$ and $\Ann(\cD) \leq k$, respectively.
  Then, $\cC_{\akh}(D) = \bigsqcup_{k\in\Z} \cC_{\kh}^{k}(D)$.

  For $x,y\in\ob(\cC_{\kh})$ we have $\cM_{\cC_{\akh}}(x,y)=\emptyset$ unless $\Ann(x)=\Ann(y)$. In the latter case $\cM_{\cC_{\akh}}(x,y)=\cM_{\cC_{\kh}}(x,y)$. For a labeled resolution
  configuration $(\cD,\xt,\yt)$ we also denote by $\cM_{\akh}(\cD,\xt,\yt)$ the moduli space $\cM_{\cC_{\akh}}((\fs(\cD),\xt),(\cD,\yt))$.

  \begin{lemma}\label{lem:annular-flow-cat}
    The categories $\cC_{\akh}(D)$, $\cC_{\kh}^{\geq k}(D)$, $\cC_{\kh}^{\leq k}$ and $\cC_{\kh}^{k}$ are cubical flow categories.
  \end{lemma}
  \begin{proof}
    The cubical functor $\ff\colon \cC_{\kh}(D)\to\Cube(n)$ restricts to a cubical functor on each of these categories. Verifying the axioms of the cubical flow category is straightforward.
  \end{proof}
  Given Lemma~\ref{lem:annular-flow-cat} we can define $||\cC_{\akh}||$, $||\cC_{\kh}^{\geq k}||$, $||\cC_{\kh}^{\leq k}||$ and $||\cC_{\kh}^{k}||$ as a cubical realization of the
  suitable categories and then the corresponding desuspensions
  $\cX(\cC_{\akh})$, $\cX(\cC_{\kh}^{\geq k})$, $\cX(\cC_{\kh}^{\leq k})$ and $\cX(\cC_{\kh}^{k})$.
  Notice that the decomposition of the annular flow category $\mathcal{C}_{\akh}(D)$ according to the quantum and annular gradings induces a decomposition
  \[\mathcal{X}_{\akh}(D) = \bigvee_{j,k \in \Z} \mathcal{X}_{\akh}^{q,k}(D),\]
  where $\mathcal{X}_{\akh}^{q,k}(D)$ is the cubical realization of $\mathcal{C}_{\kh}^{q,k}(D)$.
  Repeating the proof of the invariance of the stable homotopy type of $\mathcal{X}_{\kh}(D)$ under Reidemeister moves, we obtain the invariance of the stable homotopy type of $\mathcal{X}_{\akh}(D)$ under Reidemeister moves in the solid torus.
  Therefore, the stable homotopy type $\mathcal{X}_{\akh}(D)$ is an invariant of an underlying link $L$.
  The following result relates the cohomology of the cubical realization $\cC_{\akh}$ with the annular Khovanov homology.
  \begin{lemma}
    For any $k \in \Z$, and any quantum grading $q\in\Z$, there exists an isomorphism
    \[H^{\ast}\left(\cX(\cC_{\akh}^{q,k}(D))\right) \cong \akh^{\ast,q,k}(D).\]
  \end{lemma}
  \begin{proof}
    By construction of the category $\cC_{\akh}$, the cochain complex associated with $\akh^{\ast,q,k}$ is precisely the cochain complex for $||\cC_{\akh}^{q,k}||$ up to a shift.
    We conclude by Lemma~\ref{lem:cochain_complex}.
  \end{proof}

  \subsection{Equivariant Khovanov flow category}\label{sub:equiv_flow_cat}
  Our goal is to construct a group action on the Khovanov flow category of a periodic link.
  Let $m$ be an integer.
  Let $D_m$ be a diagram of an $m$-periodic link and consider $G=\Z_m$, which acts effectively\mpar{(40)} on $\R^2$ by rotations, preserving the diagram $D_m$.
  The action of $G$ permutes the crossings of $D_m$.
  Let $\sigma$ be a permutation corresponding to a generator of $G$.
  We have $\sigma^m=id$.
  The following proposition shows how to extend the action of $G$ on crossings of $D_m$ to the action on the Khovanov flow category.

  \begin{figure}
    \centering
    \includegraphics[width = 12cm]{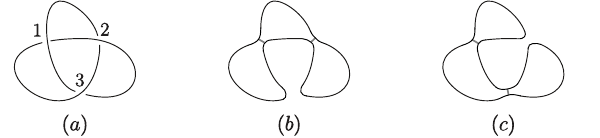}
    \caption{\small{A diagram $D$ of the trefoil knot with numbered crossings (a), the resolution $D(1,1,0)$, (b), and its image under the rotation of the group $\Z_3$, leading to the resolution $D(1,0,1)$, (c).}}\label{fig:rotation}
  \end{figure}

  \begin{proposition}\label{prop:extend}

    The action of $\Z_m$ on $D_m$ induces a group action on the Khovanov flow category $\cC_{\kh}(D_m)$.
    The assignment $\ff \colon (\cD(v),\xt)\mapsto |v|$ can be extended to an equivariant cubical functor $\ff \colon \Sigma^{\R[\Z_m]^{n_-/m}}\cC_{\kh}(D_m) \to \Cubes(n)$; in particular $\cC_{\kh}(D_m)$ is an equivariant cubical flow category.
  \end{proposition}
  \begin{proof}  The permutation $\sigma$ induces an action of $\Z_{m}$ on $\{0,1\}^{n}$ (we will denote this action by $\sigma$ as well).
    In order to define the action of $\Z_{m}$ on the set of objects of the Khovanov flow category, consider $(\cD(v),\xt)$, with $v \in \{0,1\}^{n}$, a labeled resolution configuration, and define $\cGs(\cD(v),\xt) = (\cD(\sigma(v)),x \circ \sigma^{-1})$; see Figure~\ref{fig:rotation}.
    Clearly, $\cGs$ induces an action of $\Z_m$ on the set of objects of $\cC_{\kh}(D_m)$.
    We need to describe the action on morphisms, that is, on the moduli spaces $\cM_{\kh}(\cD(v),\xt,\yt)$ (see Subsection~\ref{sub:khomotopy_type} for
    the definition of $\cM_{\kh}$ and $\cM_{\cC_{\kh}}$ and the relation between the two).

    Recall that the action of $\Z_{m}$ on the cube flow category was linear when restricted to any moduli space $\cM_{\Cube}(\ff(x),\ff(y))$.
    Therefore, the group action on morphisms in the cube category is completely determined by its restriction to the set of vertices of the respective permutohedra.
    This indicates that the group action on the moduli spaces $\cM_{\cC_{\kh}}(x,y)$ should be built inductively with respect to the dimension
    of the moduli spaces.

    We shall take care of axioms~\ref{item:EFC-1}--\ref{item:EFC-3}; 
    Items \ref{item:EFC-4}--\ref{item:EFC-7} are taken care by Lemma~\ref{lem:lift_to_cover}.\mpar{This is related to (26)}

    The construction of $\cGs$ is straightforward for index $1$ decorated configurations.
    Namely, $\cM_{\kh}(\cD(v),\xt,\yt)$  is a single point by construction (see the construction of $\cM_{\kh}(\cD(v),\xt,\yt)$ in~\cite[Section 5]{LS_stable}).
    This means that if $x,y\in\Ob(\cC_{\kh})$ have $\ind(y)=\ind(x)-1$ and $\cM_{\cC_{\kh}}(x,y)$ is non-empty, the functor $\ff$ induces a diffeomorphism between the moduli spaces $\cM_{\cC_{\kh}}(x,y)$ and $\cM_{\Cube(n)}(\ff(x),\ff(y))$, because each of them consists of a single point.
    Therefore $\cGs$ on zero-dimensional moduli space is uniquely determined by the action of $\cGs$ on the cube flow category $\Cube(n)$.

    We now pass to the construction of $\cGs$ for\mpar{(42) We reformulated slightly the construction, this also takes care of (41).}
    moduli spaces corresponding to index $k+1$ decorated configurations and $k\ge 1$. The construction is inductive. That is,
    in the construction we suppose $\cGs$ has already been constructed for all moduli spaces corresponding to resolution configurations of index $k$ or less.
    
    Consider an index $k+1$ decorated configuration $(\cD(u)\setminus \cD(v),\xt,\yt)$.
    Take the moduli space $\cM_{\mathcal{C}_{\kh}}(x,y)$ with $x=(\cD(v),\xt)$, $y=(\cD(u),\yt)$.

    Assume that $(\cD(u)\setminus\cD(v),\xt,\yt)$ is not a ladybug resolution configuration.
    By the inductive assumption the map $\cGs$ is already defined on the boundary of each connected component of $\cM_{\cC_{\kh}}(x,y)=\cM_{\kh}(\cD(v)\setminus\cD(u),\xt,\yt)$ and maps it onto the boundary of $\cM_{\cC_{\kh}}(\cGs x,\cGs y)=\cM_{\kh}(\cD(\sigma(v))\setminus \cD(\sigma(u)),x \circ \sigma^{-1},y \circ \sigma^{-1})$.
    Moreover, the following diagram is commutative
    \begin{center}
      \begin{tikzpicture}
        \matrix(m)[matrix of math nodes, row sep=1cm,column sep=1.5cm] {
          \partial\cM_{\cC_{\kh}}(x,y)                          & \partial \cM_{\cC_{\kh}}(\cGs x,\cGs y) \\
          \partial\cM_{\Cube(n)}(\ff(x),\ff(y)) & \partial\cM_{\Cube(n)}(\ff(\cGs x),\ff(\cGs y)). \\
        };
        \path[->,font=\scriptsize]
        (m-1-1) edge node[above] {$\cGs$} (m-1-2)
        (m-1-1) edge node[left] {$\ff_{x,y}$} (m-2-1)
        (m-1-2) edge node[right] {$\ff_{\cGs x,\cGs y}$} (m-2-2)
        (m-2-1) edge node[above] {$\cGs^{\Cube}$} (m-2-2);
      \end{tikzpicture}
    \end{center}
    The extension of $\cGs$ can be defined on $\cM_{\cC_{\kh}}(\cD(v),x,y)$ as follows.
    Let
    \begin{align*}
      \cM_{\cC_{\kh}}(x,y) &= Y_{1} \cup Y_{2} \cup \ldots \cup Y_{k}, \\
      \cM_{\cC_{\kh}}(\cGs x,\cGs y) &= Y_{1}' \cup Y_{2}' \cup \ldots \cup Y_{k}'
    \end{align*}
    denote the respective connected components.
    Without loss of generality, we may and will assume that $\cGs$ maps $\partial Y_{i}$ onto $\partial Y_{i}'$.
    For any $1 \leq i \leq k$, we define
    \begin{equation}\label{eq:extension-cubical-functor}
      \cGs|_{Y_{i}} = (\ff_{\sigma(\cD(v),\xt,\yt)}|_{Y_{i}'})^{-1} \circ \cGs^{\Cube} \circ \ff_{(\cD(v),\xt,\yt)}|_{Y_{i}}.
    \end{equation}
    The axioms~\ref{item:EFC-1} and~\ref{item:EFC-2} are trivially satisfied and~\ref{item:EFC-3} is guaranteed by the fact that the construction is performed inductively.

    \smallskip
    To complete the proof we need to consider the case, when $(\cD(u)\setminus\cD(v),\xt,\yt)$ is a ladybug configuration.
    The action of $\Z_{m}$ preserves the ladybug matching by~\cite[Lemma 5.8]{LS_stable}.
    Therefore, we again obtain a well-defined extension of $\ff$ to the whole $\cM_{\kh}(\cD(v),\xt,\yt)$ and the extension of $\cGs$ is given again by~\eqref{eq:extension-cubical-functor}.
      This completes the construction of the group action on the flow category $\cC_{\kh}(D_m)$.
      Conditions~\ref{item:EFC-1}--\ref{item:EFC-3} are trivially satisfied.

      \smallskip
    We define the grading via Lemma~\ref{lem:lift_to_cover}.
    As this is an important step of the construction, we unfold
    the definition of the grading. Namely,
    for an element $y=(\cD(v),\xt)\in\Ob(\cC_{\kh}(D_m))$ we define
    \begin{align}\label{eq:def_grading}
      \gr_G(y)&=\gr_{G_{\ff(y)}}(\ff(y))|_{G_y}-(\R[\Z_{m}]|_{G_{y}})^{n_-/m}= \nonumber\\
              &=(\R[G_{\ff(y)}]^{\gr(y)/|G_{\ff(y)}|})|_{G_y}-\R[G_y]^{n_-/|G_{y}|}=\\
              &=\R[G_y]^{(\gr(y)-n_-)/|G_y|}.\nonumber
    \end{align}
    In the last equality in~\eqref{eq:def_grading} we have used the fact that $\R[G]|_H=\R[H]^{|G|/|H|}$.

    With this definition, the functor $\ff\colon \Sigma^{\R[\Z_m]^{n_-/m}}\cC_{\kh}(D_m)\to\Cubes(m)$ preserves the grading.
  \end{proof}
  \begin{remark}
    We remark that shifting by $\R[\Z_m]^{n_-/m}$ is an overall shift corresponding to the grading shift by $n_-$ in the non-equivariant setting.
  \end{remark}

\begin{corollary}
  For a periodic link diagram $D_m$ in a solid torus, the category $\cC_{\akh}(D_m)$
  is a $\Z_m$-equivariant cubical flow category.
\end{corollary}  
\begin{proof}
    The action of $\Z_{m}$ on the Khovanov flow category $\cC_{\kh}(D_{m})$ constructed in Proposition~\ref{prop:extend} preserves the annular flow subcategory $\cC_{\akh}(D_{m})$
    inducing the desired structure.
\end{proof} 

  \section{Proof of Theorem~\ref{thm:main-theorem}}\label{sec:main1} 
  We will prove only part (b) of the theorem, namely that $\cX(\cC_{\kh})$
  is a well-defined object in the equivariant Spanier--Whitehead category. The case of annular Khovanov homology is completely analogous.

  Suppose that $D_m$ is an $m$-periodic diagram representing an $m$-periodic link $L$.
  By Proposition~\ref{prop:extend}, the Khovanov flow category $\cC_{\kh}(D_m)$ admits a group action.
  Proposition~\ref{prop:conclude} shows that the cubical realization $||\Sigma^{V}\cC_{\kh}(D_m)||$ admits a $\Z_m$-action, for an appropriate representation \(V\).
  In particular, it ensures the existence of the Khovanov homotopy type as an object in the equivariant Spanier-Whitehead category, see~Section~\ref{sub:terminology}.

  To conclude the proof of Theorem~\ref{thm:main-theorem}, we need to show that the equivariant stable homotopy type  $\cX(D_m)$ does not depend on the choices made.
  We prove invariance step by step.

  \begin{itemize}
  \item \emph{Independence of $R$ and $\epsilon$.}\
    Arguing as in~\cite[Lemma 3.25]{LS_stable}, we see that different choice of parameters $R$ and $\epsilon$ yields equivariantly homeomorphic spaces.
  \item \emph{Independence of $\ec$.}
    Any cubical neat embedding $\iota$ of $\cC$ relative to $\ec=(e_{1},\ldots,e_{n-1})$ induces a cubical neat embedding $\iota'$ relative to $\ec'=(e_{1},\ldots,e_{i}+1,\ldots,e_{n-1})$.
    Arguing as in the proof of~\cite[Lemma 3.26]{LS_stable} we conclude that
    \[\Sigma^{V}||\cC||_{\ec} \simeq ||\cC||_{\ec'}.\]
  \item \emph{Independence of $V$.}
    Let us introduce the following notation.
    Suppose $\iota_V$ is a cubical neat embedding relative to $\ec$ and relative to a representation $V$.
    Let $V\hookrightarrow W$ be an equivariant embedding. Composing this embedding with $\iota_V$ we obtain
    a neat embedding relative to $\ec$ and $W$, which we denote by $\iota_V^W$.
    We observe that if $W=V \oplus V'$, then by construction
    \[\Sigma^{V'}||\cC||_{\iota_V}\cong||\cC||_{\iota_V^W}.\]
    Suppose $\iota_V$ and $\iota_{V'}$ are two cubical neat embeddings relative to ${\ec}_V$ and $V$ and to ${\ec}_{V'}$ and $V'$, respectively.
    By increasing the entries of ${\ec}_V$ and ${\ec}_{V'}$ and using independence on $\ec$ discussed above, we may and will assume that ${\ec}_V={\ec}_{V'}=\ec$.
    We will also assume that the entries of $\ec$ are sufficiently large.

    Under the latter assumption, with $W=V\oplus V'$, the two embeddings $\iota_V^{W}$ and $\iota_{V'}^{W}$ a
    re equivariantly isotopic by the Mostow-Palais Theorem (Theorem~\ref{thm_mostowpalais}).
    By this we mean that for any $x,y\in\Ob(\cC)$, there exists an equivariant isotopy $\iota^t_{x,y}$ ($t\in[0,1]$) such that $\iota^0_{x,y}=(\iota_V^{W})_{x,y}$ and $\iota^1_{x,y}=(\iota_{V'}^{W})_{x,y}$ satisfying compatibility relations~\ref{item:EFC-1}--\ref{item:EFC-3} for all $t\in[0,1]$.
    Such isotopy is constructed by defining $j^t_{x,y}$, once $j^0_{x,y}$ and $j^1_{x,y}$ have been defined (see proof of Proposition~\ref{prop:equivariant_neat}).
    The construction of $j^t_{x,y}$ is inductive as in Proposition~\ref{prop:equivariant_neat}, using Mostow-Palais Theorem at each stage.
    We omit straightforward details.

    Given the isotopy, we obtain that $||\cC||_{\iota_V^{W}}$ and $||\cC||_{\iota_{V'}^{W}}$ are equivariantly homotopy equivalent, and therefore $||\cC||_{\iota_V}$ and $||\cC||_{\iota_{V'}}$ are equivariantly stably homotopy equivalent, as desired.
  \end{itemize} 
  Proving the independence on the choice of the diagram and on the ladybug matching is more complicated; we prove these results in Subsections~\ref{sub:ind1} and~\ref{sub:ind2}, respectively.

  \subsection{Independence under equivariant Reidemeister moves}\label{sub:ind1}

  Let $D_m^1$ and $D_m^2$ be periodic link diagrams representing the same periodic link $L$. Then $\cC(D_m^1)$ and $\cC(D_m^2)$ can be connected
  by a sequence of equivariant isotopies and equivariant Reidemeister moves. Here by an equivariant Reidemeister move we understand a $\Z_m$ orbit of a single Reidemeister move that is performed in a ball disjoint from the rotation axis. 
  See~\cite[Section 2.6, especially Figure~2.2]{StoffregenZhang}
  and \cite[Proposition 2.6]{Politarczyk_Khovanov} for a more detailed discussion. \mpar{(44)}

  \begin{proposition}\label{prop:isinvariant}
    The equivariant stable homotopy type of $\mathcal{X}_{\kh}(D_m)$ is invariant under the equivariant Reidemeister moves.
  \end{proposition}
  \begin{proof}
    We prove the invariance under the equivariant R2-move, following the same idea as in~\cite[Proof of Proposition 6.2]{LS_stable}.
    The proof of the invariance under other equivariant Reidemeister moves follows the same lines as in~\cite{LS_stable}; the necessary adjustments to make these proofs work in the equivariant case are the same as the adjustments for the proof of the invariance under equivariant R2-move, which we give in detail.
    For all equivariant Reidemeister moves, the main difficulty in the proof is to verify the assumptions of Lemma~\ref{lem:hardertocheck}.

    Let $D_m$ be a diagram with crossings $c_1,\ldots,c_n$.
    Let $D_m'$ be the diagram obtained after performing an equivariant R2-move on $D_m$, and let $(c_{n+1}, c_{n+2}), (c_{n+3},c_{n+4}), \ldots, (c_{n+2m-1}, c_{n+2m})$ be the $m$ pairs of new crossings created during the process; see Figure~\ref{fig:Chaincomplexes}a.

    \begin{figure}[t]
      \centering
      \includegraphics[width = 15cm]{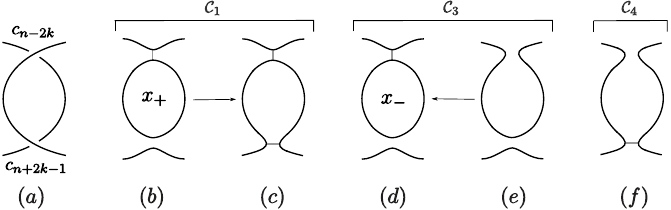}
      \caption{\small{The equivariant R2-move and the resolution configurations described in proof the independence of $\mathcal{X}_{\kh}(D_m)$ on Reidemeister moves.}}\label{fig:Chaincomplexes}
    \end{figure}

    For each $v \in \{0,1\}^{n+2m}$, write $\cD'(v) = \cD_{D_m'}(v)$, and recall that $v_i$, the $i^{th}$ coordinate of $v$, corresponds to the type of resolution of the crossing $c_i$ of $D_m'$.

    Consider now the subcategory of $\cC_{\kh}(D_m')$, denoted $\cC_1$, consisting of those labeled resolution configurations $(\cD'(v),\xt)$ such that:
    \begin{itemize}
    \item either there exists a value of $k$ with $v_{n+2k-1}= 0$ and $v_{n+2k}=1$, and $x$ assigns a label $\xp$ to the extra circle created (see Figure~\ref{fig:Chaincomplexes}b);
    \item or there exists a value of $k$ satisfying $v_{n+2k-1} = v_{n+2k} = 1$ (see Figure~\ref{fig:Chaincomplexes}c).
    \end{itemize}
    It is clear that $\cC_1$ is an upward closed category.
    It corresponds to the upward closed category $\mathscr{C}_1$ in~\cite[Proof of Proposition 6.3]{LS_stable}.

    \begin{lemma}\label{lem:c1acyclic}
      The subcategory $\cC_1$ is $G$-invariant, and for each subgroup $H \subset G$ the corresponding complex $C^*(\cC_1^H)$ is acyclic.
    \end{lemma}

    We defer the proof of Lemma~\ref{lem:c1acyclic} past the proof of Proposition~\ref{prop:isinvariant}.

    Let $\cC_2$ be the complementary downward closed subcategory of $\cC_1$, consisting of those labeled resolutions which do not satisfy any of the two previous conditions.
    Next, we consider a subcategory $\cC_3$ of $\cC_2$.
    The objects of $\cC_3$ are the labeled resolution configurations $(\cD'(v),\xt)$ such that there exists a value of $k$ satisfying:
    \begin{itemize}
    \item either $v_{n+2k-1}= 0$ and $v_{n+2k}=1$, and the extra circle is labeled by $\xm$;
    \item or $v_{n+2k-1}= 0$ and $v_{n+2k}=0$;
    \end{itemize}
    see Figure~\ref{fig:Chaincomplexes}(d) and (e).
    We observe that $\cC_3$ is an upward closed category.

    \begin{lemma}\label{lem:c3acyclic}
      The subcategory $\cC_3$ is $G$-invariant, and for any subgroup $H\subset G$ the complex $C^*(\cC_3^H)$ is acyclic.
    \end{lemma}

    We omit the proof of Lemma~\ref{lem:c3acyclic}, since it is analogous to the proof of Lemma~\ref{lem:c1acyclic}.

    Let $\cC_4$ be the complementary category of $\cC_3$ in $\cC_2$; that is to say, $\cC_4$ is the category such that there exists a value of $k$ satisfying $v_{n+2k-1}= 1$ and $v_{n+2k}=0$ (see Figure~\ref{fig:Chaincomplexes}(f)).
    Moreover, observe that $\cC_4$ is isomorphic to the category $\cC_{\kh}(D_m)$ corresponding to the original diagram $D_m$.

    In this setting, we apply Lemma~\ref{lem:hardertocheck} twice to get the desired result.
    Namely, we first state that $||\cC_{\kh}(D_m')||$ is equivariantly stably homotopy equivalent to $||\cC_2||$ and then that $||\cC_2||$ is equivariantly stably homotopy equivalent to $||\cC_4|| = ||\cC_{\kh}(D_m)||$.
    This concludes the proof of Proposition~\ref{prop:isinvariant}.
  \end{proof}

  \begin{proof}[Proof of Lemma~\ref{lem:c1acyclic}]

    \begin{figure}[t]
      \centering
      \includegraphics[width = 5.5cm]{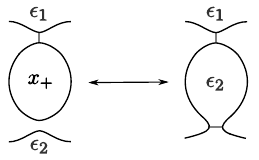}
      \caption{\small{The pairing illustrating the proof of Lemma~\ref{lem:c1acyclic} ($\epsilon_i \in \{\xp, \xm\}$, for $i=1,2$).}}\label{fig:pairing}
    \end{figure}

    For any $\mathbf{a} \in \{0,1\}^{m}$, let $C^{\ast}(\mathbf{a})$ denote the cochain complex generated by objects of $\cC_{1}$ of the form $v_{n+2k-1}=a_{i}$ and $v_{n+2k}=1$.
    Notice that $C^{\ast}(1,1,\ldots,1)$ is a subcomplex of $C^{\ast}(\cC_{1})$.  Moreover, the differential yields an isomorphism of chain complexes
    \[C^{\ast}(a_{1},\ldots,a_{i-1},0,a_{i+1},\ldots,a_{m}) \xrightarrow{\cong} C^{\ast}(a_{1},\ldots,a_{i-1},1,a_{i+1},\ldots,a_{m});\]
    see Figure~\ref{fig:pairing}.
    Therefore, there exists an isomorphism of chain complexes
    \[f \colon C^{\ast}(\cC_{1}) \cong C^{\ast}(1,\ldots,1) \otimes C^{\ast}(\Cube(m)).\]
    Since $C^{\ast}(\Cube(m))$ is acyclic, the K\"unneth Theorem implies that $C^{\ast}(\cC_{1})$ is also acyclic.

    An analogous argument works for the fixed point sets categories.
    Namely, let $H \subset G$ be a subgroup of order $k$ and consider $\cC_{1}^{H}$.
    Let $\sigma$ denote the permutation of crossings of $D_m'$.  Notice that the subset of crossings $\{c_{n+1},c_{n+2},\ldots,c_{n+2m}\}$ consists of two orbits of $G$
    \[c_{n+1},c_{n+3},c_{n+5},\ldots, c_{n+2m-1}\quad \text { and } \quad c_{n+2},c_{n+4},\ldots,c_{n+2m}.\]
    Moreover, if $\sigma(c_{n+2i-1}) = c_{n+2j-1}$ then $\sigma(c_{n+2i})=c_{n+2j}$, for any $1 \leq i \leq m$.
    Let $c_{i_{1}},c_{i_{2}},\ldots,c_{i_{m/k}}$, for $i_{1}<i_{2}< \cdots <i_{m/k}$, be representatives of orbits of the action generated by $\sigma^{m/k}$ restricted to the subset $\{c_{n+1},c_{n+3},c_{n+5},\ldots, c_{n+2m-1}\}$.
    If $(\cD'(v),\xt)$ belongs to $\cC_{1}$, then $v$ is constant on orbits of $\sigma^{m/k}$, and therefore values of $v$ on representatives of orbits determine its values on the whole subset of crossings $\{c_{n+1},\ldots,c_{n+2m}\}$.
    Let $\mathbf{a} \in \{0,1\}^{m/k}$.
    Define $C^{\ast}_{H}(\mathbf{a})$ to be the submodule of $C^{\ast}(\cC_{1}^{H})$ generated by labeled resolution configurations such that $v_{i_{j}} = a_{j}$, for $1 \leq j \leq m/k$. 
    Again, the components of the differential induce isomorphisms
    \[C^{\ast}_{H}(a_{1},a_{2},\ldots,a_{i-1},0,a_{i+1},\ldots,a_{m/k}) \xrightarrow{\cong} C^{\ast}_{H}(a_{1},a_{2},\ldots,a_{i-1},1,a_{i+1},\ldots,a_{m/k}),\]
    and therefore we obtain an isomorphism of chain complexes
    \[\psi \colon C^{\ast}(\cC_{1}^{H}) \xrightarrow{\cong} C^{\ast}(1,1,\ldots,1)^{H} \otimes C^{\ast}(\Cube(m/k)).\]
    We can apply once more the K\"unneth Theorem to conclude that $C^{\ast}(\cC_{1}^{H})$ is acyclic.
  \end{proof}

  \subsection{Independence of the ladybug matching}\label{sub:ind2}
  Before we give the proof of the independence of the ladybug matching, we introduce some notation: Given a link diagram $D_m$ together with a $\Z_m$-action by rotations, we write $\ol{D_m}$ for the link diagram with a $\Z_m$ action by rotations in the opposite direction.
  We write $\cC_{\kh}(\ol{D_m})$ for the corresponding equivariant Khovanov flow category; the underlying non-equivariant Khovanov flow category is the same, but the group action is inverted.

  \begin{proposition}\label{prop:ladybug}
    Let $D_m$ be a periodic diagram and let $\cC_{\kh}(D_m)$ \textup{(}$\cC_{\kh}^\#(D_m)$\textup{)} be its associated equivariant Khovanov flow category built using the right (respectively, left) ladybug matching.
    Then $||\cC_{\kh}(D_m)||$ and $||\cC_{\kh}^\#(D_m)||$ are stably equivariantly homotopy equivalent.
  \end{proposition}
  \begin{proof}
    The proof is essentially the argument of~\cite[Proposition 6.5]{LS_stable}, but there is one subtlety regarding the group action.

    Let $D_m'$ be the reflection of the diagram $D_m$ along the $y$-axis after switching all crossings (that is, $D_m'$ is the result of rotating $D_m$ by $180$ degrees).
    The diagrams $D_m$ and $D_m'$ represent the same link, but the rotation of the group action is inverted.
    In other words, $D_m$ and $\ol{D_m'}$ represent the same equivariant link.
    Therefore these diagrams can be related by a sequence of equivariant Reidemeister moves.
    By Proposition~\ref{prop:isinvariant} we have an equivariant stable homotopy equivalence between $||\cC_{\kh}(D_m)||$ and $||\cC_{\kh}(\ol{D_m'})||$.

    There is also an equivariant stable homotopy equivalence between $||\cC^\#_{\kh}(D_m)||$ and $||\cC_{\kh}(\ol{D_m'})||$.
    This is shown using the same argument as in the proof of~\cite[Proposition 6.5]{LS_stable}: the isomorphism of framed flow categories $\cC^\#_{\kh}(D_m)$ and $\cC_{\kh}(D_m)$ is equivariant, if we revert the group action on one side.

    The composition of the two equivariant stable homotopy equivalences yields the desired equivariant stable homotopy equivalence.
  \end{proof}

  \section{Moduli spaces via the $\Cbl$ category}\label{sec:cobcat}
  In order to prove the Categorical Fixed Point Theorem~\ref{thm:fixed-points} and, more generally, in order 
  to understand the fixed point set of $\cX_{\kh}(D)$ when $D$ is a periodic link diagram, we need a deeper understanding of the structure of moduli spaces $\cM_{\kh}$.
  The key tool is Bar-Natan's cobordism category $\Cbl$ reviewed in Subsection~\ref{sub:cobl}. The main result of this section, which is used in the proof of the fixed point theorem, is the Counting Moduli Lemma~\ref{lemma:conn-components-moduli-space},
  which computes the number of connected components of the moduli space $\cM_{\kh}$ in terms of the genus of the associated cobordism.

  \subsection{Dotted cobordism category of $\R^{3}$}\label{sub:cobl}
  As alluded to above, we begin with recalling Bar-Natan's formulation of Khovanov homology; see \cite{BarNatan}.
  Let $\Cob^{3}_{\bullet}$ denote the graded additive category whose objects set is generated by finite collections of disjoint simple closed curves, i.e.,\mpar{(45)}
  \begin{enumerate}
  \item If $Z \subset \R^{2}$ is a finite collection of pairwise disjoint simple closed curves, then $Z \in \Ob(\Cob^{3}_{\bullet})$,
  \item if $Z \in \Ob(\Cob^{3}_{\bullet})$, then a formal shift of $Z$, denoted $Z\{\ell\}$, for some $\ell \in \Z$, also belongs to $\Ob(\Cob^{3}_{\bullet})$,
  \item if $Z_{1}, Z_{2} \in \Ob(\Cob^{3}_{\bullet})$, then their formal direct sum $Z_{1} \oplus Z_{2}$ also belongs to $\Ob(\Cob^{3}_{\bullet})$.
  \end{enumerate}
  As $\Cob^3_{\bullet}$ is an additive category, it is enough to define morphisms on generators.\mpar{(46) It should be clear now, how are the morphisms defined on general objects, if not we can add more.}
  If $Z_{1},Z_{2}$ are two finite collections of disjoint simple closed curves in $\R^{2}$, elements of $\Hom_{\Cob^{3}_{\bullet}}(Z_{1},Z_{2})$ are represented by formal linear combinations of boundary preserving isotopy classes of dotted cobordisms
  \[\Sigma \subset \R^{2} \times [0,1], \quad \partial \Sigma = Z_{1} \sqcup -Z_{2}.\]
  For such a cobordism we define\mpar{(47)}
  \[\deg \Sigma = \chi(\Sigma) - 2\# \text{dots},\]
  where $\chi(\Sigma)$ denotes the Euler characteristic of $\Sigma$.
  Cobordisms are usually drawn from left to right or from bottom to top.
  Dots can move freely within the connected components of a given cobordism and decrease the degree of the respective map by $2$.

  The category $\Cbl$ is the quotient of $\Cob^{3}_{\bullet}$ by local relations depicted in Figure~\ref{fig:local-relations-cobordism-category}.
  Particularly useful is the \emph{neck cutting relation} depicted in Figure~\ref{fig:local-relations-cobordism-category-neck-cutting}.
  Indeed, a recursive application of the neck cutting relation quickly reduces any morphism in $\Cbl$ to a morphism given as a disjoint sum of unknotted surfaces of genus $0$ and $1$.

  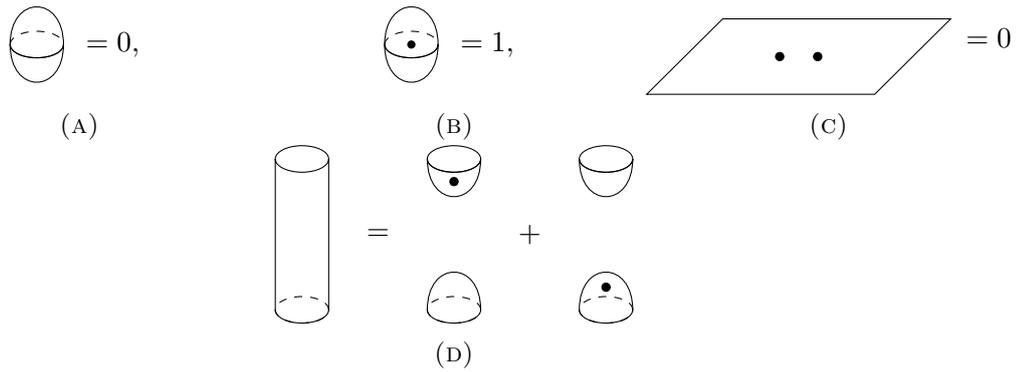
\begin{figure}
    \begin{subfigure}[t]{.3\linewidth}
      \centering
      \begin{tikzpicture}[tqft/view from=incoming]

        \pic[draw,tqft/cap,name=a,anchor=outgoing boundary, at={(0,0)}];
        \pic[draw,tqft/cup,name=b,every incoming lower boundary component/.style={draw,dashed},anchor=incoming boundary,at=(a-outgoing boundary)];
        \node[] at (1,0) {$=0,$};
      \end{tikzpicture}
      \caption{}
      \label{fig:local-relations-cobordism-category-undotted-sphere}
    \end{subfigure}
    \begin{subfigure}[t]{.3\linewidth}
      \centering
      \begin{tikzpicture}[tqft/view from=incoming]

        \pic[draw,tqft/cap,name=c,anchor=outgoing boundary, at={(0,0)}];
        \pic[draw,tqft/cup,name=d,every incoming lower boundary component/.style={draw,dashed},anchor=incoming boundary,at=(c-outgoing boundary)];
        \node[] at (1,0) {$=1,$};
        \node[circle,fill=black,at={(0,0)},scale=0.3] {};
      \end{tikzpicture}
      \caption{}
      \label{fig:local-relations-cobordism-category-dotted-sphere}
    \end{subfigure}
    \begin{subfigure}[t]{.3\linewidth}
      \centering{}
      \begin{tikzpicture}[tqft/view from=incoming]

        \draw (0,-0.75) -- (3,-0.75) -- (4,0.25) -- (1,0.25) -- cycle;
        \node[draw,at={(1.75,-0.25)},circle,fill=black,scale=0.3] {};
        \node[draw,circle,fill=black,scale=0.3,at={(2.25,-0.25)}] {};
        \node[] at (4.5,0) {$=0$};
      \end{tikzpicture}
      \caption{}
      \label{fig:local-relations-cobordism-category-double-dot}
    \end{subfigure}
    \begin{subfigure}[b]{.5\linewidth}
      \centering
      \begin{tikzpicture}[tqft/view from=incoming]

        \pic[draw,tqft/cylinder,at={(0,-1.5)}, every incoming lower boundary component/.style={draw}, every outgoing boundary component/.style={draw,dashed}];
        \node[] at (1,-2.5) {$=$};
        \pic[draw,tqft/cup,at={(2,-1.5)}, every incoming boundary component/.style={draw}];
        \node[draw,circle,fill=black,scale=0.3,at={(2,-1.8)}] {};
        \pic[draw,tqft/cap,at={(2,-1.5)}, every outgoing boundary component/.style={draw,dashed}];
        \node[] at (3,-2.5) {$+$};
        \pic[draw,tqft/cup,at={(4,-1.5)}, every incoming boundary component/.style={draw}];
        \pic[draw,tqft/cap,at={(4,-1.5)}, every outgoing boundary component/.style={draw,dashed}];
        \node[draw,circle,fill=black,scale=0.3,at={(4,-3.2)}] {};
      \end{tikzpicture}
      \caption{}
      \label{fig:local-relations-cobordism-category-neck-cutting}
    \end{subfigure}
    \caption{Local relations in $\Cbl$. The \emph{neck cutting relation} is depicted in Figure~\ref{fig:local-relations-cobordism-category-neck-cutting}.}
    \label{fig:local-relations-cobordism-category}
  \end{figure}

  \begin{lemma}[\cite{BarNatan}]\label{lemma:genus-reduction}
    Let $\Sigma$ be a dotted surface representing a morphism in $\Hom_{\Cbl}(Z_{1},Z_{2})$, for some $Z_{1},Z_{2} \in \ob(\Cbl)$.
    \begin{enumerate}
    \item If any connected component of $\Sigma$ is of genus greater than $1$, then
      \[\Sigma = 0 \in \Hom_{\Cbl}(Z_{1},Z_{2}).\]
    \item If any connected component of $\Sigma$ is a singly-dotted torus, then
      \[\Sigma = 0 \in \Hom_{\Cbl}(Z_{1},Z_{2}).\]
    \end{enumerate}
  \end{lemma}
    \begin{figure}
    \begin{center}
      \begin{tikzpicture}[tqft/view from=incoming];
        \pic[draw,tqft,genus=0,incoming boundary components=0,outgoing boundary components=2,every outgoing lower boundary component/.style={draw,dashed},name=a, at={(0,0)},scale=0.5];
        \pic[draw,tqft,genus=0,incoming boundary components=2,outgoing boundary components=1,every outgoing lower boundary component/.style={draw,dashed},anchor=incoming boundary 1,at=(a-outgoing boundary 1),offset=0.5,scale=0.5];
        \node[] at (2,-1) {$= 2$};
        \pic[draw,tqft,genus=0,incoming boundary components=0,outgoing boundary components=1,at={(3,0.5)},every outgoing lower boundary component/.style={draw,dashed}];
        \node[draw,circle,fill=black,scale=0.3,at={(3,-1.2)}] {};
      \end{tikzpicture}
    \end{center}
    \caption{The relation used in the proof of Lemma~\ref{lemma:genus-reduction}.}\label{fig:genus-reduction}
  \end{figure}
  \begin{proof}
    The lemma is a consequence of the relation drawn in Figure~\ref{fig:genus-reduction},
    which follows directly from the neck cutting relation shown in Figure~\ref{fig:local-relations-cobordism-category-neck-cutting} and the fact that a double dot annihilates every morphism, as shown in Figure~\ref{fig:local-relations-cobordism-category-double-dot}.
  \end{proof}

  \begin{figure}
    \begin{subfigure}[b]{.3\linewidth}
      \centering{}
      \begin{tikzpicture}[tqft/view from=incoming]
        \node[] at (0.75,0) {$\xp = $};
        \pic[draw, tqft/cap, at={(0,0)}, rotate=90, every outgoing lower boundary component/.style={draw,dashed}];
      \end{tikzpicture}
    \end{subfigure}
    \begin{subfigure}[b]{.3\linewidth}
      \centering{}
      \begin{tikzpicture}[tqft/view from=incoming]
        \node[] at (0.75,0) {$\xm = $};
        \pic[draw, tqft/cap, every outgoing lower boundary component/.style={draw,dashed}, at={(0,0)}, rotate=90];
        \node[draw,circle,fill=black,scale=0.3,at={(1.7,0)}] {};
      \end{tikzpicture}
    \end{subfigure}
    \caption{Caps generating $\Hom_{\Cbl}(\emptyset,S^{1})$.}
    \label{fig:morphisms-in-the-cobordism-category}
  \end{figure}

  \begin{figure}
    \begin{subfigure}[b]{.3\linewidth}
      \centering
      \begin{tikzpicture}[tqft/view from=incoming]
        \pic[draw,tqft/cap,every outgoing lower boundary component/.style={draw,dashed}, rotate=90,at={(0,0)}];
        \node[] at (2.3,0.2) {${}^{\ast}$};
        \node[] at (2.7,0) {$=$};
        \pic[draw,tqft/cup,every incoming lower boundary component/.style={draw},rotate=90,at={(3.2,0)}];
        \node[draw,circle,fill=black,scale=0.3,at={(3.5,0)}] {};      
      \end{tikzpicture}
    \end{subfigure}
    \begin{subfigure}[b]{.3\linewidth}
      \centering
      \begin{tikzpicture}[tqft/view from=incoming]
        \pic[draw,tqft/cap,every outgoing lower boundary component/.style={draw,dashed}, rotate=90,at={(0,-1.5)}];
        \node[draw,circle,fill=black,scale=0.3,at={(1.7,-1.5)}] {};
        \node[] at (2.3,-1.3) {${}^{\ast}$};
        \node[] at (2.7,-1.5) {$=$};
        \pic[draw,tqft/cup,every incoming lower boundary component/.style={draw},rotate=90,at={(3.2,-1.5)}];      
      \end{tikzpicture}
    \end{subfigure}
    \caption{Caps and their duals, i.e.\ cocaps (drawn on the right-hand side of the equality sign).}
    \label{fig:caps-and-cocaps}
  \end{figure}
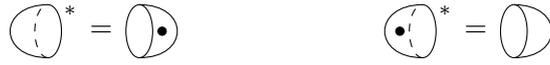

  \begin{lemma}[\cite{BarNatan}]\label{lemma:morphisms-as-disks}
    Morphisms in $\Cbl$ satisfy the following properties:
    \begin{enumerate}
    \item \label{item:morphisms-as-disks-1} A singly-dotted sphere generates
      \[\Hom_{\Cbl}(\emptyset,\emptyset) \cong \Z.\]
    \item \label{item:morphisms-as-disks-2} If $Z=S^{1}$, then
      \[\Hom_{\Cbl}(\emptyset,Z) \cong \Z^{2}\]
      is generated by dotted and undotted \emph{caps}, i.e.\ by surfaces depicted in Figure~\ref{fig:morphisms-in-the-cobordism-category}.
    \item \label{item:morphisms-as-disks-3} If $Z \in \ob(\Cbl)$ has $c$ connected components, then
      \[\Hom_{\Cbl}(\emptyset,Z) \cong \Hom_{\Cbl}(\emptyset,S^{1})^{\otimes c}\cong\Z^{2^c}\]
      is generated by surfaces consisting of $c$ disjoint dotted or undotted caps.
    \item \label{item:morphisms-as-disks-4} For any $Z \in \ob(\Cbl)$ the composition map
      \[\Hom_{\Cbl}(Z,\emptyset) \times \Hom_{\Cbl}(\emptyset,Z) \xrightarrow{\circ} \Hom_{\Cbl}(\emptyset,\emptyset)\]
      is a nonsingular bilinear pairing, hence there exists an isomorphism
      \[\Hom_{\Cbl}(Z,\emptyset) \cong \Hom_{\Cbl}(\emptyset,Z)^{\ast} = \Hom_{\Z}(\Hom_{\Cbl}(\emptyset,Z),\Z).\]
      Consequently, $\Hom_{\Cbl}(Z,\emptyset)$ is generated by $c$ disjoint surfaces dual to caps, i.e.\ \emph{cocaps} (see Figure~\ref{fig:caps-and-cocaps}).
    \item \label{item:morphisms-as-disks-5} If $Z_{1}, Z_{2} \in \ob(\Cbl)$ have $c_{1}$ and $c_{2}$ connected components, respectively, then
      \[\Hom_{\Cbl}(Z_{1},Z_{2}) \cong \Hom_{\Cbl}(\emptyset,Z_{2}) \otimes \Hom_{\Cbl}(Z_{1},\emptyset)\]
      is generated by a disjoint union of $c_{1}$ cocaps bounding $Z_{1}$ and $c_{2}$ caps bounding $Z_{2}$.
    \end{enumerate}
  \end{lemma}
  \begin{proof}
    Properties~(\ref{item:morphisms-as-disks-1}),~(\ref{item:morphisms-as-disks-2}) and~(\ref{item:morphisms-as-disks-3}) follow clearly from Lemma~\ref{lemma:genus-reduction} and the neck cutting relation (Figure~\ref{fig:local-relations-cobordism-category-neck-cutting}).
    Point~(\ref{item:morphisms-as-disks-4}) follows from points~(\ref{item:morphisms-as-disks-1}) and~(\ref{item:morphisms-as-disks-3}) together with the sphere and dotted sphere relations (Figures~\ref{fig:local-relations-cobordism-category-undotted-sphere} and \ref{fig:local-relations-cobordism-category-dotted-sphere}).
    The last point follows from points~(\ref{item:morphisms-as-disks-3}) and~(\ref{item:morphisms-as-disks-4}) together with the neck cutting relation.
  \end{proof}

  \begin{example}
    By Lemma~\ref{lemma:morphisms-as-disks}(\ref{item:morphisms-as-disks-4}), if $Z \in \ob(\Cbl)$ consists of $c$ connected components, then $\Hom_{\Cbl}(\emptyset,Z)$ is generated by $c$ disjoint dotted or undotted caps.
    For any field $\F$, we can identify
    \[\Hom_{\Cbl}(\emptyset,Z) \otimes \F \cong V^{\otimes c},\]
    where $V$ is the vector space used in Subsection~\ref{sub:review_khovanov}.
    In order to do that enumerate circles of $Z$ by $Z_1,\ldots,Z_c$.
    If $c=1$, the identification is given in Figure~\ref{fig:morphisms-in-the-cobordism-category}. The case $c>1$ is a simple extension.
    Indeed, if
    \[C = C_{1} \sqcup C_{2} \sqcup \ldots \sqcup C_{c}\]
    is a disjoint sum of caps such that $C_{i}$ bounds $Z_{i}$, for $1 \leq i \leq c$,
    then
    \[C \mapsto \underbrace{\xpm \otimes \xpm \otimes \ldots \otimes \xpm}_{c \text{ factors}} \in V^{\otimes c},\]
    where the sign of the $i$-th factor is $+$ if $C_{i}$ is undotted and $-$ otherwise, for $1 \leq i \leq c$.
  \end{example}

  For any $Z_{1},Z_{2} \in \ob(\Cbl)$ a \emph{distinguished generator} in $\Hom_{\Cbl}(Z_{1},Z_{2})$ is a disjoint union of cocaps and caps as described Lemma~\ref{lemma:morphisms-as-disks}(\ref{item:morphisms-as-disks-5}).
  The \emph{distinguished basis} of $\Hom_{\Cbl}(Z_{1},Z_{2})$ is the basis consisting of distinguished generators.
  If $\Sigma \in \Hom_{\Cbl}(Z_{1},Z_{2})$ is a distinguished generator, denote by $\Sigma^{\circ}$ and $\Sigma^{\bullet}$ the subsurface of $\Sigma$ consisting of undotted and dotted components, respectively.

  Given a diagram $D$ of a link $L \subset S^{3}$, Bar-Natan~\cite{BarNatan} constructed a cochain complex called the \emph{Bar-Natan} bracket of $D$, denoted $\bnbracket{D}$. 
  The Bar-Natan bracket is obtained by resolving crossings of $D$, as depicted in Figure~\ref{fig:resolutions}.
  Vogel~\cite{Vogel} later improved the construction of Bar-Natan in such a way that the chain homotopy type of $\bnbracket{D}$ does not depend on the choice of the diagram $D$ and is an invariant of $L$.

  \begin{lemma}[\cite{BarNatan}]\label{lem:barnatan}
    The functor $\cTKh \colon \Cbl \to \Z Mod$ given by
    \[\cTKh(Z) = \Hom_{\Cbl}(\emptyset,Z)\]
    is Khovanov's TQFT.
    In particular, given any field $\F$ and a link diagram $D$, there is a canonical isomorphism of cochain complexes\mpar{(48) we added the word `canonical'. The construction
    of the isomorphism is straightforward.}
    \[\mathcal{T}(\bnbracket{D})\otimes \F \cong \ckh^{\ast,\ast}(D).\]
  \end{lemma}

  Lemma~\ref{lem:barnatan} gives us a translation between labeled resolution configurations $(\cD,\xt)$ and distinguished generators of $\Hom_{\Cbl}(\emptyset,\tZ(\cD))$.

  \begin{remark}
    The isomorphism in Lemma~\ref{lemma:morphisms-as-disks}(\ref{item:morphisms-as-disks-5}) can be described in the following way.
    For two objects $Z_{1},Z_{2} \in \ob(\Cbl)$, the composition in $\Cbl$ induces a trilinear map
    \[\Theta_{Z_{1},Z_{2}} \colon \cTKh(Z_{2})^{\ast} \times \Hom_{\Cbl}(Z_{1},Z_{2}) \times \cTKh(Z_{1}) \to \Hom_{\Cbl}(\emptyset,\emptyset),\]
    which yields an isomorphism
    \[\Gamma_{Z_{1},Z_{2}}\colon \Hom_{\Cbl}(Z_{1},Z_{2}) \xrightarrow{\cong} \cTKh(Z_{2}) \otimes \cTKh(Z_{1})^{\ast}.\]
    If $\Sigma \in \Hom_{\Cbl}(Z_{1},Z_{2})$, then it is easy to check that
    \[\Gamma_{Z_{1},Z_{2}}(\Sigma) = \sum_{S_{1},S_{2}} \Theta_{Z_{1},Z_{2}}(S_{2}^{\ast},\Sigma,S_{1}) S_{2} \otimes S_{1}^{\ast} \in \cTKh(Z_{2}) \otimes \cTKh(Z_{1})^{\ast},\]
    where the summation extends over distinguished generators $S_{1}$ and $S_{2}$ of $\cTKh(Z_{1})$ and $\cTKh(Z_{2})$, respectively.
  \end{remark}

  \subsection{Counting Moduli Lemma}\label{sub:res_to_cob}

  Let $(\cD,\xt,\yt)$ be a decorated resolution configuration of index $n$. Enumerate the arcs in $\tA(\cD)$ by $A_1,\ldots,A_n$. Let $\fz = (z_{1},z_{2},\ldots,z_{n}) \in \Pi_{n-1}$.
  Define the surface $\Sigma(\fz,\cD)\subset\R^2\times[0,n+1]$ as a trace of the surgery on $\tZ(\cD)$, where the surgery on the $i$-th arc occurs at the level $z_i$. As the arcs
  are pairwise non-intersecting, the surface is well defined
  even if some of the coordinates $z_i$ of $\fz$ are equal.
  Notice that the map $\cTKh(\Sigma(u,\cD))$ does not depend on the choice of $\fz \in \Pi_{n-1}$. 

  For any surface $\Sigma\subset\R^2\times[a,b]$ we define its bottom boundary $\partial_0\Sigma$ and the top boundary $\partial_1\Sigma$ via
  \[\partial_0\Sigma=\Sigma\cap(\R^2\times\{a\}),\ \partial_1\Sigma=\Sigma\cap(\R^2\times\{b\}).\]

  The result we present next is the key tool in the study of the fixed points of moduli spaces.
  \begin{lemma}[Counting Moduli Lemma]\label{lemma:conn-components-moduli-space}
    Let $(\cD,\xt,\yt)$ be a decorated resolution configuration of index $n$.\mpar{(49)}
    Let $S_{1} \in \cTKh(\cD)$  be the distinguished generator corresponding to the surgery configuration $(\cD,\yt)$
    under the isomorphism from Lemma~\ref{lem:barnatan}. 
    Likewise, let $S_{2} \in \cTKh(\tZ(\fs(\cD)))$ be the distinguished generator corresponding to the labeled resolution configuration $(\fs(\cD),\xt)$
    Then, for any $z \in \Pi_{n-1}$,
    \[\Theta_{\tZ(\cD),\tZ(\fs(\cD))}(S_{2},\Sigma(\fz,\cD),S_{1}) = \# \pi_{0}(\cM_{\kh}(\cD,\xt,\yt)).\]
    In particular, if $\cM_{\kh}(\cD,\xt,\yt) \neq \emptyset$, then
    \[\cM_{\kh}(\cD,\xt,\yt) \cong \bigsqcup_{i=1}^{2^{c_{1}}} \Pi_{n-1},\]
    where $c_{1}$ is the number of genus $1$ connected components of $\Sigma(\fz,\cD)$.
  \end{lemma}
  \begin{remark}\label{rem:can_be_deduced}
    The lemma can be deduced from the discussion at the end of \cite[Section 2.11]{lawson_khovanov_2017}. As the precise statement is absent in that paper, and the result is widely used in the present paper, we present a sketch of the proof using $\Cbl$-categories and posets.\mpar{(50) We added this remark explaining the relation to \cite{lawson_khovanov_2017}.}
  \end{remark}
  \begin{proof}
    For a resolution configuration $(\fs_{\tA}(\cD),\xt')$, where $\tA \subset \tA(\cD)$, let $S_{\tA,\xt'} \in \cTKh(\tZ(\cD))$
    denote the distinguished generator corresponding to $(\fs_{\tA}(\cD),\xt')$.

    Without loss of generality we can pick $\fz = (1,2,3,\ldots,n) \in \Pi_{n-1}$.\mpar{(51)}
    Proposition~\ref{prop:chains_moduli} implies that $\# \pi_{0}(\cD,\xt,\yt) = \# \max P_{\fz}(\cD,\xt,\yt)$, so it is sufficient to prove that
    \begin{equation}\label{eq:maxPz}
      \# \max P_{\fz}(\cD,\xt,\yt) = \Theta_{\tZ(\cD),\tZ(\fs_{\tA}(\cD))}(S_{1},\Sigma(\fz,\cD),S_{2}).
    \end{equation}

    In order to prove \eqref{eq:maxPz}, we proceed by induction on the index of the resolution configuration.
    Let $Z_{01},\ldots,Z_{0\fs_y}$ be the circles in $\tZ(\cD)$ and $Z_{11},\ldots,Z_{1\fs_x}$ be the circles in $\tZ(\fs(\cD)))$.

    For an index $1$ resolution configuration, the poset $P_{\fz}(x,y)$ consists of a single chain $x\succ y$ and the surface $\Sigma$ has genus $0$.
    Then both sides of \eqref{eq:maxPz} are equal to $1$.

    Suppose now that \eqref{eq:maxPz} has been proved for all index $n-1$ resolution configurations, and let $(\cD,\xt,\yt)$ be a resolution configuration of index $n$.
    There are two cases. Either $A_1$ is a split, or it is a merge. We will deal only with the (harder) case, when $A_1$ is a split, leaving the other case to the reader.\mpar{(50): since
    the result can be deduced from \cite{lawson_khovanov_2017}, we give only the half of the proof.}

    Suppose $Z_{01}$ splits into two circles $Z_{011}$ and $Z_{012}$.
    If $\yt(Z_{01})=\xm$, then there is a unique $\yt_1$ such that $(\cD,\yt)\prec (\fs_{\{A_1\}}(\cD),\yt_1)$.
    We infer that $\#\max P_{\fz}(x,y)=\#\max P_{\fz^1}(x,y_1)$. The neck-cutting relation 
    shows that
    \begin{equation}\label{eq:cut_neck_and_go}
      \Theta(S_{2},\Sigma(\fz,\cD),S_{1})=\Theta(S_{2},\Sigma(\fz^1,\cD_1),S_{3}).
    \end{equation}

    Suppose finally that $\yt(Z_{01})=\xp$.
    Then there are two different assignments $\yt_{1}$ and $\yt_{2}$ such that for $y_{j}=(\fs_{\{A_1\}}(\cD),\yt_{j})$ we have $y\prec y_j$ (with $j=1,2$): one assigns $\xp$ to $Z_{011}$ and $\xm$ to $Z_{012}$, the other one does the opposite.
    In particular
    \[\max\# P_{\fz}(x,y)=\max\# P_{\fz^1}(x,y_1)+\max\# P_{\fz^1}(x,y_2).\]
    Let $S_3$, $S_4$ be the distinguished generators associated to $y_1$ and $y_2$, respectively.
    We need to prove that
    \begin{equation}\label{eq:cut_neck_and_go_for_beer}
      \Theta(S_{2},\Sigma(\fz,\cD),S_{1})=\Theta(S_{2},\Sigma(\fz^1,\cD_1),S_{3})+\Theta(S_{2},\Sigma(\fz^1,\cD_1),S_{4}),
    \end{equation}
    which follows from the neck-cutting relation (see Figure~\ref{fig:neck-cut-true}).
    \begin{figure}
      \input{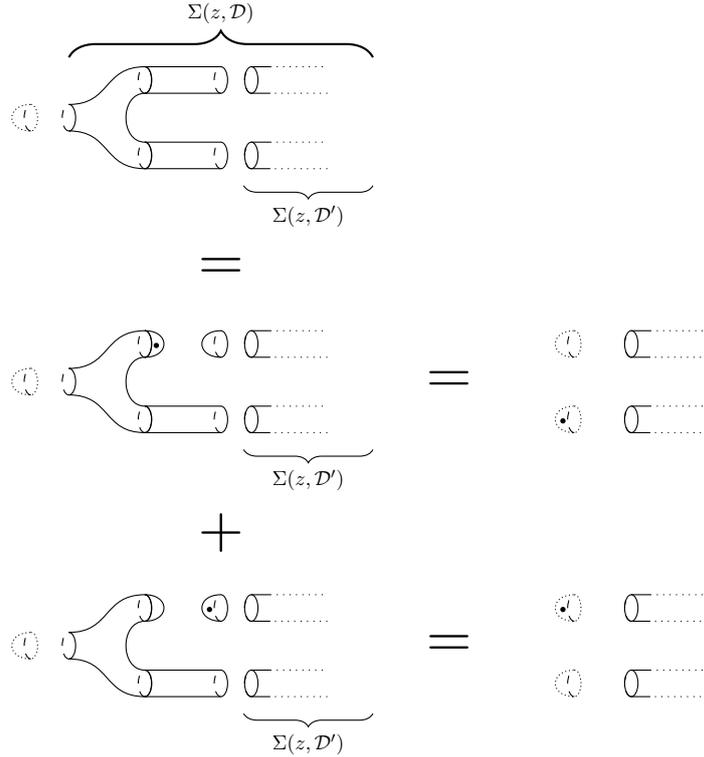}
      \caption{Proof of Lemma~\ref{lemma:conn-components-moduli-space}. The generator $S_1$ is a single circle without dots (corresponding to $\xp$). The new generators $S_3$, $S_4$ are two disks, one of them containing a dot.}\label{fig:neck-cut-true}
    \end{figure}
  \end{proof}

  \section{Categorical and Geometric Fixed Points}
  \label{sec:fixed-points}

  \subsection{Categorical Fixed Point Theorem}
  Throughout this section, $p$ denotes a fixed prime number.
  Define the map
  \[\pi_{p} \colon \bA \to \bA, \quad \pi_{p}(\zeta) = \zeta^{p},\]
  where $\zeta$ denotes the complex coordinate on $\mathbb{C} \setminus \{0\} = \bA$.
  If $\cD$ is a resolution configuration in $\bA$, then define the $p$-\emph{lift} of $\cD$ to be the $p$-periodic resolution configuration $\cD_{p}$ such that $\tZ(\cD_{p}) = \pi_{p}^{-1}(\tZ(\cD))$ and $\tA(\cD_{p}) = \pi_{p}^{-1}(\tA(\cD))$.\mpar{(55)}
  Analogously, for a labeled resolution configuration $(\cD,\xt)$ and a decorated resolution configuration $(\cD,\xt,\yt)$ we define the $p$-\emph{lift} $(\cD_{p},\xt_{p})$ and $(\cD_{p},\xt_{p},\yt_{p})$, where $\xt_{p} = \xt \circ \pi_{p}$ and $\yt_{p} = \yt \circ \pi_{p}$.
  \begin{theorem}[Categorical Fixed Point Theorem]\label{thm:fixed-points}
    Let $D_{p}$ be a $p$-periodic annular link diagram and let $D$ denote the quotient diagram.
    For any $q,k \in \Z$ there exists an isomorphism of cubical flow categories\mpar{(56)}
    \[\mathcal{C}_{\akh}^{q,k}(D) \xrightarrow{\pi_{p}} \mathcal{C}_{\akh}^{pq-(p-1)k,k}(D_{p})^{\Z_{p}},\]
    which induces the following isomorphism of cubical flow categories, for any $q \in \Z$,
    \[\bigsqcup_{\substack{q',k'\\ pq'-(p-1)k'=q}} \mathcal{C}_{\akh}^{q',k'}(D) \xrightarrow{\pi_{p}}  \mathcal{C}_{\kh}^{q}(D_{p})^{\Z_{p}}.\]
  \end{theorem}

  As a corollary we obtain the statement equivalent to Geometric Fixed Point Theorem~\ref{thm:geom_local}.
  \begin{corollary}\label{cor:cor_cor}
    For any annular link diagram $D$ we obtain 
    \[\mathcal{X}_{\akh}^{q,k}(D) = \mathcal{X}_{\akh}^{pq-(p-1)k,k}(D_{p})^{\Z_{p}}, \quad \quad \mathcal{X}_{\kh}^{q}(D_{p})^{\Z_{p}} = \bigvee_{\substack{q',k' \\ pq'-(p-1)k'=q}} \mathcal{X}_{\akh}^{q',k'}(D).\]
  \end{corollary}
  \begin{proof}[Proof of Corollary~\ref{cor:cor_cor}]
    Notice that by Proposition~\ref{prop:fixedcategory} and Theorem~\ref{thm:fixed-points} we have
    \[\mathcal{X}\left(\cC_{\akh}^{q,k}(D_{p})\right)^{\Z_{p}} = \mathcal{X}\left(\cC_{\akh}^{q,k}(D_{p})^{\Z_{p}}\right) = \mathcal{X}\left(\cC^{pq-(p-1)k,k}_{\akh}(D)\right).\]
    The case of the Khovanov flow category is analogous.
  \end{proof}
  \begin{proof}[Proof of Theorem~\ref{thm:fixed-points}]
    The desired isomorphism of cubical flow categories will be first defined on objects, then on morphisms.
    \begin{lemma}\label{lem:bijection-objects_a}
      The formula $\cF_p(\cD,\xt)=(\cD_p,\xt_p)$
      induces a bijection
      \[\cF_p\colon \ob(\cC_{\akh}(D)) \xrightarrow{\cong} \ob(\cC_{\akh}(D_{p}))^{\Z_{p}}.\]
      The map $\cF_p$ preserves the annular grading.
    \end{lemma}
    \begin{proof}[Proof of Lemma~\ref{lem:bijection-objects_a}]
      The inverse map is given by taking the quotient of a respective labeled resolution configuration.
      Moreover, invariance of the annular grading under the map \(\cF_{p}\) is evident.
    \end{proof}

    We will now pass to constructing the map on morphisms.
    The key property that we will require is that for all resolution configurations $(\cD,\xt,\yt)$ the following
    diagram commutes:
    \begin{equation}\label{eq:cube_commutes}
      \begin{tikzpicture}[baseline=-0.5ex]
        \matrix(m)[matrix of nodes,row sep=1cm, column sep=1.5cm] {
          $\mathcal{M}_{\akh}(\cD,\xt,\yt)$ & $\mathcal{M}_{\akh}(\cD_{p},\xt_{p},\yt_{p})^{\Z_{p}}$ \\
          $\mathcal{M}_{\Cube(n)}(u,v)$             & $\mathcal{M}_{\Cube_{\sigma}(pn)}(u_{p},v_{p})$ \\
        };
        \path[->,font=\scriptsize]
        (m-1-1) edge node[above] {$\cF_{p}$} (m-1-2)
        (m-1-1) edge node[left] {$\ff$} (m-2-1)
        (m-1-2) edge node[right] {$\ff$} (m-2-2)
        (m-2-1) edge node[above] {$\cFC$} (m-2-2);
      \end{tikzpicture}
    \end{equation}
    where $u = \ff(\cD,\yt)$, $v = \ff(\fs(\cD),\xt)$, $u_{p} = \ff(\cD_{p},\yt_{p})$ and $v_{p}=\ff(\fs(\cD_{p}),\xt_{p})$.

    We begin with morphisms in $\cC_{\akh}(D)$ corresponding
    to index one configurations. All these configurations are depicted in Figure~\ref{fig:index-one-res-configurations}.
    \begin{figure}
      \begin{subfigure}[b]{.4\linewidth}
        \centering
        \begin{tikzpicture}
          \node[draw,circle,fill=black,scale=0.3] at (0,0) {};
          \draw[thick] (0,0) circle [radius=.5cm];
          \draw[thick] (0,0) circle [radius=1cm];
          \draw[very thick,dashed] (0,.5) -- (0,1);
        \end{tikzpicture}
        \caption{}\label{fig:index-one-configuration-merge-1}
      \end{subfigure}
      \begin{subfigure}[b]{.4\linewidth}
        \centering
        \begin{tikzpicture}
          \path (5,0) coordinate (origin);
          \node[draw,circle,fill=black,scale=0.3] at (origin) {};
          \path (origin) ++(135:.5cm) coordinate (starting-pt);
          \path (origin) ++(63:.8cm) coordinate (res-arc);
          \draw[thick] (starting-pt) arc (135:405:.5cm) arc (225:45:.25cm) arc (45:-225:1cm) arc (135:-45:.25cm);
          \draw[very thick,dashed] (res-arc) arc (63:117:.8cm);
        \end{tikzpicture}
        \caption{}\label{fig:index-one-configuration-split-1}
      \end{subfigure}

      \begin{subfigure}[b]{.4\linewidth}
        \centering
        \begin{tikzpicture}
          \node[draw,circle,fill=black,scale=0.3] at (0,-2.5) {};
          \draw[thick] (0,-2.5) circle [radius=0.5];
          \draw[thick] (-1.5,-2.5) circle [radius=0.4];
          \draw[very thick,dashed] (-1.1,-2.5) -- (-0.5,-2.5);
        \end{tikzpicture}
        \caption{}\label{fig:index-one-configuration-merge-2}
      \end{subfigure}
      \begin{subfigure}[b]{.4\linewidth}
        \centering
        \begin{tikzpicture}
          \path (5,-2.5) coordinate (origin);
          \node[draw,circle,fill=black,scale=0.3] at (origin) {};
          \path (origin) ++(-135:.5cm) coordinate (starting-pt);
          \draw[thick] (starting-pt) arc (-135:135:.5cm) arc (-45:-135:.25cm) arc (45:315:.5cm) arc (135:45:.25cm);
          \path (origin) ++(-.53,.53) ++ (0,-.25) coordinate (P0);
          \path (origin) ++(-.53,-.53) ++ (0,.25) coordinate (P1);
          \draw[very thick,dashed] (P0) -- (P1);
        \end{tikzpicture}
        \caption{}\label{fig:index-one-configuration-split-2}
      \end{subfigure}

      \begin{subfigure}[b]{.4\linewidth}
        \centering
        \begin{tikzpicture}
          \node[draw,circle,fill=black,scale=0.3] at (0,-5) {};
          \draw[thick] (-1,-4.2) circle [radius=.5cm];
          \draw[thick] (1,-4.2) circle [radius=.5cm];
          \draw[very thick,dashed] (-.5,-4.2) -- (.5,-4.2);
        \end{tikzpicture}
        \caption{}\label{fig:index-one-configuration-merge-3}
      \end{subfigure}
      \begin{subfigure}[b]{.4\linewidth}
        \centering
        \begin{tikzpicture}
          \path (5,-5) coordinate (origin);
          \node[draw,circle,fill=black,scale=0.3] at (origin) {};
          \path (origin) ++(.53,0.8) coordinate (Q0);
          \path (Q0) ++(-135:.5cm) coordinate (starting-pt);
          \draw[thick] (starting-pt) arc (-135:135:.5cm) arc (-45:-135:.25cm) arc (45:315:.5cm) arc (135:45:.25cm);
          \path (Q0) ++(-.53,.53) ++ (0,-.25) coordinate (P0);
          \path (Q0) ++(-.53,-.53) ++ (0,.25) coordinate (P1);
          \draw[very thick,dashed] (P0) -- (P1);
        \end{tikzpicture}
        \caption{}\label{fig:index-one-configuration-split-3}
      \end{subfigure}
      \caption{Index one resolution configurations in $\R^{2} \setminus \{0\}$. The black dot is the fixed point set of the rotation.}
      \label{fig:index-one-res-configurations}
    \end{figure}
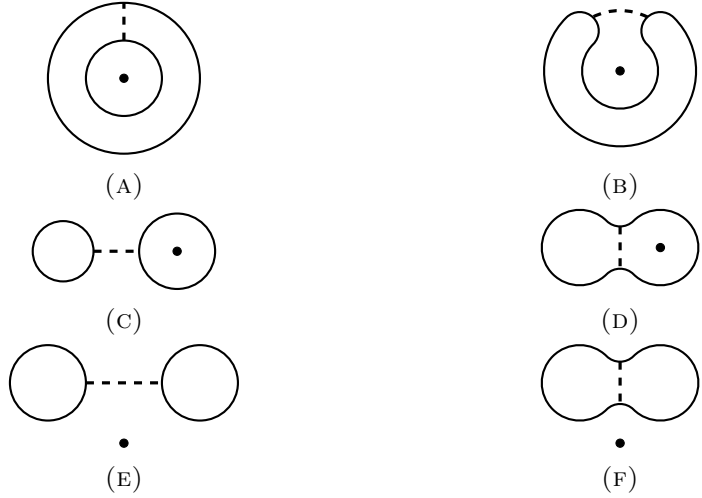
    We recall now a result of Zhang, which is proved on a detailed case-by-case analysis.
    \begin{lemma}[see \expandafter{\cite[Section 5.3]{Zhang}}]\label{lem:zhangs_lemma}
      For all configuration depicted in Figure~\ref{fig:index-one-res-configurations}
      we have
      \[\mathcal{M}_{\akh}(\cD,\xt,\yt) \neq \emptyset \iff \mathcal{M}_{\akh}(\cD_{p},\xt_{p},\yt_{p}) \neq \emptyset.\]
    \end{lemma}
    Next result connects $\cM_{\akh}$ with $\cM_{\kh}$ for index one configurations.
    \begin{lemma}\label{lem:close_to_zhang}
      Suppose $(\cD,\xt,\yt)$ is one of the configurations of Figure~\ref{fig:index-one-res-configurations}.
      If $\cM_{\kh}(\cD_p,\xt_p,\yt_p)^{\Z_p}$ is non-empty, $\Ann(\cD_p,\yt_p)=\Ann(\fs(\cD_p),\xt_p)$.
    \end{lemma}
    \begin{proof}[Proof of Lemma~\ref{lem:close_to_zhang}]
      The proof is done on a case-by-case analysis. Cases (e) and (f) of Figure~\ref{fig:index-one-res-configurations}
      are trivial, because no circles in $\cD_p$ or $\fs(\cD_p)$ is non-trivial (in the sense of Subsection~\ref{sub:annul-khov-chain}).
      Thus $\Ann(\cD_p,\yt_p)=\Ann(\fs(\cD_p),\xt_p)=0$.

      Case (b) is dual to (a), and case (d) is dual to (c), so it is enough to prove the lemma for cases (a) and (c) only.
      We will deal with case~(a) only, leaving case~(c) (which is easier) to the reader.
      The resolution configurations $\cD_p$ and $\fs(\cD_p)$ are depicted in Figure~\ref{fig:proof_close}.
      \begin{figure}
        \begin{tikzpicture}
          \begin{scope}[xshift=-2cm]
            \fill[black] (0,0) circle (0.08);
            \draw[thick] (0,0) circle (1);
            \draw (0,0.8) node[scale=0.7] {$Z_1$};
            \draw (0,2) node [scale=0.7] {$Z_2$};
            \draw[thick] (0,0) circle (1.8);
            \draw[thick, dashed] (0,0) ++ (0:1) --    ++(0:0.8);
            \draw[thick, dashed] (0,0) ++ (72:1) --  ++(72:0.8);
            \draw[thick, dashed] (0,0) ++ (144:1) -- ++(144:0.8);
            \draw[thick, dashed] (0,0) ++ (216:1) -- ++(216:0.8);
            \draw[thick, dashed] (0,0) ++ (288:1) -- ++(288:0.8);
            \draw(010:1.4) node [scale=0.7] {$A_1$};
            \draw(082:1.4) node [scale=0.7] {$A_2$};
            \draw(154:1.4) node [scale=0.7] {$A_3$};
            \draw(226:1.4) node [scale=0.7] {$A_4$};
            \draw(298:1.4) node [scale=0.7] {$A_5$};
          \end{scope}
          \begin{scope}[xshift=3cm]
            \fill[black] (0,0) circle (0.08);
            \draw (000:1.4) circle (0.6);
            \draw (072:1.4) circle (0.6);
            \draw (144:1.4) circle (0.6);
            \draw (216:1.4) circle (0.6);
            \draw (288:1.4) circle (0.6);
            \draw (000:2.2) node [scale=0.7] {$Z_3$} ;
            \draw (072:2.2) node [scale=0.7] {$Z_4$} ;
            \draw (144:2.2) node [scale=0.7] {$Z_5$} ;
            \draw (216:2.2) node [scale=0.7] {$Z_6$} ;
            \draw (288:2.2) node [scale=0.7] {$Z_7$} ;
          \end{scope}
        \end{tikzpicture}
        \caption{Proof of Lemma~\ref{lem:close_to_zhang}.}\label{fig:proof_close}
      \end{figure}
      Note that the annular grading of any resolution configuration on the right is zero, because there are no non-trivial circles.
      Therefore, it is enough to show that if the resolution configuration on the left has non-trivial annular grading, then $\cM_{\kh}(\cD_p,\xt_p,\yt_p)^{\Z_p}=0$.
      The configuration $(\cD_p,\yt_p)$ has non-trivial annular grading in precisely two cases: either $\yt_p$ assigns $\xm$ to both circles on the left, or it assigns $\xp$ to both circles.
      In the first case, as the surgery on any of the arcs merges two circles labeled with $\xm$, the moduli space $\cM_{\kh}(\cD_p,\xt_p,\yt_p)$ is empty.

      The other case is that $\yt_p$ assigns $\xp$ to both circles. After the surgery on one of the arcs connecting the two circles, we obtain a single circle labeled with $\xp$.
      All other $p-1$ arcs are splits. Any split of an $\xp$ labeled circle yields a circle labeled with $\xp$ and a circle labeled with $\xm$, while a split of a circle labeled with $\xm$ has two circles both labeled with $\xm$.
      It follows that $\xt_p$ assigns $\xp$ to a positive number of circles and $\xm$ also to a positive number of circles.
      Such configuration (the underlying $p$ circles are drawn in Figure~\ref{fig:proof_close} on the right) cannot be $\Z_p$-invariant. Hence $\cM_{\kh}(\cD_p,\xt_p,\yt_p)^{\Z_p}$ is empty.
      This concludes the proof of case~(a).
    \end{proof}
    \begin{corollary}\label{cor:index_one}
      For any index one configuration, there is a bijection between $\cM_{\akh}(\cD,\xt,\yt)$, $\cM_{\akh}(\cD_p,\xt_p,\yt_p)^{\Z_p}$ and
      $\cM_{\kh}(\cD_p,\xt_p,\yt_p)^{\Z_p}$.
    \end{corollary}
    \begin{proof}
      From Lemma~\ref{lem:close_to_zhang} we immediately obtain a bijection between $\cM_{\akh}(\cD_p,\xt_p,\yt_p)^{\Z_p}$ and $\cM_{\kh}(\cD_p,\xt_p,\yt_p)^{\Z_p}$.
      By Lemma~\ref{lem:zhangs_lemma}, it is enough to consider the case, when $\cM_{\akh}(\cD,\xt,\yt)$ is non-empty. Then it is a zero-dimensional
      permutohedron $\Pi_0$, that is, a single point. Call it $\fz$. Let $\Sigma(\fz,\cD)$ be the corresponding surface. It has genus zero.

      Let $\fz_p\in\Pi_{p-1}^{\Z_p}$ denote the unique fixed point of the $\Z_p$ action. The surface $\Sigma(\fz_p,\cD_p)$ is a $p$-fold
      cover of $\Sigma(\fz,\cD)$ and it is easily seen to have genus zero as well. From Counting Moduli Lemma~\ref{lemma:conn-components-moduli-space} we deduce
      that $\cM_{\akh}(\cD_p,\xt_p,\yt_p)$ is connected, hence it is diffeomorphic to $\Pi_{p-1}$. Therefore $\cM_{\akh}(\cD_p,\xt_p,\yt_p)^{\Z_p}$ is a single point.
    \end{proof}
    \begin{remark}
      Since the moduli spaces in Corollary~\ref{cor:index_one} are either empty or a single point, the bijection of Corollary~\ref{cor:index_one} is uniquely defined.
    \end{remark}
    Continuing the proof of Theorem~\ref{thm:fixed-points}, we extend $\cF_p$ from objects to morphisms corresponding to index one resolution configurations. We now discuss the index two resolution configurations.

    Assume first that 
    $(\cD,\xt,\yt)$ is not a ladybug and $\cM_{\akh}(\cD,\xt,\yt)$ is non-empty. Then $(\cD,\xt,\yt)$ is a genus zero resolution configuration.
    In particular,
    $\cM_{\akh}(\cD,\xt,\yt)=\Pi_1$ is an interval with two boundary components. By dimension counting argument, $\cM_{\akh}(\cD_p,\xt_p,\yt_p)^{\Z_p}$ is a union
    of some number of copies of one-dimensional permutohedra $\Pi_1$. By the induction assumption $\partial\cM_{\akh}(\cD_p,\xt_p,\yt_p)^{\Z_p}\stackrel{\cF_p}{\cong}\partial\cM_{\akh}(\cD,\xt,\yt)$. Hence $\cM_{\akh}(\cD_p,\xt_p,\yt_p)^{\Z_p}$ is also an interval. Then 
    $\ff$ takes it diffeomorphically to $\Pi_{2p-1}^{\Z_p}=\Pi_1$.
    We define $\cF_{p} = \ff^{-1} \circ \cFC \circ \ff$, so \eqref{eq:cube_commutes} commutes. 

    If the genus of $(\cD,\xt,\yt)$ is one, the moduli space is not connected.
    Decorated resolution configurations of index two and genus one are called \emph{ladybug configurations}, they are depicted in Figure~\ref{fig:ladybugs}.
    \begin{figure}
      \begin{subfigure}[b]{.3\linewidth}
        \centering
        \begin{tikzpicture}
          \path (0,0) coordinate (origin);
          \path (origin) ++(.25,0) coordinate (center);
          \node[draw,circle,fill=black,scale=0.3] at (origin) {};
          \draw[thick] (center) circle [radius=.5cm];
          \path (center) ++(0,.5) coordinate (P0);
          \path (center) ++(0,-.5) coordinate (P1);
          \draw[very thick, dashed] (P0) -- (P1);
          \path (center) ++(60:.5cm) coordinate (Q0);
          \path (center) ++(120:.5cm) coordinate (Q1);
          \path (Q0) ++(60:.75cm) coordinate (C0);
          \path (Q1) ++(120:.75cm) coordinate (C1);
          \draw[very thick, dashed] (Q0) ..controls (C0) and (C1).. (Q1);
        \end{tikzpicture}
        \caption{}\label{fig:ladubug-configuration-1}
      \end{subfigure}
      \begin{subfigure}[b]{.3\linewidth}
        \centering
        \begin{tikzpicture}
          \path (0.25,-1.25) coordinate (origin);
          \path (origin) ++(0,-.7) coordinate (center);
          \node[draw,circle,fill=black,scale=0.3] at (origin) {};
          \draw[thick] (center) circle [radius=.5cm];
          \path (center) ++(0,.5) coordinate (P0);
          \path (center) ++(0,-.5) coordinate (P1);
          \draw[very thick, dashed] (P0) -- (P1);
          \path (center) ++(60:.5cm) coordinate (Q0);
          \path (center) ++(120:.5cm) coordinate (Q1);
          \path (Q0) ++(60:.75cm) coordinate (C0);
          \path (Q1) ++(120:.75cm) coordinate (C1);
          \draw[very thick, dashed] (Q0) ..controls (C0) and (C1).. (Q1);
        \end{tikzpicture}
        \caption{}\label{fig:ladybug-configuration-2}
      \end{subfigure}
      \begin{subfigure}[b]{.3\linewidth}
        \centering
        \begin{tikzpicture}
          \path (0.25,-4) coordinate (origin);
          \path (origin) ++(-.75,0) coordinate (center);
          \node[draw,circle,fill=black,scale=0.3] at (origin) {};
          \draw[thick] (center) circle [radius=.5cm];
          \path (center) ++(0,.5) coordinate (P0);
          \path (center) ++(0,-.5) coordinate (P1);
          \draw[very thick, dashed] (P0) -- (P1);
          \path (center) ++(60:.5cm) coordinate (Q0);
          \path (center) ++(120:.5cm) coordinate (Q1);
          \path (Q0) ++(60:.75cm) coordinate (C0);
          \path (Q1) ++(120:.75cm) coordinate (C1);
          \draw[very thick, dashed] (Q0) ..controls (C0) and (C1).. (Q1);
        \end{tikzpicture}
        \caption{}\label{fig:ladybug-configuration-3}
      \end{subfigure}
      \caption{Ladybug configurations in $\R^{2} \setminus \{0\}$.}
      \label{fig:ladybugs}
    \end{figure}
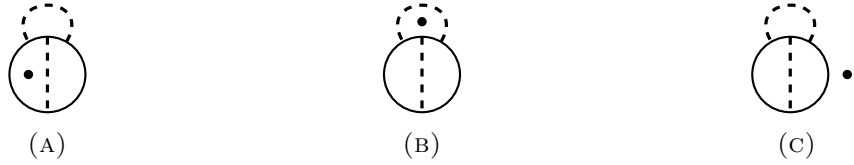

    We discuss these three cases separately.
    \begin{lemma}\label{lem:ladybug-A}
      If $(\cD,\xt,\yt)$ is as in Figure~\ref{fig:ladubug-configuration-1}, then $\cM_{\kh}(\cD_p,\xt_p,\yt_p)=\emptyset$
      and $\cM_{\akh}(\cD,\xt,\yt)=\emptyset$.
    \end{lemma}
    \begin{proof}
      For the first part note that the genus of $(\cD_p,\xt_p,\yt_p)$ is equal to $p>1$.
      Therefore, the moduli space is empty, by Lemma~\ref{lemma:genus-reduction}(1). 

      For the second part note that $(\cD,\yt) \prec (\fs(\cD),\xt)$ if and only if $\yt$ assigns $\xp$ to the unique circle in $\cD$ and $\xt$ assigns
      $\xm$ to the unique circle in $\fs(\cD)$.
      Then $\Ann(\cD,\yt)\neq\Ann(\fs(\cD),\xt)$, so $\cM_{\akh}(\cD,\xt,\yt)=\emptyset$.
    \end{proof}

    The two other cases are dealt with in the following lemma, whose proof is deferred to Subsection~\ref{sub:proof_extending}.
    \begin{lemma}\label{lem:ladybug_lifting}
      Suppose $(\cD,\xt,\yt)$ is a ladybug configuration depicted in Figure~\ref{fig:ladybug-configuration-2} or Figure~\ref{fig:ladybug-configuration-3}. Then
      \[\cM_{\kh}(\cD,\xt,\yt)\cong\cM_{\kh}(\cD_p,\xt_p,\yt_p)^{\Z_p}.\]
      The isomorphism 
      makes the diagram \eqref{eq:cube_commutes} commutative.
    \end{lemma}
    Lemma~\ref{lem:ladybug_lifting} finishes the construction of $\cF_p$ for index 2 resolution configurations.
    Suppose $\ind(\cD,\xt,\yt)>2$. The map $\cF_p\colon\cM_{\akh}(\cD,\xt,\yt)\to \cM_{\akh}(\cD_p,\xt_p,\yt_p)^{\Z_p}$ has already been constructed
    on the boundary. Now $\cM_{\akh}(\cD,\xt,\yt)$ and $\cM_{\akh}(\cD_p,\xt_p,\yt_p)^{\Z_p}$ are disjoint union of disks of dimension $\ind(\cD,\xt,\yt)-1$.
    The map $\cF_p$ is already defined on the boundaries of these moduli spaces.
    As $\ind(\cD,\xt,\yt)-1>1$, there is a unique (up to homotopy) extension of $\cF_p$ to the whole of $\cM_{\akh}(\cD,\xt,\yt)$.
  \end{proof}
  \subsection{Proof of Lemma~\ref{lem:ladybug_lifting}}\label{sub:proof_extending}
  We will prove Lemma~\ref{lem:ladybug_lifting} only for the resolution configuration depicted in Figure~\ref{fig:ladybug-configuration-2}.
  The case of Figure~\ref{fig:ladybug-configuration-3} is similar (and easier), we leave it to the reader.

  Lemma~\ref{lemma:conn-components-moduli-space} implies that $\cM_{\kh}(\cD,\xt,\yt) = \Pi_{1}^{1} \sqcup \Pi_{1}^{2}$ has two connected components.
  Let $Z$ denote the unique circle in $\cD$ and let $A_{1}$ and $A_{2}$ denote the arcs, where $A_{1}$ is the arc lying inside $Z$.
  As $x\succ y$ we must have $\yt(Z)=\xp$ and $\xt(Z') = \xm$ and in this case the poset $P(x,y)$ consists of four elements
  $x\succ \wt{x}_j\succ y$, $j=1,\ldots,4$, where:
  \begin{align*}
    \wt{x}_1=(\fs_{\{A_{1}\}}(\cD),\xt_{A_{1}}^{1}),\  \wt{x}_2=(\fs_{\{A_{1}\}}(\cD),\xt_{A_{1}}^{2}),\\
    \wt{x}_3=(\fs_{\{A_{2}\}}(\cD),\xt_{A_{2}}^{1}), \ \wt{x}_4=(\fs_{\{A_{2}\}}(\cD),\xt_{A_{2}}^{2}).
  \end{align*}
  Here $\xt_{A_{2}}^{1}$ assigns $\xp$ to the inner circle and $\xp$ to the other circle and $\xt_{A_{2}}^{2}$ the opposite way.
  The assignments $\xt_{A_{1}}^{1}$ and $\xt_{A_{1}}^{2}$ are such that $(\fs_{\{A_1\}}(\cD),\xt_{A_1}^1)$ and $(\fs_{\{A_2\}}(\cD),\xt_{A_2}^1)$ are paired under right ladybug matching, that is, the vertices corresponding to posets $x\succ \wt{x}_1\succ y$ and
  $x\succ\wt{x}_3\succ y$ belong to the same connected component of $\cM_{\kh}(\cD,x,y)$; see~\cite[Section 5.4]{LS_stable}.

  Consider now the cover.
  Denote by $Z^{1},Z^{2},\ldots,Z^{p}$ the circles in $\cD_{p}$ and $A_{1}^{1},A_{1}^{2},\ldots,A_{1}^{p}$, $A_{2}^{1},\ldots,A_{2}^{p}$ the lifts of $A_{1}$ and $A_{2}$, respectively.
  The lifts of the circles in $\fs_{\{A_i\}}(\cD)$, $i=1,2$ are denoted by $Z^s_{ij}$ where $s=1,\ldots,p$ enumerates the circles and $j=1,2$ corresponds to the left/right or inner/outer circle in the orbit.
  The lifts of $\fs_{\{A_1,A_2\}}(\tZ)$ are denoted ${Z'}^1,\ldots,{Z'}^p$.

  We need to introduce an extra piece of notation.
  Let $\sigma_1,\sigma_2\in \Perm_p$ be permutations and fix $j_1,j_2\in\{0,\ldots,p\}$. Define the resolution configuration $x^{ij}_{\sigma_1,\sigma_2}$ whose underlying set of circles is given by $\fs_{\{A_1^{\sigma_1(1)},\ldots,A_1^{\sigma_1(j_1)},A_2^{\sigma_2(1)},\ldots,A_2^{\sigma_2(j_2)}\}}$.
  The assignment $\xt^{ij}_{\sigma_1,\sigma_2}$ of $\xp,\xm$ to each of the circles is given as follows. Let $i\in\{1,\ldots,p\}$:
  \begin{itemize}
  \item if $i\in\{\sigma_1(1),\ldots,\sigma_1(j_1)\}$ and $i\in\{\sigma_2(1),\ldots,\sigma_2(j_2)\}$, then the resolution configuration contains ${Z'}^i$ and we assign $\xm$ to it;
  \item if $i\in\{\sigma_1(1),\ldots,\sigma_1(j_1)\}$ and $i\notin\{\sigma_2(1),\ldots,\sigma_2(j_2)\}$, then $Z^i_{1j}$ belong to the resolution configuration ($j=1,2$) and we assign $\xt^1_{A_1}$ to it;
  \item if $i\notin\{\sigma_1(1),\ldots,\sigma_1(j_1)\}$ and $i\in\{\sigma_2(1),\ldots,\sigma_2(j_2)\}$, then $Z^i_{2j}$ belong to the resolution configuration ($j=1,2$) and we assign $\xt^1_{A_2}$ to it;
  \item if $i\notin\{\sigma_1(1),\ldots,\sigma_1(j_1)\}$ and $i\notin\{\sigma_2(1),\ldots,\sigma_2(j_2)\}$, then $Z^i$ belong to the resolution configuration and we assign $\xp$ to it.
  \end{itemize}

  Every maximal chain containing $\wt{x}^p_1$ is of the form
  \begin{equation}\label{eq:def_of_vsigma}
    v_{\sigma_1\sigma_2}:=(\cD_{p},\yt)\prec x_{\sigma_1\sigma_2}^{10}\prec\ldots\prec x_{\sigma_1\sigma_2}^{p0}\prec x_{\sigma_1\sigma_2}^{p1}\prec\ldots\prec x_{\sigma_1\sigma_2}^{p,p-1}\prec(\fs(\cD_{p}),\xt).\end{equation}
  Likewise, every maximal chain containing $\wt{x}^p_3$ is of form
  \[w_{\sigma_1\sigma_2}:=(\cD_{p},\yt)\prec x_{\sigma_1\sigma_2}^{01}\prec\ldots\prec x_{\sigma_1\sigma_2}^{0p}\prec x_{\sigma_1\sigma_2}^{1p}\prec\ldots\prec x_{\sigma_1\sigma_2}^{p-1,p}\prec(\fs(\cD_{p}),\xt).\]
  Recall that in Definition~\ref{def:assoc_vertex} and Lemma~\ref{lem:assoc_reverse} we established a correspondence between maximal chains and vertices in the moduli space.
  We will use this correspondence in the rest of the proof. Our next aim is to prove the following result.
  \begin{lemma}\label{lem:same_component}
    The vertices $v_{\sigma_1\sigma_2}$ and $w_{\sigma_1\sigma_2}$ belong to the same connected component of $\cM_{\cC_{\kh}}(x_p,y_p)$.
  \end{lemma}
  \begin{proof}[Proof of Lemma~\ref{lem:same_component}]
    Fix $\sigma_1$ and $\sigma_2$ and write $x^{ij}$ for $x^{ij}_{\sigma_1\sigma_2}$.
    Denote by $u^{ij}$ the maximal chain
    \[x^{00}\prec x^{10}\prec\ldots\prec x^{i0}\prec x^{i1}\prec x^{ij}\prec x^{i+1,j}\prec x^{i+1,j+1}\prec\ldots\prec x^{i+1,p}\prec x^{i+2,p}\prec\ldots\prec x^{pp}.\]
    We note that $u^{p0}=v_{\sigma_1\sigma_2}$ and $u^{0p}=w_{\sigma_1\sigma_2}$.
    Moreover, $u^{i0}=u^{i+1,p}$ by construction.  It is therefore enough to prove that $u^{ij}$ and $u^{ij+1}$ belong to the same connected component of $\cM_{\cC_{\kh}}(x,y)$.
    To this end define
    \begin{align*}
      w_{start}&=x^{00}\prec\ldots\prec x^{ij}\in \cM_{\cC_{\kh}}(x^{ij},y)\\
      w_{end}&=x^{i+1,j+1}\prec\ldots \prec x^{i+1,p}\prec x^{i+2,p}\prec\ldots\prec x^{p,p}\in\cM_{\cC_{\kh}}(x,x^{i+1,j+1})\\
      w_1&=x^{i,j}\prec x^{i+1,j}\prec x^{i+1,j+1}\in\cM_{\cC_{\kh}}(x^{i+1,j+1},x^{ij})\\
      w_2&=x^{i,j}\prec x^{i,j+1}\prec x^{i+1,j+1}\in\cM_{\cC_{\kh}}(x^{ij},x^{i+1,j+1}).
    \end{align*}
    Let $\iota\colon\cM_{\cC_{\kh}}(x^{ij},y)\times \cM_{\cC_{\kh}}(x^{i+1,j+1},x^{ij})\times \cM_{\cC_{\kh}}(x,x^{i+1,j+1})\hookrightarrow\cM_{\cC_{\kh}}(x,y)$ be the inclusion.
    We have that $u^{ij}=\iota(w_{start},w_1,w_{end})$ and $u^{ij+1}=\iota(w_{start},w_2,w_{end})$.
    Therefore it is enough to prove that $w_1$ and $w_2$ belong to the same connected component of $\cM_{\cC_{\kh}}(x^{i+1,j+1},x^{ij})$.
    There are two cases.
    \begin{itemize}
    \item $\sigma_1(i+1)\neq \sigma_2(j+1)$;
    \item $\sigma_1(i+1)=\sigma_2(j+1)$.
    \end{itemize}
    In the first case the associated surface consists of two components, both being pair of pants. As the genus is zero, by Lemma~\ref{lemma:conn-components-moduli-space} $\cM_{\cC_{\kh}}(x^{i+1,j+1},x^{ij})$ is connected.

    The second case corresponds to the ladybug matching.
    Writing $s=\sigma_1(i+1)$ we obtain that $\cM_{\cC_{\kh}}(x^{i+1,j+1},x^{ij})=\cM_{\cC_{\kh}}(\cD^p_s,\xt^p_s,\yt^p_s)$, where $\tZ(\cD^p_s)=Z^s$, $\tA(\cD^p_s)=\{A^s_1,A^s_2\}$, $\yt^p_s(Z^s)=\xm$, $\xt^p_s({Z'}^s)=\xp$, that is, this is the moduli space corresponding to a ladybug configuration.
    Our definition of assignments $\xt^p_s$ and $\yt^p_s$ implies that $w_1$ and $w_2$ belong to the same connected component of the moduli space.
    This concludes the proof of Lemma~\ref{lem:same_component}.
  \end{proof}
  \begin{lemma}
    Let $\sigma_1',\sigma_2'\in \Perm_p$ be another permutation and let $v_{\sigma_1'\sigma_2'}$ be a chain as in \eqref{eq:def_of_vsigma}. Then
    the vertices corresponding to $v_{\sigma_1\sigma_2}$ and $v_{\sigma_1'\sigma_2'}$ belong to the same connected component of $\cM_{\cC_{\kh}}(x_p,y_p)$.
  \end{lemma}
  \begin{proof}
    It is enough to prove the result if $\sigma_1=\sigma_1'$ and $\sigma_2'$ differs from $\sigma_2$ by swapping two adjacent elements (or $\sigma_2=\sigma_2'$ and $\sigma_1$
    differs from $\sigma_1'$ by a single transposition of elements). The proof in that case is essentially a repetition of the argument used in the proof of Lemma~\ref{lem:same_component}
    so we leave it to the reader.
  \end{proof}
  Let $\Pi$ be the connected component of $\cM_{\cC_{\kh}}(x_p,y_p)$ that contains all of the $v_{\sigma_1\sigma_2}$ and $v_{\sigma_1\sigma_2}$. Note that the group action
  takes $v_{\sigma_1\sigma_2}$ to $v_{\sigma_1'\sigma_2'}$ for some other permutation $\sigma_1',\sigma_2'$, therefore the component $\Pi_1$ is preserved. The fixed point set
  $\Pi^{\Z_p}$ is a one-dimensional permutohedron $\Pi_2$, which is diffeomorphic with $\Pi_{2p-1}^{\Z_p}$. The diffeomorphism is realized by the restriction
  of the cover map $\ff\colon\cM_{\cC_{\kh}}(x_p,y_p)\to \cM_{\Cubes(2p)}(\ff(x_p),\ff(y_p))$.

  We define now the isomorphism $\cM_{\cC_{\kh}}(x,y)\to \cM_{\cC_{\kh}}(x_p,y_p)^{\Z_p}$ in such a way that the segment connecting the vertices $x\succ \wt{x}_1\succ y$
  and $x\succ\wt{x}_3\succ y$ in $\cM_{\cC_{\kh}}(x,y)$ is mapped to a segment in $\Pi^{\Z_p}$ (which is a disjoint union of segments).
  The segment connecting the vertices $x\succ\wt{x}_2\succ y$ to $x\succ\wt{x}_4\succ y$ is mapped to the other connected component of $\cM_{\kh}(x_p,y_p)^{\Z_p}$.

  \section{Equivariant Khovanov homology}\label{sec:equivhomo}
  We begin with a brief review of the construction of the equivariant Khovanov homology~\cite{Politarczyk_Khovanov}. Later on, we merge this construction with the construction of the equivariant Khovanov homotopy type that we introduced in Section \ref{sec:review}.

  \subsection{Review of the construction}
  Let $D$ be an $m$-periodic diagram representing an $m$-periodic link $L$.
  The symmetry of $L$ can be realized by a cobordism in $S^3 \times I$ in the following way.
  Suppose the rotation center is at $0\in \R^2$.
  Consider $D\times I \subset \R^2\times I$ and twist it by the diffeomorphism $\eta\colon \R^2 \times I\to \R^2 \times I$ given by $(x,t)\mapsto (\Psi_{2\pi t/m} x,t)$, where $\Psi_\theta$ is a counterclockwise rotation by the angle $\theta$.
  The image
  \begin{equation}\label{eq:cobordism}
    \Sigma_D=\eta(D\times I)
  \end{equation}
  is a cobordism from $D$ to $D$.
  Note that this is a product cobordism, and there are no handle attachments.

  In~\cite{BarNatan} a map $\phi_{\Sigma_D}$ of Khovanov chain complexes was associated to each cobordism of diagrams $\Sigma$.
  The chain homotopy class of this map was later shown to be functorial, i.e.,\ not depending on the isotopy type of $\Sigma$; see~\cite{Vogel}.\mpar{(60): we added a reference
    to Vogel's result proving functoriality. From functoriality it follows that $\phi_{\Sigma_D}$ is well-defined and can be used to show that the group action on the Khovanov
  chain complex is well-defined. In fact, defining a group action is relatively easy, but many proofs of invariance can be simplified, once we have a functorial map $\phi_{\Sigma_D}$}
  Notice that only the chain homotopy type of \(\phi_{\Sigma_{D}}\) is well-defined.
  However, since \(\Sigma_{D}\) is a composition of Reidemeister moves, it is possible to choose a representative for \(\phi_{\Sigma_{D}}\), which induces a group action on the Khovanov complex.

  \begin{proposition}\cite[Section 2]{Politarczyk_Khovanov}\label{prop:cobordisminduces}
    Let $R$ be a commutative ring with a unit.
    It is possible to choose a representative of the cobordism map $\phi_{\Sigma_D}$, which induces a group action on the chain complex $\ckh(D;R)$. 
  \end{proposition}

  The Khovanov complex $\ckh(D;R)$ admits an action of the cyclic group $\Z_m$.
  Therefore $\ckh(D;R)$ can be regarded as a $\Lambda$-module, where $\Lambda=R[\Z_m]$.
  For an $R$-module $M$, define the equivariant Khovanov homology of $L$ in gradings $k$ and $q$ as
  \begin{equation}\label{eq:extdef}
    \ekh^{j,q}(L; M)=\Ext^j_\Lambda(M;\ckh^{*,q}(D;R)).
  \end{equation}
  Since $\Ext^j_\Lambda(M;\ckh^{*,q}(D;R))$ does not depend on the chosen diagram $D$, equivariant Khovanov homology is an invariant of periodic links~\cite{Politarczyk_Khovanov}.
  \mpar{(61): The $\Ext$
  functor is well defined (not up to homotopy), because $\ckh(D;R)$ is defined up to quasiisomorphism. This is homological algebra, not topology.} 
  \smallskip
  The construction of equivariant Khovanov homology also works in the annular case. The methods of \cite{Politarczyk_Khovanov} carry over
  to the annular case without significant modifications.
  Namely, we observe that the annular chain complex $\cakh^{*,q,k}(D;R)$ admits a $\Z_m$-action, hence it is a $\Lambda$-module.
  Next, we define
  \[\eakh^{j,q,k}(L;M)=\Ext^j_\Lambda(M;\cakh^{*,q,k}(D;R)).\]
  Essentially the same argument as in \cite{Politarczyk_Khovanov} can be used to show that $\eakh$ is an invariant of an annular link.

  \subsection{Equivariant Khovanov homology as Borel cohomology}
  We now have two ways of getting equivariant homology from the Khovanov theory.
  One way is to use the definition given in~\eqref{eq:extdef}.
  Another way uses the Borel cohomology of space $\mathcal{X}_{\kh}(D)$.
  We will now show that the two constructions agree.
  In the rest of this section we denote by $\widetilde{C}^{\ast}(\mathcal{X}_{\kh}(D);R)$ the reduced cellular cochain complex of $\mathcal{X}_{\kh}(D)$ associated to the CW-structure described in Subsection~\ref{sub:equivariant_geometric}.

  First we state a preparatory result.
  \begin{proposition}\label{prop:identification-kh-hom}
    Let $D_{m}$ be an $m$-periodic diagram of a link.
    There exists an identification of cochain complexes of $R[\Z_m]$-modules
    \begin{equation}\label{eq:identification-kh-hom}
      \widetilde{C}^{\ast}(\mathcal{X}_{\kh}(D_{m});R) \cong \ckh^{\ast}(D_{m};R).
    \end{equation}
    Here it should be understood that the structure of the $R[\Z_m]$ cochain complex is given by the $\Z_m$-action on $\widetilde{C}^{\ast}(\mathcal{X}_{\kh}(D_{m});R)$ and on $\ckh(D_{m})$.
  \end{proposition}
  \begin{proof}
    The statement is a consequence of the construction of the cochain complex of $\widetilde{C}^{\ast}(\mathcal{X}_{\kh}(D_{m}))$.
    The cellular cochain complex $\widetilde{C}^{\ast}(\mathcal{X}_{\kh}(D_{m});R)$ was constructed in Subsection~\ref{sub:chain_complex}.
    The construction is that the generators of $\widetilde{C}^{\ast}(\mathcal{X}_{\kh}(D_{m});R)$ correspond to the generators of $\ckh^{\ast}(D_{m};R)$.
    The differential on $\widetilde{C}^{\ast}(\mathcal{X}_{\kh}(D_{m});R)$ is the same as in $\ckh(D_{m};R)$.
    In Subsection~\ref{sub:equivchain} it was shown that the induced group actions on $\widetilde{C}^{\ast}(\mathcal{X}_{\kh}(D_{m});R)$ and $\ckh(D_{m};R)$ coincide.
  \end{proof}

  In order to state and prove the next result, we need to set up some notation and recall some basic facts from homological algebra.
  If $C_{\ast}$ is a chain complex, we will associate to it a cochain complex $C_{r}^{\ast}$ defined by $C_{r}^{-n}=C_{n}$ with the differential $d^{-n}_{r} \colon C_{r}^{-n} \to C_{r}^{-n+1}$ defined by $d^{-n}_{r} = (-1)^{n} d_{n}$, where $d_{n} \colon C_{n} \to C_{n-1}$ is the differential in $C_{\ast}$.
  For two cochain complexes $C^{\ast}$ and $D^{\ast}$ we define the $\Hom$ cochain complex
  \[\Hom^{n}_{R}(C^{\ast},D^{\ast}) = \prod_{p\in \Z} \Hom_{R}(C^{p},D^{p+n}),\]
  with the differential $d_{\Hom}^{n}(f) = d_{D}^{\ast} \circ f - (-1)^{n} f \circ d_{C}^{\ast}$.
  If $P^{\ast}$ is a projective resolution of a cochain complex $C^{\ast}$ and $I^{\ast}$ is an injective resolution of $D^{\ast}$, then
  \begin{equation}\label{eq:uuu}
    \Ext^{n}_{R}(C^{\ast},D^{\ast}) = H^{n}(\Hom^{\ast}_{R}(P^{\ast},I^{\ast})) \cong H^{n}(\Hom^{\ast}_{R}(P^{\ast},D^{\ast})) \cong H^{n}(\Hom^{\ast}_{R}(C^{\ast},I^{\ast})).
  \end{equation}

  Recall that to any discrete group $G$, we can associate a contractible space $EG$ equipped with the free action of $G$.
  By $BG = EG/G$ we denote the classifying space of $G$.
  For a $G$-space $X$ and any finitely generated $R[G]$-module $M$ we define the \emph{Borel equivariant cohomology} of $X$\mpar{(62) we replaced $G$-module by $R[G]$-module
  to clarify that we work with algebraic and not topological objects}
  \[H^{\ast}_{G}(X;M) = H^{\ast}(EG \times_{G} X;M) \cong \Ext_{R[G]}^{\ast}(C^{\ast}_{r}(X;R);M),\]
  where $C^{\ast}_{r}(X)$ denotes the cochain complex associated to the cellular cochain complex $C_{\ast}(X)$ of $X$ using the convention described above.
  In particular, we have $C^{\ast}(X;R) = \Hom^{\ast}_{R}(C^{\ast}_{r}(X);R)$.

  There is a natural $G$-map $EG \times X \to EG$ which, after taking quotient of both sides, yields a map
  \[EG \times_{G} X \xrightarrow{p} BG.\]
  We define the \emph{reduced Borel cohomology} of $X$, 
  to be \[\widetilde{H}^{\ast}_{G}(X;M)=\coker\left(H^{\ast}(BG;M) \xrightarrow{p^{\ast}} H^{\ast}_{G}(X;M)\right).\]
  It is easy to check that
  \[\widetilde{H}^{\ast}_{G}(X;M) \cong \Ext_{R[G]}^{\ast}\left(\widetilde{C}^{\ast}_{r}(X;R);M\right).\]

  \begin{theorem}\label{thm:borel}
    Let $D_{m}$ be an $m$-periodic link diagram of a link $L_{m}$ and fix a field $\F$.
    For any $\F[\Z_{m}]$-module $M$, the equivariant Khovanov homology $\ekh^{i,q}(L_{m};M)$ is isomorphic to the reduced Borel cohomology of $\mathcal{X}_{\kh}(D_{m})$
    \[\ekh^{i,q}(L_{m};M) \cong H^{\ast}_{\Z_{m}}(\mathcal{X}_{\kh}(D_{m}),\Hom_{\F}(M,\F)),\]
    where $g \in G$ acts on $\Hom_{\F}(M,\F)$ via $(g \cdot f)(x) = f(g^{-1}\cdot x)$.
  \end{theorem}
  \begin{proof}
    To begin with, observe that
    \[\ekh^{\ast,q}(L_{m};M) = \Ext^{\ast}_{\F[G]}(M,\ckh^{\ast,q}(D_{m};\F)) = \Ext^{\ast}_{\F[G]}\left(M,\widetilde{C}^{\ast}(\mathcal{X}_{\kh}^{q}(D_{m});\F)\right).\]
    Here the first equality is the definition of equivariant Khovanov homology while the second equality follows from Proposition~\ref{prop:identification-kh-hom}.

    Next, let $P_{M}^{\ast}$ be a projective resolution of $M$.
    We have
    \begin{align*}
      \Ext^{\ast}_{\F[G]}(M,\widetilde{C}^{\ast}(\mathcal{X}_{\kh}^{q}(D_{m});\F))
      &\stackrel{(1)}{\cong}\Ext^{\ast}_{\F[G]}(M,\Hom_{\F}(\widetilde{C}^{\ast}_{r}(\mathcal{X}_{\kh}(D_{m})^{q};\F),\F))\cong\\
      &\stackrel{(2)}{\cong} H^{\ast}(\Hom^{\ast}_{\F[G]}(P^{\ast}_{M},\Hom_{\F}(\widetilde{C}^{\ast}_{r}(\mathcal{X}_{\kh}^{q}(D_{m});\F),\F))) \cong \\
      &\stackrel{(3)}{\cong} H^{\ast}(\Hom_{\F}^{\ast}(P^{\ast}_{M} \otimes_{\F[G]} \widetilde{C}^{\ast}_{r}(\mathcal{X}_{\kh}^{q}(D_{m});\F),\F)) \cong \\
      &\stackrel{(4)}{\cong} H^{\ast}(\Hom_{\F}^{\ast}(\widetilde{C}^{\ast}_{r}(\mathcal{X}_{\kh}(D_{m});\F),\Hom^{\ast}_{\F}(P_{M}^{\ast},\F))),
    \end{align*}
    where (1) comes from the definition of $C_{r}^{\ast}(\mathcal{X}_{\kh}(D_{m})^{q};\F)$, (2) is the definition of the $\Ext$ functor and  the isomorphisms (3) and (4) come from~\cite[Exercise I.0.6]{Brown_book}.
    Since $\F$ is a field, any projective $\F[G]$-module is also injective by~\cite[Exercise 1.10.24]{CurtisReiner}.
    The functor
    \[M \mapsto \Hom_{\F}(M,\F)\]
    defined on the category of left $\F[G]$-modules is exact by~\cite[Exercise 1.10.22]{CurtisReiner} and maps projective modules to projective modules by~\cite[Corollary 1.10.29]{CurtisReiner}.
    Therefore, $\Hom^{\ast}(P_{M}^{\ast},\F)$ becomes an injective resolution of the $\F[G]$-module $\Hom_{R}(M,\F)$.
    Consequently, by~\eqref{eq:uuu}
    \[\Ext^{\ast}_{\F[G]}\left(M,\widetilde{C}^{\ast}(\mathcal{X}_{\kh}^{q}(D_{m});\F)\right) \cong \Ext^{\ast}_{\F[G]} \left(\widetilde{C}^{\ast}_{r}(\mathcal{X}_{\kh}(D_{m})),\Hom_{\F}(M,\F)\right)\]
    and the proof is finished.
  \end{proof}

  We remark that the same argument as in the proof of Theorem~\ref{thm:borel} shows the following result.
  \begin{proposition}
    Let $L_{m}$ be an $m$-periodic annular link and let $D_{m}$ be an $m$-periodic diagram of $L_{m}$.
    For a field $\F$ and any $\F[G]$-module $M$, it holds
    \[\eakh^{i,q,k}(L_m;M)\cong H^{\ast}_G(\mathcal{X}_{\akh}^{q,k}(D_{m}),\Hom_{\F}(M,\F)).\]
  \end{proposition}

  \subsection{Stable cohomology operations}\label{sub:main2}
  Given two generalized cohomology theories $X(\cdot)$ and $Y(\cdot)$, a stable cohomology operation of degree $k$ is a family of natural transformations between functors $X^l(\cdot)$ and $Y^{k+l}(\cdot)$ commuting with suspension.
  We focus on stable cohomology operations in singular homology over a finite field.
  These operations form a Steenrod algebra.
  Standard references include~\cite[Section 4.L]{Hatcher} and~\cite[Section 10.4]{KirkDavis}.

  The Steenrod algebra $\cA_2$ over $\Z_2$ is generated by the Steenrod squares $\Sq^i\colon H^{*}(\cdot,\Z_2)\to H^{*+i}(\cdot,\Z_2)$, with $\Sq^1$ being the Bockstein homomorphism corresponding to the short exact sequence $0\to\Z_2\to\Z_4\to\Z_2\to 0$.

  For a prime $p>2$, the Steenrod algebra $\cA_p$ is generated by the Bockstein homomorphism $\beta$, and operations $P^k\colon H^{*}(\cdot,\Z_p)\to H^{*+2k(p-1)}(\cdot,\Z_p)$.\mpar{(63)}
  The homomorphism $\beta$ is of degree $1$, and it is the connecting homomorphism of the long exact sequence of cohomology induced by the short exact sequence of groups $0\to\Z_p\to\Z_{p^2}\to\Z_p\to 0$.

  Coming back to Khovanov homology we make the following observation, see~\cite{LS_Steenrod,LS_refinement}.
  \begin{proposition}\label{prop:steenrod}
    Let $\alpha$ be a stable cohomology operation of degree $k$ over $\Z_p$.
    Then, given a link $L$ and $q\in\Z$, the map $\alpha$ induces a well defined map
    \[\alpha_q\colon \kh^{*,q}(L;\Z_p)\to\kh^{*+k,q}(L;\Z_p).\]
  \end{proposition}

  There appeared several algorithms for computing Steenrod squares in Khovanov homology, so the invariants based on Steenrod squares can be effectively computed (see~\cite{LS_Steenrod,LOS}).
  The \emph{knotkit} package~\cite{Knotkit} implements the algorithm of~\cite{LS_Steenrod}.
  We remark that the maps $\Sq^1$ and $\beta$ are determined by the integral Khovanov homology, see~\cite[Section 2.5]{LS_Steenrod}.

  The next statement shows that Steenrod operations commute with group action.
  \begin{theorem}\label{thm:naturality}
    Let $L_{m}$ be an $m$-periodic link and $\F$ a field.
    Any stable cohomology operation
    \[\alpha \colon H^{\ast}(-;\F) \to H^{\ast+k}(-;\F)\]
    commutes with the action of $\Z_m$ on $\kh(L_{m};\F)$.
  \end{theorem}
  \begin{proof}
    Cohomology operations are natural, so they commute with the group action on cohomology of $\mathcal{X}_{\kh}(D_{m})$, where $D_{m}$ is some $m$-periodic diagram of $L_{m}$.
    On the other hand, Proposition~\ref{prop:identification-kh-hom} shows that the $\Z_m$-action on the cohomology of $\X_{L_m}$ commutes with the group action on the Khovanov homology of $L_m$. 

  \end{proof}
  \subsection{Fixed Point theorems}
  \label{sub:localization}
  Recall that $B\Z_{p}$ is the classifying space of the finite cyclic group of order $p$.
  The cohomology ring of $B\Z_{p}$ is given below, see~\cite[Example 3E.2]{Hatcher},
  \[H^{\ast}(B\Z_{p};\F_{p}) \cong
    \begin{cases}
      \F_{2}[X], \text{ for } X \in H^{1}(B\Z_{2};\F_{2}), & p=2, \\
      \F_{p}[Y] \otimes_{\F_{p}} \Lambda^*_{\F_{p}}(Z), \text{ for } Y \in H^{2}(B\Z_{p};\F_{p}) \text{ and } Z \in H^{1}(B\Z_{p};\F_{p}), & p >2,
    \end{cases}
  \]
  where $\Lambda^*_{\F_{p}}(Z)$ is the exterior algebra over $\F_{p}$ generated by $Z$.
  Write $S_{p} \subset H^{\ast}(B\Z_{p};\F_{p})$ for the multiplicative set generated either by $X$, if $p=2$, or by $Y$, when $p>2$.

  \begin{theorem}[{\cite{Quillen,Borel}}]\label{thm:localization-theorem}
    Let $X$ be a $\Z_{p}$-CW-complex with $p$ a prime.\mpar{(64) various definitions of $G$-CW-complexes are given in Subsection~\ref{sub:equiv_cell}.}
    There exists an isomorphism of graded $\Z_{p}$-algebras
    \[S_{p}^{-1}H^{\ast}_{G}(X;\F_{p}) \cong S_{p}^{-1}H^{\ast}(X^{G};\Z_{p}).\]
  \end{theorem}

  As an immediate corollary of Theorem~\ref{thm:localization-theorem} and Theorem~\ref{thm:main-theorem}, we get\mpar{(65)}
  \begin{theorem}\label{thm:localization}
    Let $L_{p} \subset S^{1} \times D^{2}$ be a $p$-periodic link with $L$ denoting the quotient link.
    For any $q,k \in \Z$ there exists an isomorphism of $S_{p}^{-1}H^{\ast}(B\Z_{p};\F_{p})$-modules
    \[S_{p}^{-1}H^{\ast}(B\Z_{p};\F_{p}) \otimes_{\F_{p}} \akh^{\ast,q,k}(L;\F_{p}) \xrightarrow{\cong} S_{p}^{-1} \eakh^{\ast,pq-(p-1)k,k}(L_{p};\F_{p}).\]
  \end{theorem}

  Let $\cA_{p}$ denote, for any prime $p$, the \emph{mod $p$ Steenrod algebra}, i.e.\ the algebra of stable $\F_{p}$-cohomology operations.
  It turns out that Theorem~\ref{thm:localization} can be strengthened considerably when we take into account the action of the Steenrod algebra.
  Before stating the main result, let us introduce the following terminology.

  For $p=2$, consider the sequence of nonnegative numbers $I = (s_{1},s_{2},\ldots,s_{m})$.
  The sequence $I$ is \emph{admissible} if $s_{i} \geq 2s_{i+1}$, for $1 \leq i < m$.
  Define the \emph{degree} of $I$, $|I| = s_{1}+s_{2}+\ldots+s_{m}$ and \emph{excess} of $I$, $e(I) = 2s_{1}-|I|$.
  To any sequence $I$ we associate the cohomology operation $Sq^{I} = Sq^{s_{1}} Sq^{s_{2}} \cdots Sq^{s_{m}} \in \mathcal{A}_{2}$.
  For any $k \geq 0$ we set:
  \[L_{2}(k) = \{I \colon e(I) \geq k\}.\]

  For $p>2$, let $I = (\epsilon_{0},s_{0},\epsilon_{1},s_{1},\ldots,s_{m},\epsilon_{m})$, where $s_{1},s_{2},\ldots,s_{m}$ are nonnegative integers and $\epsilon_{i} \in \{0,1\}$, for $0 \leq i \leq m$.
  The sequence $I$ is \emph{admissible} if $s_{i}\geq ps_{i+1} + \epsilon_{i}$, for any $0 \leq i < m$.
  Define the \emph{degree} of $I$, $|I| = 2(p-1)(s_{1}+s_{2}+\ldots+s_{m}) + \epsilon_{0}+\epsilon_{1}+\ldots+\epsilon_{m}$ and \emph{excess} of $I$, $e(I) = 2s_{1}p+2\epsilon_{0}-|I|$.
  To any sequence $I$ we can associate a cohomology operation $P^{I} = \beta^{\epsilon_{0}} P^{s_{1}} \beta^{\epsilon_{1}}P^{s_{2}} \cdots P^{s_{m}}\beta^{\epsilon_{m}} \in \mathcal{A}_{p}$.
  Let, for any $k \geq 0$,
  \[L_{p}(k) = \{I \colon e(I) \geq k+1, \text{ or } e(I) = k \text{ and } \epsilon_{0}=0\}.\]

  \begin{definition}\label{defn:unstable-elements}
    For a graded $\cA_{p}$-module $M^{\ast}$ the submodule of \emph{unstable elements}, $\Un(M)^{\ast}$, is a graded submodule defined as
    \[\Un(M)^{k} =
      \begin{cases}
        \{x \in M^{k} \colon \forall_{I \in L_{2}(k)} \quad Sq^{I}(x) = 0\}, & p=2,\\
        \{x \in M^{k} \colon \forall_{I \in L_{p}(k)} \quad P^{I}(x) = 0\}, & p>2.
      \end{cases}
    \]
  \end{definition}

  Equivariant annular Khovanov homology $\eakh^{\ast,\ast,\ast}(L;\F_{p})$ is isomorphic to Borel cohomology of $\mathcal{X}_{\akh}(L)$, hence the action of the Steenrod algebra on $\akh^{\ast,\ast,\ast}(L;\F_{p})$ extends to $\eakh^{\ast,\ast,\ast}(L;\F_{p})$.
  By~\cite[Proposition 2.1]{Wilkerson}, the action of $\cA_{p}$ on equivariant annular Khovanov homology extends uniquely to an action on the localization $S_{p}^{-1}\eakh^{\ast,\ast,\ast}(L;\F_{p})$.
  \begin{theorem}\label{thm:khovanov-homology-quotient_link}\ Let $p$ be a prime.\mpar{(65)}
    \begin{itemize}
    \item[(a)]
      If $L_{p} \subset S^{1} \times I $ is $p$-periodic link and $L$ is the quotient link, then for any $q,k \in \Z$ there exists an isomorphism of rings
      \[\akh^{\ast,q,k}(L;\F_{p}) \otimes_{\F_{p}} H^{\ast}(B\Z_{p};\F_{p}) \xrightarrow{\cong} \Un\left(S_{p}^{-1} \eakh^{\ast,pq-(p-1)k,k}(L_{p};\F_{p})\right)^*.\]
      Consequently
      \[\akh^{\ast,q,k}(L;\F_{p}) \cong \F_{p} \otimes_{H^{\ast}(B\Z_{p};\F_{p})} \Un\left(S_{p}^{-1} \eakh^{\ast,pq-(p-1)k,k}(L_{p};\F_{p})\right)^*.\]

    \item[(b)]
      For a $p$-periodic link $L_{p} \subset S^{3}$ and for any $q \in \Z$ it holds:
      \[\bigoplus_{\substack{q',k' \in \Z \\ pq'-(p-1)k'=q}} \akh^{\ast,q',k'}(L_{p};\F_{p}) \cong \F_{p} \otimes_{H^{\ast}(B\Z_{p};\F_{p})} \Un \left(S_{p}^{-1}\ekh^{\ast,q}(L;\F_{p})\right)^*,\]
      where $L$ denotes the quotient link.
    \end{itemize}
  \end{theorem}
  \begin{proof}
    This is an immediate corollary of~\cite[Corollary 2.5.]{dwyer-wilkerson} and Theorem~\ref{thm:borel}.
  \end{proof}

  Smith inequalities given in Theorem~\ref{cor:smith-inequality-annular-kh} and Theorem~\ref{cor:smith-inequality-kh} are corollaries of Theorem~\ref{thm:khovanov-homology-quotient_link}. We prove now Theorem~\ref{cor:smith-inequality-annular-kh}; an analogous proof works for the case of Theorem~\ref{cor:smith-inequality-kh}.

  \begin{proof}[Proof of Theorem~\ref{cor:smith-inequality-annular-kh}]
    We have the following chain of inequalities
    \begin{align*}
      \dim_{\F_{p}}\akh^{\ast,pq-(p-1)k,k}(L_{p};\F_{p}) &\geq \operatorname{rank}_{H^{\ast}(B\Z_{p};\F_{p})}\eakh^{\ast,pq-(p-1)k,k}(L_{p};\F_{p}) \geq \\
                                                         &\geq \operatorname{rank}_{H^{\ast}(B\Z_{p};\F_{p})}\Un\left(S_{p}^{-1}\eakh^{\ast,pq-(p-1)k,k}(L_{p};\F_{p})\right)^* = \\
                                                         &= \dim_{\F_{p}} \akh^{\ast,q,k}(L;\F_{p}).
    \end{align*}
    The first inequality is a consequence of the definition of equivariant annular Khovanov homology.
    The second inequality is a natural consequence of the properties of the localization, and the last equality follows from Theorem~\ref{thm:khovanov-homology-quotient_link}.
  \end{proof}

  \appendix
  \section{$\langle n\rangle$-manifolds}\label{sec:n_manifolds}

  \subsection{Manifolds with corners}\label{sub:manifolds with corners}

  We say that $M$ is a $k$-\emph{manifold with corners} if it is locally modelled on open subsets of $\mathbb{R}_{+}^{k}$, where $\mathbb{R}_{+} = [0,\infty)$.
  In other words, $M$ is equipped with an atlas $\cA = \{(U,\psi_{U})\}$ such that every $U$ is an open subset in $\mathbb{R}_{+}^k$ and the transition maps $\psi_{V}^{-1} \circ \psi_{U}$, for any $U$ and $V$ such that $U \cap Y \neq \emptyset$, are $C^{\infty}$ diffeomorphisms.
  Compare~\cite[Definition 2.1]{Joyce}.

  For every point $x \in M$ we can define its \emph{codimension}, denoted by $c(x)$, which records the number of coordinates of $\psi_{U}(x)$ which are zero for any chart $(U,\psi_U)$ for which $x \in U$.
  Moreover, we define \emph{codimension-$i$ boundary} to be the set $\{x \in M \colon c(x) = i\}$.
  A \emph{connected face} of codimension $i$ of $M$ is the closure of a connected component of the codimension-$i$ boundary of $M$.
  A \emph{face} is a (possibly empty) disjoint union of connected faces of the same codimension. A codimension-1 face is usually called a \emph{facet}.
  A \emph{$k$-dimensional manifold with faces} is a $k$-manifold with corners such that every point $x \in M$ belongs to exactly $c(x)$ non-empty connected facets of $M$.
  Moreover, we say that $M$ is an \emph{$k$-dimensional $\langle n \rangle$-manifold} if $M$ is an $k$-dimensional manifold with faces and there exists a decomposition $\partial M = \partial_1 M \cup \partial_2 M \cup \ldots \cup \partial_n M$ such that
  \begin{itemize}
  \item $\partial_i M$ is a facet of $M$, for every $1 \leq i \leq n$,
  \item $\partial_i M \cap \partial_j M$ is a facet of both $\partial_i M$ and $\partial_j M$, for every $1 \leq i < j \leq n$.
  \end{itemize}
  We refer to~\cite[Section 2]{Joyce} for discussion of the notion of an $\langle n\rangle$-manifold.

  \begin{example}[see \expandafter{\cite[Definition 3.9]{LS_stable}}]
    For an $(n+1)$-tuple $\dc=(d_0,d_1,\ldots,d_n)$ of non-negative integers define
    \[E_n^{\dc}=\R^{d_0}\times \R_{\ge 0}\times \R^{d_1}\times\R_{\ge 0}\ldots\times \R_{\ge 0}\times\R^{d_n}.\]
    We make $E_n^{\dc}$ an $\langle n\rangle$-manifold by declaring that
    \[\partial_i E_n^{\dc}=\R^{d_0}\times \R_{\ge 0}\times \R^{d_1}\times\R_{\ge 0}\ldots\R^{d_{i-1}}\times\{0\}\times\R^{d_i}\times\cdots\times \R_{\ge 0}\times\R^{d_n}.\]
  \end{example}
  \begin{example}
    An $n$-dimensional permutohedron has a structure of an $\langle n\rangle$-manifold, see Subsection~\ref{sub:permutstart} below.
  \end{example}

  We will need the following construction, see \cite[Construction 3.4]{LLS_long}.
  \begin{construction}
    Suppose $X$ is an $\langle n\rangle$-manifold and $Y$ is an $\langle m\rangle$-manifold. The product $X\times Y$ is given the structure of an $\langle n+m\rangle$-manifold by declaring.
    \[
      \partial_i(X\times Y)=\begin{cases} (\partial_i X)\times Y & \textrm{ if $i\le n$}\\ X\times(\partial_{i-n} Y) & \textrm{ if $i>n$.}\end{cases} 
    \]
  \end{construction}

  In some cases, it is more convenient to view $\langle n \rangle$-manifolds as certain functors to the category of topological spaces.
  Let $\underline{2}^{1}$ denote the category consisting of two objects $0$ and $1$ with a single non-identity morphism $0 \to 1$.
  For an integer $n>1$ let
  \[\underline{2}^{n} = \underbrace{\underline{2}^{1} \times \underline{2}^{1} \times \cdots \times \underline{2}^{1}}_{n}.\]
  For two objects $a, b \in \underline{2}^{n}$, we set $b \leq a$ if $b_{i} \leq a_{i}$, for any $1 \leq i \leq n$, where $a = (a_{1},a_{2},\ldots,a_{n})$ and $b = (b_{1},b_{2},\ldots,b_{n})$.
  An $n$-\emph{diagram}, for $n \geq 1$, is a functor from the category $\underline{2}^{n}$ to the category of topological spaces.

  We can associate an $n$-diagram to a given $\langle n \rangle$-manifold $X$ by declaring, for every $a = (a_{1}, a_{2}, \ldots, a_{n}) \in \underline{2}^{n}$:
  \[
    X(a) =
    \begin{cases}
      X, & \text{ if } a = (1,1, \ldots, 1), \\
      \bigcap_{i: a_{i}=0} \partial_{i} X, & \text{ otherwise}.
    \end{cases}
  \]
  Moreover, for any $b \leq a$ in $\underline{2}^{n}$ the map $X(b) \to X(a)$ is the inclusion.
  We point out that $X(a)$ is an $\langle |a| \rangle$-manifold with the corresponding $|a|$-diagram obtained by restricting $X$ to
  the full subcategory of objects $b$ such that $b<a$ (recall that  $|a|=\sum a_i$).

  Let $M$ be a manifold with corners.
  Choose a Riemannian metric on $M$.
  We have the following generalization of a classical result.
  \begin{proposition}\label{prop:ident}
    There is an open tubular neighborhood $U$ of $\partial M$, homeomorphic to $M\times[0,1)$ and a subset $V\subset TM|_{\partial M}$ such that the $\exp$ map yields a diffeomorphism between $U$ and $V$.
  \end{proposition}
  \begin{proof}
    The proof is analogous to the proof of the collar neighborhood theorem, see, for instance, ~\cite[Section 4.6]{Hirsch}.
  \end{proof}

  We now recall the concept of a neat embedding, which roughly means an embedding with no pathological behavior near the boundary.
  Various similar notions are discussed in detail in~\cite[Section 3]{Joyce}.
  \begin{definition}\label{def:neat_embedding}
    Let $X$ and $Y$ be two $\langle n \rangle$-manifolds.
    A \emph{neat embedding} is an embedding $\iota \colon X \to Y$ such that:
    \begin{enumerate}
    \item $\iota$ is an $n$-map, i.e. $\iota^{-1}(\partial_{i}Y) = \partial_{i} X$, for any $1 \leq i \leq n$,
    \item the intersection of $X(a)$ and $Y(b)$ is perpendicular with respect to some Riemannian metric on $Y$, for $b<a$ in $\underline{2}^{n}$.
    \end{enumerate}
  \end{definition}

  The following result is a direct consequence of~\cite[Lemma 3.11]{LS_stable}.\mpar{(69)}
  \begin{theorem}[\cite{LS_stable}]
    Every compact $\langle n \rangle$-manifold admits a neat embedding
    \[\iota \colon X \hookrightarrow E_{n}^{\dc},\]
    for some $\dc \in \mathbb{N}^{n+1}$.
  \end{theorem}

  \subsection{Group actions on $\langle n\rangle$-manifolds}
  Let $\Diff_{n}(X)$ denote the group of diffeomorphisms of {the $\langle n \rangle$-manifold} $X$ that are also $n$-maps.
  If $G$ is a finite group, then a \emph{smooth action} of $G$ on $X$ is a homomorphism
  \[\gamma \colon G \to \Diff_{n}(X).\]
  An action of $G$ is said to be \emph{effective} if $\gamma$ is injective.
  Throughout this paper, we assume that group actions are effective.
  Moreover, we will often identify $g \in G$ with its image $\gamma(g) \in \Diff_{n}(X)$.

  \begin{definition}\label{def:eq_dimension}
    Let $V-W \in RO(G)$.
    We say that $X$ is of \emph{dimension} $V-W$, and denote it by $\dim X = V-W$, if for any interior point $x$ there exists an isomorphism of representations
    \[T_{x}X \oplus W|_{G_{x}} \cong V|_{G_{x}}.\]
  \end{definition}

  We have the following equivariant analog of Proposition~\ref{prop:ident}.
  \begin{proposition}\label{prop:Gident}
    Let $M$ be an $\langle n\rangle$-manifold with an action of a finite group $G$.
    Choose a $G$-invariant Riemannian metric on $M$.
    Then $\partial M$ admits a $G$-equivariant tubular neighborhood $U$ such that there exists a $G$-invariant subset $V\subset TM|_{\partial M}$
    such that the $\exp$-map takes $V$ diffeomorphically and $G$-equivariantly to $U$.
  \end{proposition}
  \begin{proof}
    The proof for standard manifolds with boundary is given in~\cite[Section 3]{Kanka}.
    The case of manifolds with corners is analogous.
  \end{proof}

  \begin{definition}
    Let $M$ be an $\langle n\rangle$-manifold acted upon by $G$. Let $V$ be an orthogonal representation of $G$. The manifold $M$ is said to be \textit{subordinate} to $V$ if for each $x \in M$ there exists an invariant neighborhood $\mathcal{U}_x$ of $x$, and an equivariant differentiable embedding of $U_x$ in $V^t \setminus \{0\}$ for some $t$.
    We write $\mathcal{G}(V)$ for the category whose objects are $G$-manifolds subordinate to $V$ and whose maps are continuous equivariant maps.
  \end{definition}

  Next result is an equivariant version of~\cite[Lemma 3.11]{LS_stable}.\mpar{(70)}
  It is needed in the proof of Proposition~\ref{prop:equivariant_neat}.
  \begin{lemma}\label{lem:equivversion}
    Let $M$ be a compact $\langle n\rangle$-manifold with a group action.
    Suppose $M$ is subordinate to a representation $V$.
    If $\iota_\partial\colon\partial M\to V^d$ is a $G$-equivariant embedding, then there exists $d'\ge d$ and a $G$-equivariant embedding $\iota\colon M\to V^{d'}$ such that $\iota|_{\partial M}=\iota_\partial$.
  \end{lemma}
  A key ingredient in the proof is the equivariant version of Whitney embedding theorem, due to Mostow and Palais.
  \begin{theorem}[Mostow-Palais Theorem \expandafter{\cite[Corollary 1.10]{wasserman_equivariant_1969}}]\label{thm_mostowpalais}
    If $M$ is in $\mathcal{G}(V)$ and the (real) dimension of $M$ is $n$, then any $G$-equivariant map
    \[f: M \longrightarrow V^t\]
    can be uniformly $C^k$-approximated by an equivariant immersion if $t \geq 2n$, and by an equivariant 1--1 immersion if $t \geq 2n+1$.
    Moreover, if $A$ is a closed subset of $M$ and $f\mid_A$ is a 1--1 immersion, then the approximation may be chosen in such a way that it agrees with $f$ on $A$.
  \end{theorem}
  \begin{proof}[Proof of Lemma~\ref{lem:equivversion}]
    Let $U$ be a $G$-equivariant collar neighborhood of $\partial M$.
    Using Proposition~\ref{prop:Gident} we identify $U$ with a subset $Z\subset T_{\partial M}M$ via a $G$-equivariant diffeomorphism $\psi\colon U\to Z$ that takes $\partial M$ to the zero section of $T_{\partial M}M$. 
    The $G$-bundle $T_{\partial M} M$ is a subbundle of a trivial bundle $V^{d_0}$ for some $d_0>0$.
    Then $\iota_\partial$ can be extended to $U$ via the composition
    \[Z\xrightarrow{\psi} T_{\partial M} M\hookrightarrow V^{d_0}\times \partial M\xrightarrow{id\times\iota_\partial} V^{d_0}\times V^{d}.\]
    Suppose $d_0+d>2\dim M$, otherwise increase $d_0$.
    Extend the embedding of $Z$ to a smooth $G$-equivariant map $\wt{\iota}\colon M\to V^{d_0}\times V^d$.
    Then, $\wt{\iota}$ can be perturbed to an equivariant embedding by the Mostow-Palais Theorem (Theorem~\ref{thm_mostowpalais}).
    By the second part of this theorem, we can keep $\wt{\iota}$ equal to $\iota$ in a neighborhood of $\partial M$.
  \end{proof}

  \section{Permutohedra}\label{sec:permutohedra}
  \subsection{The construction}\label{sub:permutstart}
  We refer to~\cite[Chapter 1]{toric_topology} for general properties of permutohedra.

  \begin{definition}\label{def:permuto}
    Choose a strictly increasing sequence $\cS=\{s_1,\ldots,s_r\}$ of positive integers.
    The \emph{permutohedron} $\Pi_\cS$ is the convex hull in $\R^r$ of the set of points
    \[(s_{\sigma(s_1)},\ldots,s_{\sigma(s_r)}),\]
    where $\sigma$ runs through all permutations of the set $\cS$.
  \end{definition}
  We write $\Pi_{r-1}$ for the permutohedron in the special case when $\cS=\{1,\ldots,r\}$.
  The subscript is $r-1$ and not $r$, because  $\dim\Pi_{r-1}=r-1$.
  \begin{example}\label{ex:basicpermuto}
    A permutohedron for $\cS=\{1,2,3\}$ is depicted in Figure~\ref{fig:permuto}.
  \end{example}
  \begin{figure}
    \begin{tikzpicture}

      \foreach \i in {1,2,3} { \draw (\i,1.05,1) -- (\i,0.95,1); \draw (1,\i,1.05) -- (1,\i,0.95); \draw  (1.05,1,\i)--(0.95,1,\i);}
      \draw (1,1,1) -- (1,1,3) -- (1,3,3) -- (1,3,1) -- (3,3,1) -- (3,1,1) -- (3,1,3) -- (3,3,3) -- (1,3,3);
      \draw (3,1,1,) -- (1,1,1) -- (1,3,1);
      \draw (3,3,3) -- (3,3,1);
      \draw (1,1,3) -- (3,1,3);

      \draw[fill=blue!30,opacity=0.3] (1,2,3)--(1,3,2)--(2,3,1)--(3,2,1)--(3,1,2)--(2,1,3) -- (1,2,3);
      \draw[fill=black] (1,2,3) circle (0.05) ++ (-0.5,0,0) node [scale=0.7] {$(1,2,3)$};
      \draw[fill=black] (1,3,2) circle (0.05) ++ (-0.2,0.2,0) node [scale=0.7] {$(1,3,2)$};
      \draw[fill=black] (2,1,3) circle (0.05) ++ (0,-0.2,0) node [scale=0.7] {$(2,1,3)$};
      \draw[fill=black] (2,3,1) circle (0.05) ++ (0,0.2,0) node [scale=0.7] {$(2,3,1)$};
      \draw[fill=black] (3,1,2) circle (0.05) ++ (0.5,0,0) node [scale=0.7] {$(3,1,2)$};
      \draw[fill=black] (3,2,1) circle (0.05) ++ (0.5,0,0) node [scale=0.7] {$(3,2,1)$};

    \end{tikzpicture}
    \caption{A permutohedron with $\cS=\{1,2,3\}$.}\label{fig:permuto}
  \end{figure}

  To describe the faces of a permutohedron we first set
  \begin{equation}\label{eq:taudef}
    \tau_{i}=\sum_{j=1}^{i} s_{j}.
  \end{equation}

  \begin{lemma}[see \expandafter{\cite[Theorem 1.5.7]{toric_topology}}]\label{lem:faces}
    The permutohedron $\Pi_\cS$ is given by the equation $\sum x_i=\tau_{r}$ and the set of inequalities
    \[\sum_{i\in\cP} x_i\ge \tau_{|\cP|},\]
    where $\cP$ runs through all non-empty proper subsets of $\{1,2,\ldots,r\}$.
  \end{lemma}
  From Lemma~\ref{lem:faces} we deduce the following fact~\cite[Section 2]{LLS_long}.
  \begin{lemma}\label{lem:codimk}
    For any proper subset $\cP$ of $\{1,\ldots,r\}$ the intersection of $\Pi_{\cS}$ with the hyperplane $\sum_{i\in\cP} x_i=\tau_{|\cP|}$ defines a facet of $\Pi_{\cS}$, which is diffeomorphic to $\Pi_{\cS_{1}}\times\Pi_{\cS_{2}}$, where $\cS_{1}=\{s_1,\ldots,s_{|\cP|}\}$ and $\cS_{2}=\cS\setminus\cS_{|\cP|}=\{s_{|\cP|+1},\ldots,s_r\}$.
  \end{lemma}
  \begin{example}[Example~\ref{ex:basicpermuto} continued]
    For $\cS=\{1,2,3\}$ and $\cP=\{1,2\}$, the facet is given by $x_1+x_2=3$.
    This is a product of a one-dimensional permutohedron $\Pi_1$ spanned by $(1,2)$ and $(2,1)$ in the $(x_1,x_2)$-coordinates, and a zero-dimensional permutohedron $\Pi_{\{3\}}$ given by $\{x_3=3\}\subset\R$.
    For $\cP=\{3\}$ we obtain the opposite facet of the hexagon.
    It is given by $\{x_3=1\}$, $\{x_1+x_2=5\}$.  See Figure~\ref{fig:permutoface}.
  \end{example}
  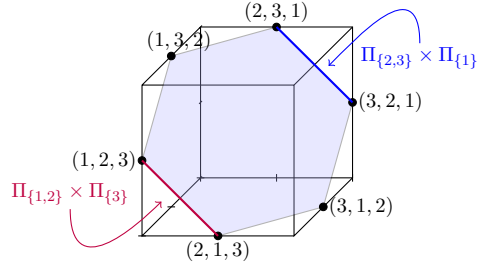
\begin{figure}
    \begin{tikzpicture}
      \foreach \i in {1,2,3} { \draw (\i,1.05,1) -- (\i,0.95,1); \draw (1,\i,1.05) -- (1,\i,0.95); \draw  (1.05,1,\i)--(0.95,1,\i);}
      \draw (1,1,1) -- (1,1,3) -- (1,3,3) -- (1,3,1) -- (3,3,1) -- (3,1,1) -- (3,1,3) -- (3,3,3) -- (1,3,3);
      \draw (3,1,1,) -- (1,1,1) -- (1,3,1);
      \draw (3,3,3) -- (3,3,1);
      \draw (1,1,3) -- (3,1,3);
      \draw[fill=blue!30,opacity=0.3] (1,2,3)--(1,3,2)--(2,3,1)--(3,2,1)--(3,1,2)--(2,1,3) -- (1,2,3);
      \draw[fill=black] (1,2,3) circle (0.05) ++ (-0.5,0,0) node [scale=0.7] {$(1,2,3)$};
      \draw[fill=black] (1,3,2) circle (0.05) ++ (0.10,0.2,0) node [scale=0.7] {$(1,3,2)$};
      \draw[fill=black] (2,1,3) circle (0.05) ++ (0,-0.2,0) node [scale=0.7] {$(2,1,3)$};
      \draw[fill=black] (2,3,1) circle (0.05) ++ (0,0.2,0) node [scale=0.7] {$(2,3,1)$};
      \draw[fill=black] (3,1,2) circle (0.05) ++ (0.5,0,0) node [scale=0.7] {$(3,1,2)$};
      \draw[fill=black] (3,2,1) circle (0.05) ++ (0.5,0,0) node [scale=0.7] {$(3,2,1)$};
      \draw[thick,purple] (1,2,3) -- (2,1,3);
      \draw[purple] (-1.1,0.4) node [scale=0.7] {$\Pi_{\{1,2\}}\times \Pi_{\{3\}}$};
      \draw[blue]  (3.5,2.2) node [scale=0.7] {$\Pi_{\{2,3\}}\times \Pi_{\{1\}}$};
      \draw[purple,->] (-1.1,0.2) .. controls ++(270:1) and ++(225:0.5) .. (0.1,0.3);
      \draw[blue,->] (3.5,2.4) .. controls ++ (90:1) and ++ (45:0.5) .. (2.3,2.1);
      \draw[thick,blue] (3,2,1) -- (2,3,1);
    \end{tikzpicture}
    \caption{Illustration of Better Example B.6}\label{fig:permutoface}
  \end{figure}
  \begin{proof}[Sketch of proof of Lemma~\ref{lem:codimk}]
    If $\cP=\{a_1,\ldots,a_{|\cP|}\}$ and $\{1,\ldots,r\}\setminus\cP=\{b_1,\ldots,b_{r-|\cP|}\}$, the diffeomorphism takes the element $(x_1,\ldots,x_r)$ to an element $(x_{a_1},\ldots,x_{a_{|\cP|}})\times(x_{b_1},\ldots,x_{b_{r-|\cP|}})\in\R^{|\cP|}\times\R^{r-|\cP|}$.
    It is routine to verify, using Lemma~\ref{lem:faces}, that the image is indeed equal to $\Pi_{\cS_{1}}\times\Pi_{\cS_{2}}$.
  \end{proof}
  \begin{definition}
    The facet corresponding to $\cP$ is denoted by $\Pi_{\cP,\cS\setminus\cP}$ (often simply $\Pi_{\cP}$) and it is called the \emph{facet associated with subset $\cP$}.
  \end{definition}
  In the following corollary we use the notation of the proof of Lemma~\ref{lem:codimk}.
  \begin{corollary}\label{lem:codimcorol}
    For a proper subset $\cP\subset\{1,\ldots,r\}$ the facet $\Pi_{\cP,\cS\setminus\cP}$ is contained in the subset of \(\Pi_{r-1}\) consisting of points such that
    \begin{align*}
      \sum_{i=1}^{|\cP|} x_{a_i} & \le s_{|\cP|},\\
      \sum_{i=1}^{r-|\cP|} x_{b_{i}} &\ge s_{|\cP|+1}.
    \end{align*}
  \end{corollary}

  From Lemma~\ref{lem:codimk} we can obtain an inductive description of codimension $k$ faces of $\Pi_{\cS}$.
  They correspond to partitions $\fp=(\cP_1,\ldots,\cP_{k+1})$ of $\{1,\ldots,r\}$ into $k+1$ non-empty subsets $\cP_1,\ldots,\cP_{k+1}$.
  Each such face, denoted by $\Pi_{\cP_1,\ldots,\cP_{k+1}}$ or, in short $\Pi_{\mathfrak{p}}$, is a product of $(k+1)$ permutohedra $\Pi_{\cS_1}\times\Pi_{\cS_2}\times\cdots\times\Pi_{\cS_{k+1}}$, where
  \[\cS_j=\{s_t\colon |\cP_1|+\cdots+|\cP_{j-1}|<t\leq |\cP_{1}|+\cdots+|\cP_{j}|\}.\]
  In particular if we set
  \[\partial_{i} \Pi_{\mathcal{S}} = \bigsqcup_{\stackrel{\mathcal{P} \subset \mathcal{S}}{|\mathcal{P}|=i}} \Pi_{(\mathcal{P},\mathcal{S} \setminus \mathcal{P})},\]
  for $1 \leq i \leq r-1$, $\Pi_{\cS}$ becomes an $(r-1)$-dimensional $\langle r-1\rangle$-manifold.

  \begin{remark}
    For consistency of the notation, we observe that the interior of $\Pi_{\cS}$ corresponds to the trivial partition $\fp$ of $\{1,\ldots,r\}$ into a single subset.
  \end{remark}
  The following notion is intuitive.
  \begin{definition}\label{def:refinereduce}
    Let $\fp=(\cP_1,\ldots,\cP_{k+1})$ be a partition of $\{1,\ldots,r\}$.
    A \emph{refinement} of $\fp$ is a partition $(\wt{\cP}_1,\ldots,\wt{\cP}_{\wt{k}+1})$ such that there exist indices $0=a_0<a_1<a_2<\cdots<a_k=\wt{k}+1$ with the property that $\wt{\cP}_{a_{i-1}+1}\cup\wt{\cP}_{a_{i-1}+2}\cup\ldots\cup\wt{\cP}_{a_i-1}=\cP_i$ for $i=1,\ldots,k$.
  \end{definition}
  It is clear that if $\fp$ and $\fp'$ are two partitions of $\{1,\dots,r\}$, then $\Pi_{\fp'}\subset\Pi_{\fp}$ if and only if $\fp'$ is a refinement of $\fp$.
  \subsection{Intersecting a permutohedron with a hyperplane}
  We describe the intersection of a permutohedron with hyperplanes given by sets of equations $\{x_{a_i}=x_{b_i}\}$.
  It turns out that this intersection is a lower-dimensional permutohedron.
  The key statement in this section is Proposition~\ref{prop:permutinter}, which identifies the intersection of a permutohedron $\Pi_{\mathcal{S}}\cap L$ with $\Pi_{s_1,\ldots,s_{r-1}}$.
  The identification of Proposition~\ref{prop:permutinter} is such that the combinatorial structure of the boundary is preserved.
  In order to spell this control over the combinatorial structure, we need to introduce a simple notion.
  \begin{definition}
    Let $\fp=(\cP_1,\dots,\cP_s)$ be a partition of $\{1,\dots,r\}$. Let $1\le b\le r$. Assume that no $\cP_i$ is equal to the singleton $\{b\}$. 
    A \emph{reduction} of $\fp$ with respect of $b$ is
    a partition $\fp^{b}=(\cP_1^b,\dots,\cP_s^b)$ of $\{1,\dots,r-1\}$, where
    \begin{itemize}
    \item if $x\in\cP_i$ and $x<b$, then $x\in\cP^b_i$;
    \item if $x\in\cP_i$ and $x>b$, then $x-1\in\cP^b_i$.
    \end{itemize}
    If $B=\{b_1,\dots,b_\ell\}$ is a finite subset of $\{1,\dots,r\}$ and $\fp$ is a partition such that no $\cP_i$ is a subset of $B$, the \emph{reduction} of $\fp$ with respect to $B$ is a partition $\fp^B$ obtained as a subsequent reduction of $\fp$ by the elements $b_i$, starting from the largest element, then taking the second largest and so on.
  \end{definition}

  The following result can be deduced from the cubical decomposition of a permutohedron, see~\cite[Section 3.4]{LLS_long}.
  We give a self-contained proof of that result. 
  \begin{proposition}\label{prop:permutinter}
    Let $L$ be a hyperplane in $\R^r$ given by $\{x_a=x_b\}$ for some $a\neq b$.
    Let $\Pi_{\cS}$ be a permutohedron in $\R^r$ for some (strictly) increasing sequence $\cS=(s_1,\ldots,s_r)$.
    Consider
    \begin{equation}\label{eq:pih}
      \Pi_L=\Pi_{\cS}\cap L.
    \end{equation}
    Then there exists a diffeomorphism $\psi$ between $\Pi_{L}$ and $\Pi_{(s_1,\ldots,s_{r-1})}$.
    Moreover, if $\fp=(\cP_1,\dots,\cP_{k+1})$ is a partition of $\{1,\ldots,r\}$ then
    \begin{itemize}
    \item if $a,b$ do not belong to the same subset of the partition, then $\Pi_{\fp}\cap L$ is empty;
    \item if $a,b$ belong to the same subset of the partition, then $\Pi_{\fp}\cap L$ is mapped by $\psi$ diffeomorphically to $\Pi_{\fp^b}$.
    \end{itemize}
  \end{proposition}
  \begin{example}\label{ex:notsouseless}
    Let $r=3$  and $L=\{x_1=x_3\}$.
    The intersection of $L$ with $\Pi_{\cS}$ is an interval whose endpoints are $(\frac{s_1+s_2}{2},s_3,\frac{s_1+s_2}{2})$ and $(\frac{s_2+s_3}{2},s_1,\frac{s_2+s_3}{2})$.
    For $\cS=\{1,2,3\}$ this yields the segment connecting $(1.5,3,1.5)$ with $(2.5,1,2.5)$.
    See Figure~\ref{fig:permutointer}.
  \end{example}
  \begin{figure}
    \begin{tikzpicture}
      \draw[->] (0,0,0)--(4,0,0) node[at end,below,scale=0.7] {$x$};
      \foreach \i in {1,2,3} { \draw (\i,0.1,0) -- (\i,-0.1,0); \draw (0,\i,0.1) -- (0,\i,-0.1); \draw  (0.1,0,\i)--(-0.1,0,\i);}
      \draw[->] (0,0,0)--(0,4,0) node[at end,left,scale=0.7] {$y$};
      \draw[->] (0,0,0)--(0,0,4) node[at end,below,scale=0.7] {$z$};
      \draw[fill=blue!30,opacity=0.3] (1,2,3)--(1,3,2)--(2,3,1)--(3,2,1)--(3,1,2)--(2,1,3) -- (1,2,3);
      \draw[fill=black] (1,2,3) circle (0.05) ++ (-0.5,0,0) node [scale=0.7] {$(1,2,3)$};
      \draw[fill=black] (1,3,2) circle (0.05) ++ (0.05,0.2,0) node [scale=0.7] {$(1,3,2)$};
      \draw[fill=black] (2,1,3) circle (0.05) ++ (0,-0.2,0) node [scale=0.7] {$(2,1,3)$};
      \draw[fill=black] (2,3,1) circle (0.05) ++ (0,0.2,0) node [scale=0.7] {$(2,3,1)$};
      \draw[fill=black] (3,1,2) circle (0.05) ++ (0.5,0,0) node [scale=0.7] {$(3,1,2)$};
      \draw[fill=black] (3,2,1) circle (0.05) ++ (0.5,0,0) node [scale=0.7] {$(3,2,1)$};
      \draw[fill=red!30,opacity=0.1] (3,0,3) -- (3,3,3) -- (-3,3,-3) -- (-3,0,-3) -- (3,0,3);
      \draw[thick,green!70!black] (1.5,3,1.5) -- (2.5,1,2.5);
    \end{tikzpicture}
    \caption{Illustration of Proposition~\ref{prop:permutinter}.
      Observe that the plane $x_1=x_3$ intersects only two faces of the boundary: $\Pi_{\{1,3\}}\times\Pi_{\{2\}}$ and $\Pi_{\{2\}}\times\Pi_{\{1,3\}}$ in accordance with Lemma~\ref{lem:isempty}. }\label{fig:permutointer}
  \end{figure}
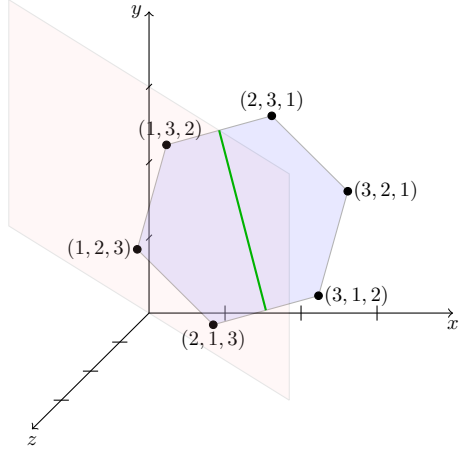
  \begin{proof}[Proof of Proposition~\ref{prop:permutinter}]

    The following lemma takes care of the first item in the statement of Proposition~\ref{prop:permutinter} and is needed for the proof of the second one.
    \begin{lemma}\label{lem:isempty}
      Suppose $\fp=(\cP_1,\ldots,\cP_{k+1})$ is a partition of $\{1,\ldots,r\}$.
      Then $\Pi_{\mathfrak{p}} \cap L$ is empty, unless there exists an index $i$ such that $a,b\in\cP_i$.
    \end{lemma}
    \begin{proof}
      We argue by contradiction.
      Suppose, $a\in\cP_i$, $b\in\cP_j$ with $i\neq j$.
      We have an inclusion $\Pi_{\mathfrak{p}}\subset \Pi_{\cP_i,\cS\setminus\cP_i}$.

      By Corollary~\ref{lem:codimcorol}, if $x=(x_1,\ldots,x_r)\in\Pi_{\cP_i,\cS\setminus\cP_i}$, then $x_a\le s_{|\cP_i|}$ and $x_b\ge s_{|\cP_i|+1}$.
      Since $s_k$ is a strictly increasing sequence, we conclude that $x_a<x_b$, hence $\Pi_{\cP_i,\cS\setminus\cP_i}\cap L=\emptyset$.
      In particular $\Pi_{\mathfrak{p}}\cap L=\emptyset$.
    \end{proof}
    \noindent\emph{Continuation of the proof of Proposition~\ref{prop:permutinter}.}
    We will construct the isomorphism by induction, starting with the lowest dimension faces, i.e. vertices.
    
    For a partition $\fp$ of length $r$, $\Pi_{\fp}\cap L$ is empty by Lemma~\ref{lem:isempty}.
    Suppose $\fp$ is a length $r-1$ partition. Unless $a$ and $b$ belong to the same subset of the partition, $\Pi_{\fp}\cap L=\emptyset$.
    Consider the case, when $\{a,b\}$ subset $\cP_j$ for some $j=1,\dots,r-1$. Obviously $\cP_j=\{a,b\}$ and all other subsets $\cP_i$,
    consists of single elements, $\cP_i=\{p_i\}$ for some $p_i$ different than $a$ and $b$.
    By Corollary~\ref{lem:codimcorol}, $\Pi_{\fp}$ is given by
    \begin{align*}
      x_{p_i}&=i, \textrm{ if $i<j$} &
                                       x_{p_i}&=i+1, \textrm{ if $i>j$}.\\
      x_a+x_b&=j+(j+1),&
                         x_a,x_b&\in[j,j+1]
    \end{align*}
    The intersection of $\Pi_{\fp}$ with $x_a=x_b$ is clearly given by
    \[
      x_{p_i}=i,\textrm{ if $i<j$},\ \ 
      x_{p_i}=i+1,\textrm{ if $i>j$},\ \ 
      x_a=x_b=j+\frac12.
    \]
    Now $\fp^b$ is a partition such that all its subsets have one element, namely $\fp^b=(\cP^b_1,\dots,\cP^b_{r-1})$ with:
    $\cP^b_i=\{p_i\}$ for $i\neq j$ and $p_i<b$,
    $\cP^b_i=\{p_i-1\}$ for $i\neq j$ and $p_i>b$, and finally
    $\cP^b_j=\{a\}$.
    Therefore, $\Pi_{\fp^b}$ is a single point, given by $\{x_{p_i}=s_i\}$ if $p_i<b$, $\{x_{p_i-1}=s_i\}$ if $p_i>b$, and $x_a=s_j$.
    The map $\psi$ is a map between to points $\psi\colon \Pi_{\fp}\cap L\xrightarrow{\cong}\Pi_{\fp^b}$

    We pass now to the induction step.
    Suppose $\psi$ has been constructed for all boundary components of codimension at least $k+1$ and consider a partition $\fp=\cP_1\cup\ldots\cup\cP_{k+1}$.
    Then, $\Pi_{\fp}$ is a convex polytope whose boundary is a union of polytopes $\Pi_{\wt{\fp}}$ for all refinements $\wt{\fp}$ of $\fp$.
    Therefore, the intersection $\Pi_{\fp}\cap\{x_a=x_b\}$ is a convex polytope whose boundary is the union of $\Pi_{\wt{\fp}}\cap \{x_a=x_b\}$.

    By Lemma~\ref{lem:isempty}, for each such refinement $\wt{\fp}$, either $\Pi_{\wt{\fp}}\cap L$ is empty, or $\Pi_{\wt{\fp}}\cap L=\Pi_{\wt{\fp}^b}$.
    By the induction assumption, in the latter case the restriction of $\psi$
    \[\psi|_{\Pi_{\wt{\fp}}\cap L}\colon\Pi_{\wt{\fp}}\cap L\to\Pi_{\wt{\fp}^b}\]
    has already been constructed.
    This means that $\psi$ restricted to the boundary of $\Pi_{\fp}\cap L$ is an isomorphism onto the boundary of $\Pi_{\fp^b}$.
    Therefore, we can extend $\psi$ to a diffeomorphism
    \[\psi\colon\Pi_{\fp} \cap L \to \Pi_{\fp^b}.\]
  \end{proof}
  \begin{remark}
    The isomorphism $\psi$ constructed in the proof of Proposition~\ref{prop:permutinter} does not need to be affine.
    Two convex polytopes with the same combinatorics are not necessarily affine equivalent.
  \end{remark}
  Now we state a result on intersections of $\Pi_{\cS}$ with more than one hyperplane.
  \begin{theorem}\label{thm:generalpermut}
    Suppose $\Pi_{\cS}\subset \R^r$ is a permutohedron and $H$ is a linear subspace of $\R^r$ given by equations $\{x_{a_{11}}=\cdots=x_{a_{1k_1}}\}$, $\{x_{a_{21}}=\cdots=x_{a_{2k_2}}\}$, \dots, $\{x_{a_{w1}}=\cdots=x_{a_{wk_w}}\}$ with all the $a_{ij}$ pairwise distinct.
    Then, the intersection $\Pi_{\cS}\cap H$ is diffeomorphic to $\Pi'=\Pi_{(s_1,\ldots,s_{r-K})}$, where $K=\sum(k_i-1)=\codim H$.
    The diffeomorphism takes the face of $\Pi_{\cS}$ corresponding to a partition $\fp$ to the face corresponding to the reduction $\fp^B$, where
    \[B=\{a_{12},\dots,a_{1k_1},a_{22},\dots,a_{2k_2},\dots,a_{w2},\dots,a_{wk_w}\}.\]
  \end{theorem}
  \begin{proof}
    Apply inductively Proposition~\ref{prop:permutinter}.
  \end{proof}

  \subsection{Permutohedra, group actions and posets}\label{sub:combiposets}
  Choose a permutation $\sigma \in \Sn$ of order $m$.
  We can define an action of a cyclic group $\Z_{m}$ on $\R^{n}$ in terms of $\sigma$:
  \begin{equation}\label{eq:sigma-action}
    \bar{\sigma}(x_1,x_2,\ldots,x_n) = (x_{\sigma(1)}, x_{\sigma(2)}, \ldots, x_{\sigma(n)}).
  \end{equation}
  Since $\Pi_{n-1}$ is an invariant subset, $\Z_{m}$ acts on $\Pi_{n-1}$.

  We are now going to show that action $\bar{\sigma}$ endows $\Pi_{n-1}$ with a structure of a $\Z_m$-manifold, and we will compute its $\Z_m$-dimension.
  Let \(V_{\sigma}\) denote the \(\Z_{m}\)-representation induced by the action of \(\sigma\) on \(\R^{n}\).
  If \(\sigma\) is a product of \(p\) disjoint cycles of lengths \(n_{1},n_{2},\ldots,n_{p}\), then there is an isomorphism of representations
  \[V_{\sigma} \cong \bigoplus_{i=1}^{p} \R[\Z_{n_{i}}],\]
  where \(\R[\Z_{n_{i}}]\) denotes the real group algebra of \(\Z_{n_{i}}\), for \(i=1,2,\ldots,p\), and \(\Z_{m}\) acts on \(\R[\Z_{n_{i}}]\) via the projection \(\Z_{m} \to \Z_{n_{i}}\).
  Permutohedron $\Pi_{n-1}$ is contained in the affine hyperplane
  \[L = \left\{(x_1,x_2,\ldots,x_n) \colon \sum_{i=1}^{n} x_i = \frac{n(n+1)}{2}\right\},\]
  which is invariant under $\Z_m$.
  Orthogonal projection of $L$ onto the hyperplane
  \[L_{0} = \left\{(x_1,x_2,\ldots,x_n) \colon \sum_{i=1}^n x_i = 0\right\} \cong V_{\sigma} - \R\]
  shows that the action of $\bar{\sigma}$ restricted to $L$ yields a representation isomorphic to $V_{\sigma}-\R$.
  \begin{remark}
    Recall that $V_{\sigma} - \R$ denotes the orthogonal complement of a one-dimensional trivial representation $\R$ inside $V_{\sigma}$.
  \end{remark}

  \begin{lemma}\label{lemma:equiv-dim-permutohedron}
    Let $\bar{\sigma}$ be as above.
    \begin{enumerate}
    \item At every interior point $x$ of $\Pi_{n-1}$ the tangent representation $T_x \Pi_{n-1}$ is isomorphic to $(V_{\sigma} - \R)|_{(\Z_m)_x}$, where $(\Z_{m})_{x}$ denotes the isotropy group at $x \in \Pi_{n-1}$.
    \item If $\cP \subset \{1,2,\ldots,n\}$ and $\Pi_{\cP}$ denotes the face of $\Pi_{n-1}$ defined in Lemma~\ref{lem:codimk}, then $\bar{\sigma}(\Pi_{\cP}) = \Pi_{\sigma(\cP)}$.
    \item $\bar{\sigma}$ restricted to $\Pi_{n-1}$ is an $n$-map, i.e., $\bar{\sigma}(\partial_i \Pi_{n-1}) = \partial_i \Pi_{n-1}$.
    \item Suppose that $a_1 < a_2 < \cdots < a_k$ are elements of $\cP$ and $b_1 < b_2 < \cdots < b_{n-k}$ are the elements of its complement.
      Suppose that $\sigma(\cP) = \cP$ and define maps
      \begin{align*}
        \tau_1 \colon \cP \to \{1, 2, \ldots, k\}, \quad \tau_1(a_i) &= i, \quad 1 \leq i \leq k, \\
        \tau_2 \colon \{1,2,\ldots,n\} \setminus \cP \to \{1,2,\ldots,n-k\}, \quad \tau_2(b_i) &= i, \quad 1 \leq i \leq n-k.
      \end{align*}
      If we identify $\Pi_{\cP}$ with $\Pi_{|\cP|-1} \times \Pi_{n-|\cP|-1}$ as in Lemma~\ref{lem:codimk}, then
      \[\bar{\sigma}|_{\Pi_{\cP}} = (\tau_1 \circ \sigma|_{\cP} \circ \tau_1^{-1}) \times (\tau_2 \circ \sigma|_{\{1,2,\ldots,n\} \setminus \cP} \circ \tau_2^{-1}).\]
    \end{enumerate}
  \end{lemma}
  \begin{proof}
    The first statement follows readily because $\Pi_{n-1} \subset L \cong V_{\sigma} - \R$ has non-empty interior.

    In order to prove the second statement recall that $\Pi_{(\cP, \{1,\ldots,n\}\setminus \cP)} = \Pi_{n-1} \cap \partial L_{\cP}$, where
    \[L_{\cP} = \left\{(x_{1},\ldots,x_{n}) \colon \sum_{i \in |\cP|}x_i \geq \frac{|\cP|(|\cP|+1)}{2}\right\}.\]
    Since $\bar{\sigma}(L_{\cP}) = L_{\sigma(\cP)}$, the statement (2) follows.

    Next, recall that by the definition of $\partial_i$
    \[\partial_{i} \Pi_{n-1} = \bigcup_{\stackrel{\cP \subset \{1,2,\ldots,n\}}{|\cP|=i}} \Pi_{(\cP,\{1,\ldots,n\}\setminus \cP)}.\]
    The fact that $|\sigma(\cP)| = |\cP|$ completes the proof of the third statement.

    In order to prove the last assertion notice that if $\cP$ is invariant under $\sigma$, then its complement is also invariant.
  \end{proof}

  \begin{proposition}\label{prop:fixed_permut}
    Fix a permutation \(\sigma \in \Sn\) of order \(m\).
    Assume that \(\sigma\) is a product of \(p\) disjoint cycles of lenghts \(n_{1},n_{2},\ldots,n_{p}\) and let \(N = \sum_{i=1}^{p} (n_{i}-1)\).
    There is a smooth diffeomorphism $\psi$ taking $\Pi_{n-1}^{\Z_m}$ to a permutohedron $\Pi_{n-N-1}$.
    For a partition $\fp$ of $\{1,\dots,n\}$, if $\Pi_{\fp}^{\Z_m}$ is not empty, then $\psi(\Pi_{\fp})$ is the face $\Pi_{\fp^B}$ of $\Pi_{n-m}$, where $B$ is the set obtained from $\{1,\dots,n\}$ by removing the smallest element of each cycle of $\sigma$.
  \end{proposition}
  \begin{proof}
    As $\bar{\sigma}$ acts by permuting coordinates, the fixed point set of the action is a hyperplane $L$ defined as $x_i=x_j$, whenever $i$ and $j$ belongs to the same orbit of the action of $\sigma$ on $\{1,2,\ldots,n\}$.
    The statement follows from  Theorem~\ref{thm:generalpermut}.
  \end{proof}

\emph{On behalf of all authors, the corresponding author states that there is no conflict of interest.}

  \bibliographystyle{amsalpha}
  \bibliography{homotopy}

\end{document}